\documentclass[11pt]{article}
\usepackage[leqno]{amsmath}
\usepackage{amssymb}
\usepackage{amsfonts} 
\usepackage{amstext} 
\usepackage{amsthm} 
\usepackage{euscript} 
\usepackage{eufrak} 
\usepackage{array} 
\usepackage{multirow} 
\usepackage{graphicx}

\makeatletter
\renewcommand{\theenumi}{\roman{enumi}}

\renewcommand{\p@enumii}{\theenumi--}
\makeatother

\newcommand{\Z}{\ensuremath{\mathbb{Z}}}
\newcommand{\N}{\ensuremath{\mathbb{N}}}

\newcommand{\R}{\ensuremath{\mathbb{R}}} 
\newcommand{\C}{\ensuremath{\mathbb{C}}}

\newcommand{\dP}{\ensuremath{\mathbb{P}}}

\newcommand{\cC}{\ensuremath{\mathcal{C}}}
\newcommand{\cD}{\ensuremath{\mathcal{D}}} 
\newcommand{\cE}{\ensuremath{\mathcal{E}}}
\newcommand{\cF}{\ensuremath{\mathcal{F}}} 
\newcommand{\cG}{\ensuremath{\mathcal{G}}} 
\newcommand{\cH}{\ensuremath{\mathcal{H}}}

\newcommand{\cL}{\ensuremath{\mathcal{L}}}
\newcommand{\cM}{\ensuremath{\mathcal{M}}} 
\newcommand{\cN}{\ensuremath{\mathcal{N}}}

\newcommand{\cR}{\ensuremath{\mathcal{R}}} 

\newcommand{\cT}{\ensuremath{\mathcal{T}}} 
 
\newcommand{\cV}{\ensuremath{\mathcal{V}}}

\newcommand{\cZ}{\ensuremath{\mathcal{Z}}}

\newcommand{\eg}{\ensuremath{\textsf{g}}}

\newcommand{\fD}{\ensuremath{\EuFrak{D}}}

\newcommand{\ff}{\ensuremath{\EuFrak{f}}}

\newcommand{\pa}{\ensuremath{\partial}} 

\newcommand{\h}{\ensuremath{\theta}} 
\newcommand{\vh}{\ensuremath{\vartheta}}

\newcommand{\br}{\ensuremath{\bar{r}}} 
\newcommand{\bx}{\ensuremath{\bar{x}}} 
\newcommand{\by}{\ensuremath{\bar{y}}} 
\newcommand{\bth}{\ensuremath{\bar{\theta}}} 
\newcommand{\bD}{\ensuremath{\bar{D}}} 
\newcommand{\bJ}{\ensuremath{\bar{J}}} 
\newcommand{\bX}{\ensuremath{\bar{X}}}

\newcommand{\g}{\ensuremath{\gamma}} 
 
\newcommand{\e}{\ensuremath{\epsilon}}

\newcommand{\s}{\ensuremath{\sigma}}

\begin{document}

\begin{center} 

{\Large \bf  Surgery and Invariants of Lagrangian Surfaces} \\ 

\vspace{.3in} 
Mei-Lin Yau \footnote{Research 
Supported in part by National Science Council grant  99-2115-M-008-005-MY2


} 
\\  

\end{center} 
%
%
%
%
%
%
%
%
%
%
%
%
%
%
%
\vspace{.2in} 
\begin{abstract} 

We considered a surgery, called {\em $la$-disk surgery}, that 
can be applied to a Lagrangian surface $L$ at the presence of 
a Lagrangian attaching disk $D$, to obtain a new 
Lagrangian surface $L':=\eta_D(L)$ which is always smoothly isotopic to 
$L$. We showed that this type of surgery includes all even 
generalized Dehn twists as constructed by Seidel. We also constructed 
a new symplectic invariant, called $y$-index, for orientable 
closed Lagrangian surfaces immersed in a 
parallelizable symplectic 4-manifold $W$. With $y$-index we proved 
that $L$ and $\eta_D(L)$ are not Hamiltonian isotopic under 
this setup. We also obtained new examples of 
smoothly isotopic nullhomologous Lagrangian 
tori which are  not Hamiltonian isotopic pairwise. 

\end{abstract} 
\newtheorem{prop}{Proposition}[subsection]
\newtheorem{theo}[prop]{Theorem}
\newtheorem{cond}[prop]{Condition}
\newtheorem{defn}[prop]{Definition}
\newtheorem{exam}[prop]{Example}
\newtheorem{lem}[prop]{Lemma}
\newtheorem{cor}[prop]{Corollary}
\newtheorem{ques}[prop]{Question}
\newtheorem{conj}[prop]{Conjecture}

\newtheorem{rem}[prop]{Remark}
\newtheorem{notn}[prop]{Notation}
\newtheorem{fact}[prop]{Fact}

\tableofcontents

\section{Introduction} 


\subsection{Isotopy of Lagrangian surfaces} 

A Lagrangian submanifold in a symplectic manifold $(W,\omega)$ 
is a submanifold $L\overset{\iota}{\hookrightarrow}W$ such that 
$\iota^*\omega=0$ and $\dim L=\frac{1}{2}\dim W$. 
In symplectic topology 
Lagrangian submanifolds distinguish themselves from other 
types of submanifolds in that, though the symplectic form $\omega$ 
vanishes on them, their mere presence uniquely 
determine the nearby symplectic structure, as 
stated in Weinstein's Lagrangian neighborhood theorem \cite{W} (see also \cite{MS}): 

\begin{theo}[{\bf Lagrangian neighborhood theorem}]
Let $L$ be an embedded compact Lagrangian submanifold of a 
symplectic manifold $(W,\omega)$. Then there exist  tubular 
neighborhoods $U\subset W$ of $L$, $V\subset T^*L$ of the 
zero section $L_0=L\subset T^*L$, and a diffeomorphism $\phi:V\to U$ such that $\phi(L_0)=L$ and $\phi^*\omega=\omega_{can}$, where $\omega_{can}$ is the canonical symplectic structure of the cotangent bundle $T^*L$. 
\end{theo} 

In symplectic category one would like to classify Lagrangian 
submanifolds up to symplectic isotopies or even Hamiltonian 
isotopies. 
Two Lagrangian submanifolds $L,L'$ 
are {\em symplectically isotopic} if there exist a family 
of diffeomorphisms $\phi_t$, $t\in[0,1]$, $\phi_0=id$, $\phi_t^*\omega=\omega$ for all $t\in[0,1]$, such that $\phi_1(L)=L'$. The maps $\phi_t$, being symplectomorphisms, are 
generated by symplectic vector fields $X_t$ defined by the 
condition $d(\iota(X_t)\omega)=0$. 
If moreover, $\iota(X_t)\omega=-dH_t$ is exact, then $L,L'$ are {\em Hamiltonian 
isotopic}. When $H^1(W,\R)=0$, symplectically isotopic Lagrangian surfaces in $W$ are automatically Hamiltonian isotopic. 
Clearly symplectically isotopic Lagrangian 
submanifolds are {\em smoothly isotopic} 
if we drop 
the condition of $\phi_t$ being $\omega$-preserving. 
Note that an alternative definition for  $L$ and $L'$ being smoothly isotopic is if $L=L_0$ and $L'=L_1$ can be 
included in a smooth $1$-parameter family of submanifolds 
$L_t\subset W$, $t\in[0,1]$, and all $L_t$ diffeomorphic to $L$. The family $L_t$ is called a {\em smooth isotopy}  
between $L$ and $L'$. We mention here that there is a 
middle ground in between called {\em Lagrangian isotopy} 
where we request that all $L_t$ are Lagrangian too. 
 
 Smoothly isotopic Lagrangian surfaces need not be 
 symplectically isotopic. It is our goal here to 
better understand the subtle difference between symplectic 
isotopy and smooth isotopy and we will focus on the case 
when the Lagrangian submanifold is a 2-dimensional surface, 
i.e, a {\em Lagrangian knot} as in \cite{EP3}.

For earlier results concerning smoothly but not symplectically isotopic Lagrangian surfaces, please see \cite{C, Se, H1, H2, AF, Y5, Kna}. Here we consider a different approach. We ask 
the following two questions: 

\begin{ques} \label{Q1} 
Find a procedure that can be applied to a given Lagrangian surface $L$ (as general as possible) to yield a new one $L'$ 
which is always smoothly isotopic to $L$ but potentially 
not symplectically so. 
\end{ques} 

\begin{ques} \label{Q2} 
Construct a symplectic isotopy invariant for Lagrangian surfaces 
and use it to show that $L,L'$ as in Question \ref{Q1} are 
not symplectically isotopic. 
\end{ques}

\subsection{Main constructions and results} 

To answer Question \ref{Q1} we consider a Lagrangian surface  
$L$ which admits an {\em Lagrangian attaching  disk} 
($la$-disk) $D$. We 
classify $D$ into three types: {\em parabolic}, {\em hyperbolic}, and 
{\em elliptic}, according to the homology/homotopy type of the 
boundary $C:=\pa D\subset L$ in $L$. Each type of a $la$-disk 
$D$ also comes with two flavors which we call the {\em polarity}  (see Definition \ref{polarity}) of $D$. With $D$ 
we construct the {\em $la$-disk surgery} $\eta_D$ 
which can be applied to  $L$ to 
get a new Lagrangian surface $L':=\eta_D(L)$.  Roughly speaking this 
surgery cut out a collar neighborhood of $C$ in $L$ and glue in a different Lagrangian annulus along the boundary. We remark 
here that our surgery here is closely related to the {\em Lagrangian surgery}  as defined in \cite{P2}.  In \cite{P2} a  2-dimensional Lagrangian surgery is to remove a positive self-intersection point of a Lagrangian surface. In contrast our 
$la$-disk surgery is about the transition between the two 
Lagrangian surfaces $L,L'$ resulting from different ways of 
de-singularizing a Lagrangian surface at a positive self-intersection point.

 We obtained the following results among other things: 
\begin{itemize} 
\item The resulting Lagrangian surface $L':=\eta_D(L)$ is 
smoothly isotopic to $L$ ({\bf Proposition} \ref{smooth-iso}). 

Dually, $L'$ has a $la$-disk $D'$ of the 
same type as $D$ but with different relative polarity, and $\eta_{D'}(L')=L$. Note the actual surface $L':=\eta_DL$ depends on 
the size of the surgery, nevertheless its Lagrangian isotopy 
class is unique.

\item If $L$ is monotone then under suitable conditions $L'$ is 
also monotone ({\bf Proposition} \ref{L'monotone}). 

As examples, we have (i) if $L\subset \R^4$ is a {\em Chekanov 
torus} as defined in \cite{C}, then $L'$ is a monotone Clifford torus; (ii) if $L\subset T^*S^2$ is a monotone Clifford torus, 
and $D$ is a {\em unstable} (see Definition \ref{stable}) parabolic $la$-disk of $L$, then (up to scaling) $\eta_D(L)$ is the non-displaceable torus constructed in \cite{AF}.  

\item If $L$ is a null-homologous Lagrangian torus, then the 
maximal number of mutually disjoint $la$-disks of $L$ imposes 
some restriction on the homology of $W$ ({\bf Proposition} \ref{nulltorus}, {\bf Corollary} \ref{atmost1}). 

\item If $D$ is {\em elliptic} then $\eta_D$ is equal to  the {\em square of a positive or negative generalized Dehn twist} 
as considered in \cite{Se}, 
where the sign is determined by the polarity of $D$ ({\bf Proposition}  
\ref{dehn}). 

\end{itemize} 

The $la$-disk surgery provides a general way of potentially 
changing the symplectic isotopy type of a Lagrangian surface 
without affecting its smooth isotopy type. 
On the other hand, $la$-disks seems intimately related to the 
build-up of a symplectic $4$-manifold around a Lagrangian 
surface. For example, in the integrable system considered in 
Section \ref{integrable}, a Chekanov torus $L$ lives in the 
boundary of the Stein domain diffeomorphic to $S^1\times B^3$ 
the product of $S^1$ with a $3$-ball, $\R^4$ can be obtained 
by attaching to $S^1\times B^3$  a Lagrangian $2$-handle such 
that the core disk of the $2$-handle is a stable parabolic 
$la$-disk of $L$ in $\R^4$. Similarly, $T^*S^2$ is obtained by 
attaching along the boundary $S^3=\pa B^4$ of a $4$-ball 
a Lagrangian $2$-handle whose core disk is again a 
stable parabolic $la$-disk in $T^*S^2$ of a monotone Clifford torus $L'\subset S^3$.

For Question \ref{Q2} we were able to construct a new symplectic isotopy invariant, called {\em $y$-index}, 
for {\em orientable} compact Lagrangian 
surfaces {\em immersed} in $W$ provided that $W$ is {\em parallelizable}. 

First of all, 
if $W$ is parallelizable then we can fix a {\em $\omega$-compatible unitary framing}  
$\ff:=(J,u,v)$, where $J$ is an $\omega$-compatible almost 
complex structure over $W$, $(u,v)$ is a $J$-complex unitary 
basis of $TW$. 
 The framing $\ff$ allows us to define the 
{\em projected Lagrangian Gauss map} (PLG-map)   
\[ 
g'_L:L\to \dP(K')\cong S^2 
\] 
for any oriented immersed Lagrangian surface $L$ in $W$. 
Here $\dP(K')$ is an $S^2$-family of oriented Lagrangian planes 
which are $K'$-complex for some orthogonal complex structure 
$K'$. If $L$ is also closed then we define the $\mu_2$-index of 
$L$ relative to $\ff$ to be the degree of the map $g'_L$: 
\[ 
\mu_2(L;\ff):=\deg(g'_L)\in\Z. 
\] 
This $\mu_2$-index seems classical but 
we do not know if it has appeared elsewhere 
in the literature.

We obtained the following results concerning 
$\mu_2$: 

\begin{prop} 
\begin{enumerate} 
\item $\mu_2(L;\ff)$ is independent of the orientation of $L$. 
\item $\mu_2(L;\ff)$ depends only on the homotopy 
class of $\ff$ in $\cF^\omega$ the set of all $\omega$-compatible unitary framings. 
\item $\mu_2(L,\ff)$ is invariant under the $la$-disk surgery. 
\end{enumerate} 
\end{prop} 
 Moreover, if $H_1(W,\Z)=0=H_3(W,\Z)$ 
then $\cF^\omega$ is connected and $\mu_2(L,\ff)$ is 
independent of $\ff$. 

Clearly $\mu_2$-index is invariant under 
regular homotopy of immersed Lagrangian surfaces, hence 
not sensitive enough to 
distinguish Hamiltonian or symplectic isotopy classes of Lagrangian surfaces. 
For example, both Chekanov tori and Clifford tori have their 
$\mu_2$-indexes equal to $0$ (see Section \ref{WhTo}).

A closer inspection on the map $g'_L$ reveals that a $la$-disk surgery 
seems make some essential change that can be described in terms of the variation of the intersection 
subspaces between complex planes and  Lagrangian planes, 
leading to the definition of the $y$-index which we 
sketch below:

The framing $\ff=(J,u,v)$ associates a unique family of 
$J$-complex line bundles 
\[ 
E_\h=u_\h\wedge Ju_\h, \quad \h\in\R/\pi\Z, 
\] 
where $u_\h:=u\cos \h+v\sin \h$. Let $v_\h:=Ju_\h$. The pair $(u_\h, v_\h)$ 
is well-defined up to a simultaneous change of $\pm$-signs. 
This sign ambiguity however will not hinder our construction 
below. 
Observe  that $E_{\h+\frac{\pi}{2}}$ is orthogonal to $E_\h$, 
so we also denote 
\[ 
\h^\perp:=\h+\frac{\pi}{2} \ \text{ and } \  E^\perp_\h:=
E_{\h^\perp}, \quad  \text{for } \h\in\R\mod \pi. 
\] 
Let $L\subset W$ be 
an oriented compact immersed Lagrangian surface. Then 
for any $p\in L$, there exists $\h,\h^\perp\in \R/\Z$ such 
that 
\[ 
\dim T_pL\cap E_\h|_p=1=\dim T_pL\cap E_{\h^\perp}|_p. 
\] 
We call the set 
\[ 
\Gamma_\h:=\{ p\in L\mid \dim T_pL\cap E_\h=1\}
\] 
 the {\em intersection locus of $E_\h$} or the {\em $E_\h$-locus}  in $L$. It turns out that we can orient a subset $\check{\Gamma}_\h$ of $\Gamma_\h$, call the {\em proper $E_\h$-locus}, in a consistent way for all 
$\h$ so that the positively oriented proper loci, denoted as $\check{\Gamma}^+_\h$, satisfy 
\[ 
\check{\Gamma}^+_{\h^\perp}=-\check{\Gamma}^+_\h=:\check{\Gamma}^-_\h, 
\] 
with 
$\check{\Gamma}^-_\h$ denotes the negatively oriented 
$\check{\Gamma}_\h$.  
Each $E_\h$ comes with a trivialization 
given by $(u_\h,v_\h)$, with which we can count the total 
angle of variation $\alpha^+_\h$ (resp. $\alpha^-_\h$) in $E_\h$ 
(resp. in $E_{\h^\perp}$) of the intersection subspace $T_pL\cap
 E_\h$ (resp. $T_pL\cap E_{\h^\perp}$) as we traverse all 
 components of $\check{\Gamma}^+_\h$ once. The difference 
 \[ 
 \alpha_\h:=\alpha^+_\h-\alpha^-_\h
 \] 
  which we call the {\em relative $E_\h$-phase along $\check{\Gamma}^+_\h$} is independent of $\h$ 
 and is an integral multiple of $2\pi$.
Observe that we get the other uniform 
orientation for all $\check{\Gamma}_\h$ by simultaneously reversing 
the orientations of all $\check{\Gamma}^+_\h$. The sign of 
$\alpha_\h$ will be revered if we change the uniform orientation.

That $\check{\Gamma}^+_\h$ can be uniformly oriented comes 
from  a decomposition of $L$ into a finite number of {\em crossing domains} (see 
Definition \ref{pdomain}) 
$L_i$ $i\in I$,  
of $g'_L$, and the existence of a symmetric function $\varepsilon :I\times I\to \{ \pm 1\}$ is defined so that  $\varepsilon(i,j)\varepsilon(j,k)=\varepsilon(i,k)$ for 
all $i,j,k\in I$. 
The choice of a 
reference {\em crossing domain} $L_{i_0}\subset L$ 
determines a uniform orientation of $\check{\Gamma}^+_\h$. 
Choosing another $L_i$ gives the same uniform orientation 
iff $\varepsilon(i_0,i)=1$.

Fix a reference crossing domain $L_{i_0}$ (and hence a uniformly 
oriented $\check{\Gamma}^+_\h$), we define the $y$-index of 
$L$ to be 
\[ 
 y(L,q;\ff):=\frac{1}{2\pi}\alpha_\h. 
 \] 
Here $q\in L_{i_0}$ is a regular point of $g'_L$, and it is to 
represent the uniform orientation determined by $L_{i_0}$.  
Clearly $y(L,p,\ff)=y(L,q,\ff)$ if $p\in L_i$ and $q\in L_j$ 
satisfy $\varepsilon(i,j)=1$. 
We define the {\em absolute 
 $y$-index} to be 
\[ 
\bar{y}(L;\ff):=|y(L,q;\ff)|. 
\]

For each $i$, the degree $d_i\in\Z$ of the restricted 
PLG-map $g'_L|_{L_i}:L_i\to \dP(K')$ is defined. 
Let $L_{i_0}$ and $q$ be as above.  Then $y(L,q,\ff)$ 
can also be expressed as 
\[ 
 y(L,q;\ff)= \sum_{j\in I}\epsilon(i_0,j)d_j . 
\] 
So $y(L,q;\ff)$ can be 
defined by summing up the signed degrees of crossing 
domains of $L$ relative to a reference crossing domain. 
It is this extra sign 
that 
set the $y$-index apart from the $\mu_2$-index.

Given two oriented immersed Lagrangian surfaces $L,L'$, 
suppose  that $L\cap L'$ contains an open domain $U$ 
on which $L$ and $L'$ induces the same orientation, 
and suppose that $q\in U$ is a regular value for $g'_L$ 
(and hence for $g'_{L'}$, then a {\em relative $y$-index} 
is defined: 
\[ 
y(L',L,q;\ff):=y(L',q;\ff)-y(L,q;\ff). 
\]

We have the following results: 

\begin{theo} 
\begin{enumerate} 
\item
$y(L,q;\ff)$ and $y(L',L,q;\ff)$ are  independent of the orientation of $L$ and $L'$.  

\item For $L,L'$ fixed, $y(L,q;\ff)$, $\bar{y}(L;\ff)$ 
 and $y(L',L,q;\ff)$ depends only on the path 
connected components of $\ff)$.

\item If $L$ and $L'$ are symplectically isotopic, then 
$\bar{y}(L;\ff)=\bar{y}(L';\ff)$ for $\ff\in\cF^\omega$. 

\item If $L'=\eta_D(L)$ then 
\[ 
|y(L', q;\ff)-y(L,q;\ff)|=4, 
\] 
\end{enumerate} 
\end{theo}

When $H_1(W,\Z)=0=H_3(W, \Z)$  every 
$\omega$-compatible almost complex structure $J$, the 
set of $J$-complex unitary framings $(u,v)$ is path connected. 
Since the set of all $\omega$-compatible almost complex 
structures over $W$ is contractible, the set $\cF^\omega$ 
is connected. Then the $y$-index $y(L,q;\ff)$ is independent 
of $\ff\in \cF^\omega$ and is invariant under symplectomorphisms 
of $L$. In this case we will usually omit $\ff$ and simply denote the $y$-index of $L$ 
as $y(L,q)$, and $\bar{y}(L;\ff)$ as $\bar{y}(L)$, etc.. 

By computing the relative $y$-index we reproved that iterations 
of the generalized Dehn twist can produce an infinite number 
of smoothly isotopic Lagrangian surfaces representing distinct 
symplectic/Hamiltonian isotopy classes. 

\begin{theo} \label{LvS} 
Let $W$ be the plumbing of the cotangent bundles of a smooth 
orientable surface $L$ and a sphere $S$ at their intersection 
point. Let $L^n:=\tau^{2n}_S(L)$ denote the Lagrangian surface 
in $W$ obtained by applying the $2n$-generalized Dehn twist 
along $S$ to $L$, $n\in \Z$, $L^0=L$.Then 
for a suitable common point $q$ of $L^m$'s, 
\[ 
y(L^m,L^n,q;\ff)=4(m-n), \quad m,n\in \Z. 
\] 
In particular, 
\[ 
y(L^n,L,q;\ff)=-4n, \quad n\in \Z. 
\] 
This implies that 
there are infinitely many Lagrangian surfaces in $W$ which are pairwise Hamiltonian non-isotopic, but all smoothly isotopic.  
\end{theo} 

Let $W_n$ denote the cotangent bundle of the $A_n$-configuration of $n$'s $2$-spheres, or equivalently, the pluming of 
cotangent bundles $T^*S_j$ of $n$'s $2$-spheres $S_1,...,S_n$, so that in $W_n$, $S_i\cap S_j=\emptyset$ unless $j=i+1$, 
$S_i$ intersects transversally with $S_{i+1}$ and in one point. 
We call $W_n$ the $A_n$-manifold. 
Observe that  the above result also applies to  the 
plumbing $T^*L$ with $W_n$ for $n\geq 2$. 
We remark here that a relevant result was proved by Seidel 
\cite{Se} by way of Lagrangian Floer homology.

The $la$-disk surgery and the relative $y$-index enable us to 
construct new nullhomologous monotone Lagrangian tori beyond the known ones. 

\begin{theo} \label{tori} 
Let $W_n$ be the $A_n$-manifold, $n\geq 0$, $W_0=\R^4$. Then on $W_n$ there are 
$n+2$ smoothly isotopic 
nullhomologous monotone Lagrangian tori, $T_{-1}, T_0, 
T_1,...,T_n$, with a common domain containing a regular point $q$, such that 
\[ 
y(T_k,T_j,q)=4(j-k), \quad -1\leq k,j\leq n,  
\] 
hence are pairwise Hamiltonian (and symplectically) non-isotopic. 
\end{theo} 
Note that $T_{-1}$ is a Chekanov torus, $T_0$ a Clifford torus, 
and $T_1$ is Hamiltonian isotopic to the torus in $T^*S^2$ 
constructed in \cite{AF}.

With $y$-index and its relative version we reproved earlier examples of smooth 
but not symplectically isotopic monotone Lagrangian tori 
obtained by Chekanov \cite{C} and Albers-Frauenfelder \cite{AF}, 
and spheres by Seidel \cite{Se}. Their methods include 
{\em symplectic capacities} \cite{EH1, EH2, V1, V2} applied in \cite{C} and 
{\em Lagrangian (intersection) Floer cohomology}   
\cite{O1, O2, Poz, KS} used in \cite{Se, AF}. We remark here that by computing {\em superpotentials}  
Auroux \cite{Aur1} proved that the monotone Clifford torus and Chekanov torus are not Hamiltonian isotopic in $\C P^2$.  
Complement to current methods,  
our $y$-index provides an alternative and 
simpler way of distinguishing Lagrangian 
surfaces in symplectic $4$-manifolds with vanishing Chern 
classes. Extension of the $y$-index to general symplectic 
$4$-manifolds is yet to be explored. 

It is expected that a contact version of the $la$-disk surgery 
can be defined for Legendrian 
surfaces in a contact $5$-manifold $(M,\xi)$, and the same to  the $y$-index  provided that the 
contact distribution $\xi$ is parallelizable. 
We also expect that both the $la$-disk surgery and the $y$-index can be generalized  to  Lagrangian submanifolds immersed in higher dimensional parallelizable symplectic manifolds.


\subsection{Outline of this paper}

This paper is organized as follows: 
In Section \ref{ladisk} we introduce  the notion of a Lagrangian 
attaching disk ($la$-disk) and analyze the types of such disks.  
A standard model for a $la$-disk is established in Section \ref{model} to define the $la$-disk surgery $\eta_D$ in 
Section \ref{surg-la} and to compare $\eta_D$ with the 
generalized Dehn twists $\tau^n_S$ in Section \ref{ddtwist}. 
In Section \ref{prop-surg} we study the effect of $\eta_D$ on 
Maslov class, Liouville class and monotonicity of a Lagrangian 
surface. In Section \ref{LG} we consider an equivariant  factorization of 
oriented Lagrangian Grassmannian $\Lambda^+$. In Section \ref{intCL} 
 the intersection loci between complex and  Lagrangian planes are analyzed and applied to define a coordinate system 
for $\Lambda^+$ in Section \ref{sph} and to study maps into 
$\dP(K')$ in Section \ref{PK'}. In Section \ref{PK'} we define 
{\em crossing $p$-curves} and analyze their deformations. Also 
defined are {\em crossing domains} which will be used to define 
the $y$-index in Section \ref{y}. 
Staring from Section \ref{paraM} 
we will work with parallelizable symplectic 4-manifolds. The 
definitions and basic properties of the $\mu_2$-index and 
the $y$-index are presented in Section \ref{mu2index} and Section \ref{y} 
respectively. Effects of a $la$-surgery on both indexes 
are discussed in Section \ref{la-effect}. 
Examples are computed in Section 5. In Section \ref{R4} indexes of 
tori and Whitney spheres in $\R^4$ are determined. Indexes of 
the zero section of $T^*S^2$ are obtained in Section \ref{T*S}. 
Section \ref{plumb} is devoted to prove 
Theorem \ref{LvS}, and Section \ref{Wn} Theorem \ref{tori}. 

\vspace{.1in} 
\noindent
{\bf Acknowledgments.} \  
Part of the work was done during the author's visit 
at University of California, Berkeley in 2009-2010, which was 
partially founded by NSC grants 98-2918-I-008-003 and  99-2115-M-008-005-MY2. 
The author would like to thank Department of Mathematics 
of UC Berkeley  and in particular  Professor Robion Kirby for their hospitality. 

%
%

%
%

%
%
\section{Surgery of a Lagrangian surface} \label{surg}

Lagrangian surgery, viewed as attaching Lagrangian handlebodies along contact boundary, is an important ingredient in the construction/decomposition of of Stein manifolds 
\cite{E1}. In a different meaning it was defined by Polterovich  in \cite{P2} as a 
Lagrangian de-singularization procedure. In fact, it was already 
hidden in the 
Hamiltonian integrable system literature (see for example 
\cite{Du1,Ba1}) and is responsible for the monodromy of the 
integrable system. Yet its effect on isotopy of Lagrangian 
surfaces (the so called {\em Lagrangian knots}) has not been fully explored, and this is the 
direction we will pursue here. 
We will consider a 2-dimensional Lagrangian surgery 
as a surgery on a Lagrangian surface in the presence of an 
embedded Lagrangian attaching disk.

Recall that a submanifold $\iota:C\hookrightarrow W$ of a 
symplectic manifold $(W,\omega)$ is called {\em isotropic} 
if $\iota^*\omega=0$. One dimensional submanifolds are 
always isotropic. For an isotropic submanifold $C$ we denote 
\[ 
(TC)^\omega:=\{ v\in T_CW\mid v\in T_pW \text{ for some $p\in C$}, \ \omega (v, v')=0 \ \forall v'\in T_pC\}.  
\] 
Then $(TC)^\omega$ is a vector subbundle of rank $\dim W-\dim 
C$ over $T_CW$, and it contains $TC$ as its subbundle. 
The quotient bundle 
\[ 
N^\omega_C:=(TC)^\omega/TC 
\] 
is a symplectic vector bundle over $C$, called the {\em 
symplectic normal bundle} of $C$ in $W$. 

We will also denote the normal bundle of $C\subset D$ by 
$N_{C/D}$ when $C$ is viewed as a submanifold of a manifold 
$D$. 

\subsection{Lagrangian attaching disk} \label{ladisk}

Let $L\subset (W,\omega)$ be a closed Lagrangian surface immersed in a symplectic 
4-manifold $(W,\omega)$. 

\begin{defn} \label{attachD}
{\rm A closed embedded Lagrangian disk $D\subset W$ is called a {\em Lagrangian 
attaching disk} ($la$-disk) of $L$ if 
\begin{enumerate} 
\item $C:=\pa D=D\cap L$, and 
\item $D\pitchfork _CL$, i.e. $D$ is transversal to $L$ along 
$C$ in the sense that the normal bundles $N_{C/D}$ and 
$N_{C/L}$ are transversal along $C$, or equivalently, the intersection 
$L\cap D=C$ is {\em clean}, i.e. $T_CL\cap T_CD=TC$.  
\end{enumerate} 
We call  $C\subset L$ a {\em vanishing cycle} of $L$ if there exists a $la$-disk 
$D$ of $L$ with $\pa D=C$. 
} 
\end{defn} 

The following are two simple propositions regarding a vanishing cycle of $L$ and its neighborhood. 

\begin{prop} \label{v-cycle} 
If $C$ is a vanishing cycle of $L\subset W$ then 
\begin{enumerate} 
\item $C=\pa D\subset L$ is $\omega$-exact, and 
\item the Maslov index of the loop of Lagrangian planes $T_CL$ ($C$ is oriented in 
either way) with respect to $[D]\in H_2(W,L;\Z)$ is 
\[ 
\mu(T_CL, D)=\mu(T_CD, D)=0. 
\] 
\end{enumerate} 
In particular, if $c_1(W)=0$ then $\mu(T_CL)=0$ is well-defined, independent of the 
choice of a la-disk $D$ with $\pa D=C$. 
\end{prop}

\begin{prop} \label{annulus} 
The outward normal vector field $\nu$ of  a la-disk of $L$ along $C=\pa D$  induces a trivialization of the 
symplectic normal bundle $N^\omega_C$ over $C$. Since 
$N_{C/L}\subset N^\omega_C$ as a subbundle, this implies 
that a collar neighborhood of $C\subset L$ is an annulus, and in particular, not a M\"{o}bius band. 
\end{prop} 

\noindent
{\bf Type of Lagrangian attaching disk.} \ 
According to the free homotopy type of its boundary cycle $C$, 
a $la$-disk $D$  of a Lagrangian surface $L$ is called 

\begin{itemize} 
\item {\em parabolic} if $0\neq [C]\in H_1(L,\Z)$, 
\item {\em elliptic} if $0=[C]\in\pi_1(L)$ (hence $C$ bounds a disk in $L$), 
\item {\em hyperbolic} if $[C]=0\in H_1(L,\Z)$ but $0\neq [C]\in \pi_1(L)$. 
\end{itemize}

\noindent
{\bf Polarity of a $la$-disk.} \ 
The outward normal vector field $\nu$ of $D$ along $\pa D=C$ induces an orientation 
of the $\R$-bundle $N^\omega_C/N_{C/L}=(TC)^\omega/T_CL$ over $C$, thus we have  a dichotomy for each of the three types of the $la$-disks according to the orientation of 
$(TC)^\omega/T_CL$ induced by $\nu$. 

\begin{defn} \label{polarity} 
{\rm 
We say two la-disks $D,D'$ of $L$ with $\pa D=C=\pa D'$ have 
different {\em polarity} if their corresponding outward normals give different orientations 
of the bundle $(TC)^\omega/T_CL$ over $C$. 
} 
\end{defn}

Recall that $\omega=d\lambda$ is exact when restricted to a 
neighborhood of a Lagrangian surface $L\overset{\iota}{\hookrightarrow} W$. The pullback 1-from $\iota^*\lambda$ 
is closed in $L$ and its cohomology class in $H^1(L,\Z)$ is 
 called the {\em Liouville class} of $L$.

 Assume that $H^1(W,\Z)$ and the first Chern class $c_1(W)=0$ vanishes for the moment. 
Then the Liouville class $\lambda_L\in H^1(L,\R)$ is 
independent of the choice of a local primitive $\lambda$ of 
$\omega$ near $L$, and the Maslov class $\mu_L$ of $L$ 
is a cohomology class in $H^1(L,\Z)$. We say $L$ is 
{\em monotone} if there exists a number $c>0$ such that 
$\lambda_L=c\cdot\mu_L$. 

In this case, if $L$ is a torus 
and $D$ is a parabolic 
$la$-disk of $L$, the polarity of $D$ can be described as follows:  
Parametrize $L$ as $\R^2_{x_1,x_2}/\Z^2$ so that the Liouville form 
is $adx_1\in \Omega^1(L)$ for some $a>0$, 
and $C$ is  identified with $\{ x_1=0\}$. 
Let $(y_1,y_2)$ be the fiber coordinates of $T^*L$ dual to $(x_1,x_2)$. 
The canonical symplectic 
form of $T^*L$ is then $\sum_{j=1}^2dx_j\wedge dy_j$. 
Let $\nu$ denote the outward normal vector field of $D$ along $C=\pa D$. 

\begin{defn} \label{stable} 
{\rm Let $D$ be a parabolic $la$-disk of a Lagrangian torus $L$  as in above. 
We say that $D$ is 
\begin{itemize} 
\item {\em stable} if the $\pa_{y_1}$-component of $\nu$ is positive; 
\item {\em unstable} if the $\pa_{y_1}$-component of $\nu$ is negative. 
\end{itemize} 
}
\end{defn} 
The terminologies come from the observation that, with Lagrangian neighborhood theorem, 
we can take a local 
primitive 1-form of $\omega$ (defined near $L$) to be $(a-y_1)dx_1$. Then we get a 
family of monotone Lagrangian tori $L_t$ defined by $y_1=t$ and $y_2=0$. The Liouville 
class of $L_t$ increases (resp. decreases) in the direction of $-\pa_{y_1}$ 
(resp. $\pa_{y_1}$), 
so an unstable $D$ indicates that the Liouville class grows if we let $L_t$ vary along 
the direction of $\nu$, whilst a stable $D$ suggests the opposite.

Also defined is the notion of  {\em relative polarity} associated 
to a Lagrangian disk surgery. See Section \ref{surg-la} for detail.


\begin{exam} \label{cheka} 
{\rm 
Let $L\subset \R^4$ be the Chekanov torus defined as the orbit of the plane curve 
$\gamma=\{ (x_1-2)^2+y_1^2=1\}\subset \R^2_{x_1,y_1}\times \{ 0\}\subset\R^4$ 
under the $S^1$ group action induced by the Hamiltonian vector field 
\begin{equation} \label{X_G}
X_G:=x_1\pa_{x_2}-x_2\pa_{x_1}+y_1\pa_{y_2}-y_2\pa_{y_1}
\end{equation} 
defined by  
\[ 
\iota(X_G)\omega=-dG, 
\] 
where $\omega=\sum_{j=1}^2dx_j\wedge dy_j$ is the standard 
symplectic form on $\R^4$, and 
\[ 
G:=x_2y_1-x_1y_2 :\R^4\to \R 
\] 
is the corresponding Hamiltonian function. 
The Lagrangian disk $D:=\{ x_1^2+x_2^2\leq 1, \ y=0\}$ is a {\em stable} 
parabolic $la$-disk of $L$.   
}
\end{exam} 

\begin{exam} \label{cliff}
{\rm 
Let $L'\subset \R^4$ be the monotone Clifford torus defined as the orbit of the plane curve 
$\gamma:=\{ x_1^2+y_1^2=1\}\subset \R^2_{x_1,y_1}\times \{ 0\}\subset\R^4$ 
under the $S^1$ group action induced by the Hamiltonian vector field $X_G$ as defined in (\ref{X_G}).  
The Lagrangian disks $D:=\{ x_1^2+x_2^2\leq 1, \ y=0\}$ and 
$D':=\{ y_1^2+y_2^2\leq 1, \ x=0\}$ are both {\em unstable} parabolic 
$la$-disks of $L'$. 
}
\end{exam} 

\begin{exam} \label{geode}
{\rm 
Let $L''\subset T^*S^2\cong TS^2$ be the union of graphs of geodesics on $S^2$ passing 
through the north pole $q_+$ and south pole $q_-$ and with unit speed. $L''$ is an embedded monotone Lagrangian torus, $0=[L]\in H_2(T^*S^2,\Z)$. The normal 
disks $D_\pm:=\{ (q_\pm,p)\mid p\in T^*_{q_{\pm}}S^2, \ |p|\leq 1\}$ are disjoint 
unstable parabolic $la$-disks of $L''$. 
}
\end{exam} 

In examples above the tori are all nullhomologous and the 
$la$-disks are parabolic. 
Proposition \ref{[L]} below shows that indeed nullhomologous Lagrangian surfaces allow 
only parabolic $la$-disks. 
%
%
\subsection{Standard model and topological implications} \label{model}

We start with a standard model for $L$ near its $la$-disk 
$D$.

 Let $(x_1,y_1,x_2,y_2)$ be coordinates of $\R^4$ 
so that $x_1+iy_1, x_2+iy_2$ are the corresponding complex 
coordinates of $\C^2\cong \R^4$. 
Consider the $S^1$-group $\cG\subset SU(2)$ 
whose matrix representation with respect to the complex basis 
$\{ \pa_{x_1},\pa_{x_2}\}$ is 
\begin{equation} \label{Ggroup} 
\cG=\Big\{ g_\theta:=\begin{pmatrix} \cos\theta & -\sin\theta\\ 
\sin\theta & \cos\theta \end{pmatrix}\mid \theta\in\R/2\pi\Z \Big\}. 
\end{equation} 
Note that $g_\theta$ is the time $\theta$ map $X_G^\theta$ 
of the flow generated by the Hamiltonian vector field 
$X_G$ defined in (\ref{X_G}). One can check easily the following 
fact.  

\begin{fact} \label{revolutionL} 
The 
$\cG$-orbit of any curve immersed in $\R^2_{x_1y_1}$ is 
an immersed Lagrangian surface in $\R^4$. 
\end{fact} 

\begin{notn} 
{\rm 
For a group $\cG$ and a set $\gamma$ we denote by 
$Orb_\cG(\gamma)$ the $\cG$-orbit of $\gamma$. 
} 
\end{notn} 

Let $D$ be a $la$-disk of $L$. 
We can pick an open neighborhood $U\subset W$ of $D$ 
which  can be symplectically identified with an open domain $V\subset T^*\R^2$ so that 
under this identification $U$ contains the closed  ball  
$B_r$ of radius $r$ with center $0\in \R^4$, such that 
\begin{itemize} 
\item $D=\{ x_1^2+x_2^2\leq (\sqrt{2}-1)^2r^2, y_1=0=y_2\}$. 
\item $Q:=L\cap B_r=Orb_\cG(\gamma)$, where  $\gamma:(\frac{3\pi}{4},\frac{5\pi}{4})\to \R^2_{x_1y_1}$  is  the 
curve defined by 
\begin{equation} \label{gamma}
\gamma (s):=(x_1=\sqrt{2}r+r\cos s, \ y_1=r\sin s). 
\end{equation} 
\end{itemize} 
Note that $D,Q, C:=L\cap D$ are all invariant under the Hamiltonian $\cG$-action. Without loss of generality we may 
also assume that $U$ and $L\cap U$ are also $\cG$-invariant. Objects with such symmetry can be  
viewed as $X_G$-orbits of their sections in 
(the right half-space of ) $\R^2_{x_1,y_1}$. 

For example, $D$ is the $\cG$-orbit of the line segment 
\[ \ell:=\{ (x_1,0)\in\R^2_{x_1y_1}\mid 0\leq x_1\leq (\sqrt{2}-1)r\}, 
\] 
and 
the complement $(L\cap B_r)\setminus C$ consists of two annuli 
\[ 
Q':=Orb_\cG(\gamma([\frac{3\pi}{4},\pi))), \quad 
Q'':=Orb_\cG(\gamma((\pi,\frac{5\pi}{4}])). 
\] 

Below we construct in $U$  a pair of 
$\cG$-invariant Lagrangian disks  
to be used in Proposition \ref{[L]}. 

First observe that $Q'\cup D$ is a piecewise smooth Lagrangian disk which is the $\cG$-orbit of the broken 
curve $\ell\cup \gamma([\frac{3\pi}{4},\pi)))$. We smooth out the broken curve at the corner to get a new smooth curve 
$\sigma'$. Then $D':=Orb_\cG(\sigma')$ is a smooth Lagrangian 
disk tangent to $L$ near its boundary. Similarly, $Q''\cup D$ is the $\cG$-orbit of the broken 
curve $\ell\cup\gamma((\pi,\frac{5\pi}{4}])$. Let $\sigma''$ be 
the smooth curve obtained by smoothing out the corner. Then 
$D'':=Orb_\cG(\sigma'')$ is another smooth Lagrangian disc 
tangent to $L$ near its boundary. 
We may perturb $D',D''$ 
in an $\cG$-invariant way (by perturbing $\sigma',\sigma''$) 
so that both $D',D''$ are contained in $B_r\subset U$, $D'\pitchfork D''$, and $D'\cap D''$ consists of a 
single point: the origin of $\R^4$ in our local model. 
 Let $C':=\pa D'$ and $C'':=\pa D''$.


\begin{prop} \label{[L]} 
Let $L$ be a closed oriented Lagrangian surface in a symplectic 
manifold $(W,\omega)$. Suppose that $L$ admits a non-parabolic  $la$-disk,  
then $[L]\in H_2(W,\Z)$ is nontrivial and of infinite order. In other words, if $[L]\in 
H_2(W,\Z)$ is a torsion then $L$ has only parabolic $la$-disks. 
\end{prop}

\begin{proof} 
Let $D$ be a non-parabolic $la$-disk of $L$ and let $C:=\pa D\subset L$. Then $L\setminus C$ consists of two connected 
components. 
Take $D' ,D''$ as constructed above. Let $Q\subset L$ denote the annulus containing $C$ 
with $\pa Q=C'\cup C''$. 
Let $L',L''$ be the two connected components of $L\setminus Q$ with genus 
$g',g''$ respectively, so that $\pa L'=C'$ and $\pa L''=C''$. 
Both $L'$ and $L''$ are equipped with 
the induced orientation coming from that of $L$. 
Now let $\tilde{L}':=L'\cup D'$ and $\tilde{L}'':=L''\cup D''$.  The two Lagrangian 
surfaces $\tilde{L}',\tilde{L}''$  intersect transversally and in a single point, with  intersection number $1$ (with orientations 
induced from $L$). Thus  $\tilde{L}'$ and $\tilde{L}''$ represent nontrivial elements of $H_2(W,\Z)$ 
of infinite order. Let $g:=g'+g''$ denote the genus of $L$. 
Since $[L]=[\tilde{L}']+[\tilde{L}'']\in H_2(W,\Z)$ we have 
$[L]^2=([\tilde{L}']+[\tilde{L}''])^2=(2g'-2)+2+(2g''-2)=2g-2$. 
If $g\neq 1$ then $[L]$ is not a torsion class. If $g=1$, then 
up to a change of notation we may assume that $g'=0$ and 
$g''=1$. Since $[\tilde{L}']^2=2g'-2=-2\neq 0=2g''-2=[\tilde{L}'']^2$ and both $[\tilde{L}']$ and $[\tilde{L}'']$ are of infinite 
order, $[\tilde{L}']$ and $[\tilde{L}'']$ are 
linearly independent over $\Z$, hence $[L]$ is of 
infinite order. This completes the proof. 
\end{proof} 

So it is impossible to find a non-parabolic $la$-disk for a monotone Lagrangian torus $L$ in $\R^4$, as $L$ is nullhomologous. Note that an orientable nullhomologous Lagrangian surface 
must be a torus. It turns out that there is a uniform upper bound to the maximal number of 
disjoint parabolic $la$-disks of any nullhomologous Lagrangian torus $L\subset W$, provided 
that $W$ has bounded topology. 

\begin{prop} \label{nulltorus} 
Let $L\subset W$ be a Lagrangian torus with $[L]=0\in H_2(W,\Z)$. Suppose that $L$ 
possesses $n+1$ pairwise disjoint (parabolic) $la$-disks, $n\in\N$. We denote these disks as 
$D_0,D_1,D_2,...,D_n$ in a cyclic order (so $D_n=D_{-1}$ and $D_{n+1}=D_0$, etc.) and their corresponding boundaries as 
$C_0,C_1,C_2,...,C_n$. Let $B_i\subset L$, $1\leq i\leq n$, denote the annulus bounded by 
$C_{i-1}$ and $C_i$. 
Then there are $n$ embedded Lagrangian spheres 
$S_1,S_2,...,S_n$ in $W$ such that under suitable orientations 
\begin{enumerate} 
\item $[S_i]=[D_{i-1}\cup B_i\cup D_i]\in H_2(W,\Z)$ for $1\leq i\leq n$; 
\item $S_i\pitchfork S_j$ for $1\leq i<j\leq n$; and 
\item 
\[ S_i\cap S_j= \begin{cases} 
\emptyset & \text{ if $|i-j|\neq 1$  (mod n),} \\ 
\{ pt\} & \text{ if $|i-j|= 1$}. 
\end{cases} 
\] 
\end{enumerate} 
In particular, the second Betti number of $W$ is $b_2(W)\geq n$. 
\end{prop} 

\begin{proof} 
For $0\leq i\leq n$ pick an open neighborhood $U_i$ of $D_i$ 
such that closures of $U_i$ are pairwise disjoint and for each $i$, $L\cap U_i$ 
is an open annulus. In each $U_i$ we construct a pair embedded Lagrangian disks 
$D'_i,D''_i$ as before, so that interiors of $D'_i,D''_i$ are disjoint from $L$, both 
disks are tangent to $L$ along their corresponding boundaries 
$C'_i,C''_i$, and 
$D'_i,D''_i$ intersect transversally and in a single point. 
By interchanging the notations $D'_i, D''_i$ if necessary, we may 
assume that 
\[ 
C'_i\subset B_{i+1},\quad C''_i\subset B_i. 
\] 
Let $Q_i$ denote the 
annulus containing $C_i$ and with boundary $\pa Q_i=C'_i\cup C''_i$. 
Let $\tilde{B}_i:=B_i\setminus(Q_{i-1}\cup Q_i)$ and 
\[ 
S_i:=D'_{i-1}\cup \tilde{B}_i\cup D''_i, \quad 0\leq i\leq n. 
\] 
Note that outward normals of $D'_{i-1}$ and $D''_i$ point into 
the interior of $\tilde{B}_i$. 
Then each of $S_i$ is an embedded Lagrangian sphere. With suitable orientations we have  
$[S_i]=[D_{i-1}\cup B_i\cup D_i]\in H_2(W,\Z)$ and $0=[L]=\sum_{i=0}^n[S_i]\in H_2(W,\Z)$, so is verified (i). 
One sees that (ii) and (iii) also follow easily from the construction of $S_i$. The intersection 
pattern among $S_i$ implies that $[S_1],...,[S_n]$ are linearly independent over $\Z$ as 
elements in $H_2(W,\Z)$. This completes the proof. 
\end{proof} 

\begin{rem}
{\rm if $n=0$ then $(L\setminus Q_0)\cup D'_0\cup D''_0$ is a Lagrangian sphere with one nodal point, i.e., a Lagrangian Whitney $2$-sphere. 
} 
\end{rem}

\begin{cor} \label{atmost1} 
Let $L\subset \R^4$ be an embedded Lagrangian torus in the standard symplectic 4-space. 
Then any two la-disks of $L$ must intersect, i.e., the maximal number of disjoint la-disks of 
$L$ is $\leq 1$. 
\end{cor}

The following two questions appear to be open (at least to the author). 

\begin{ques} 
It is true that every embedded monotone Lagrangian torus in $\R^4$ has a $la$-disk? 
\end{ques}

\begin{ques} 
Is it true that any two $la$-disks of a given monotone Lagrangian torus in $\R^4$ have  the same polarity?
\end{ques}

%
%
\subsection{Surgery via a $la$-disk} \label{surg-la}

Below we define a surgery on $L$ via a $la$-disk $D$ 
of $L$. Note that the union $L\cup D$ cannot be contained in any 
cotangent neighborhood of $L$, and neither is the new Lagrangian surface which we will construct below, hence the 
surgery is "not local" from the point of view of $L$. On the 
other hand, the surgery takes place in a symplectic chart (the cotangent neighborhood of $D$) and hence can  be described 
explicitly in local coordinates. \\ 

%
%

Let $D$ be a $la$-disk of $L$. Recall the standard model 
from Section \ref{ladisk}.

Let $M\in SO(4)$ be the anti-symplectic linear map whose matrix 
representation with respect to the basis $\{ \pa_{x_1},\pa_{x_2},
\pa_{y_1},\pa_{y_2}\}$ is 
\begin{equation} \label{M} 
M=\begin{pmatrix} O & I \\ I & O \end{pmatrix} 
\end{equation}
where $I,O\in SO(2,\R)$, $I$ is the identity matrix, and 
$O$ is the zero matrix. Note that $M$ commutes with $\cG$. 
Let 
\begin{equation} \label{gamma'}
\gamma':=M(\gamma). 
\end{equation}

\begin{figure}[th] 
 \begin{center} 
 \includegraphics[scale=0.6]{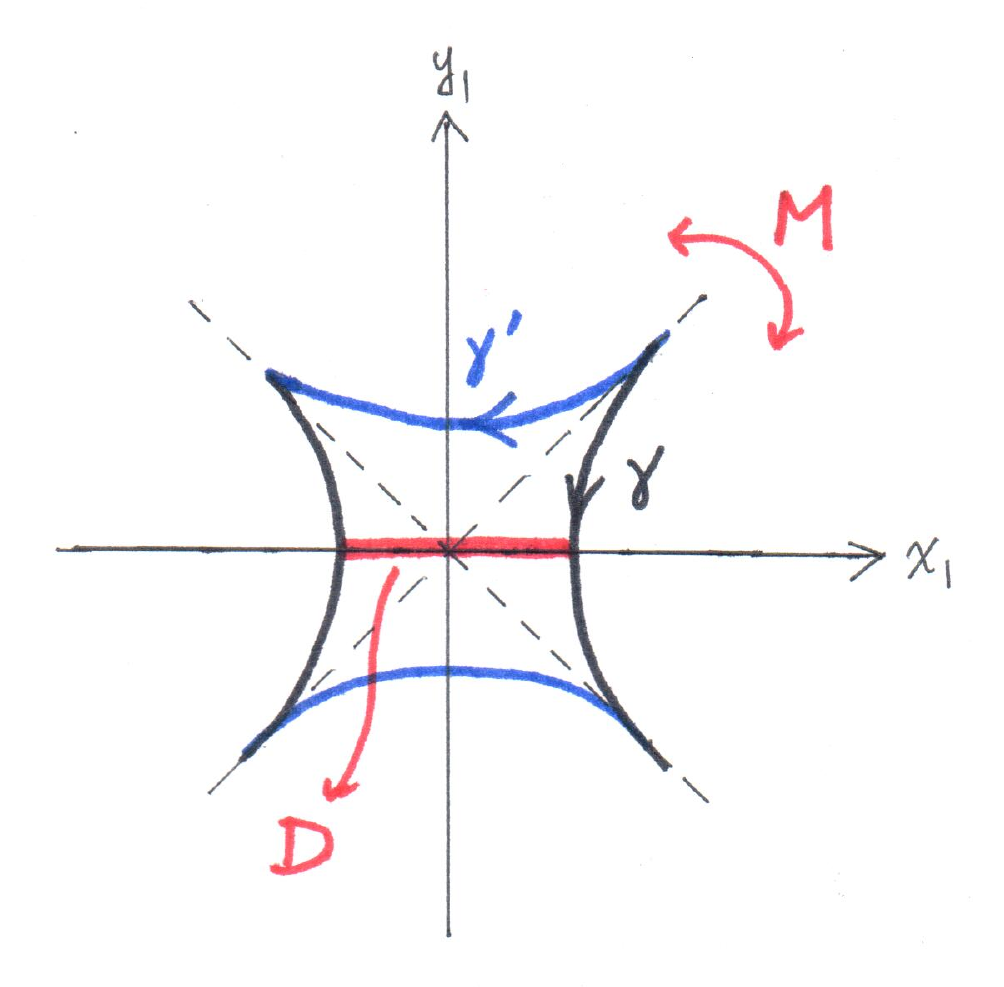}
\caption{$la$-disk surgery.  \label{laDsurg1}} 
\end{center} 
\end{figure}

Now define 
\begin{equation} \label{L'} 
\eta_D(L):=(L\setminus Orb_\cG(\gamma)) \cup Orb_\cG(\gamma'). 
\end{equation} 
Note that  $Q=Orb_\cG(\gamma)$ and $M(Q)=Orb_\cG(\gamma')$  
are tangent along their boundary $\pa Q=\pa (M(Q))=M(\pa Q)$, and both tangent to the pair of Lagrangian planes 
\[ 
E_+:=(\pa_{x_1}+\pa_{y_1})\wedge (\pa_{x_2}+\pa_{y_2}), \quad 
E_-:=(-\pa_{x_1}+\pa_{y_1})\wedge (\pa_{x_2}-\pa_{y_2}). 
\]

\noindent
{\bf Notation alert:} \ The notations $L',Q',C'$ that will be used 
below have noting to do with the same notations from Section 
\ref{model}.

\begin{prop} \label{smooth-iso} 
Let $L$ be a Lagrangian surface and $D$ a $la$-disk of $L$.  
Let $\eta_D(L)$ be obtained from  $L$ by performing a Lagrangian 
surgery on $L$ via $D$ as described above. 
Then $\eta_D(L)$ is Lagrangian surface smoothly isotopic to $L$. 
\end{prop} 

\begin{proof}
It is easy to see that $L':=\eta_D(L)$ is Lagrangian by Fact \ref{revolutionL}. 
To show that $L'$ is smoothly isotopic to $L$, observe 
that $M= M_\pi$ is an element of  the 
$S^1$-subgroup $\cM:=\{ M_\h\mid \h\in\R/2\pi\Z\}$ of $SO(4)$ whose elements in matrix form with 
respect to the orthonormal basis $\{ \pa_{x_1},\pa_{x_2},
\pa_{y_1},\pa_{y_2}\}$ are 
\begin{equation} \label{A_theta} 
M_\h:=\frac{1}{2}\begin{bmatrix} I+R_\h & I-R_\h\\ 
I-R_\h & I+R_\h\end{bmatrix}\in SO(4,\R), 
\end{equation}  
where 
\[ 
R_\h:=\begin{bmatrix} \cos\h & -\sin\h \\ \sin\h & \cos\h
\end{bmatrix} \in SO(2,\R). 
\] 

Both $E_\pm$ are preserved by all elements of $\cM$: 
$M_\h$ fixes the plane $E_+$ pointwise and rotates the (oriented) plane 
$E_-$ 
by an angle of $\h$ radians. 
Moreover, $\cM$ commutes with $\cG$.  
Let 
\[ 
Q_\h:=Orb_\cG(M_\h(\gamma))=M_\h(Orb_\cG(\gamma)). 
\] 
For $\h\in[0,\pi]$ define 
\[ 
L_\h:=(L\setminus Q)\cup Q_\h  
\] 
Each of $L_\h$ is diffeomorphic to $L$, 
$L_0=L$, $L_\pi=L'$. 
Then $L_\h$ is a smooth isotopy between $L$ and $L'$. 
So $L$ and $\eta_D(L)=L'$ are smoothly isotopic.  
\end{proof} 

%
%
%
%
%
%

\noindent
{\bf Dual $la$-disk surgery} \ 
Let $L$, $M$, and $L':=\eta_D (L)$ be as above. 
Let $D':=M(D)$ and $C':=M(C)=\pa D'$. Then $C'\subset 
L'$ is a vanishing cycle, $D'$ is a $la$-disk
of $L'$ along $C'$. Applying the standard model for $\eta_D$ one sees that 
\[ 
L=\eta_{D'}(L')=\eta_{D'}(\eta_D(L)). 
\]

\begin{rem}[{\bf Relation with Polterovich's Lagrangian surgery}] 
\label{Polter} 
{\rm 
In \cite{P2} Polterovich defined a Lagrangian surgery (for all dimensions) as a way of removing transversal self-intersection 
points of Lagrangian submanifolds. In the $2$-dimensional case  the surgery is done by first cutting off a neighborhood 
of the nodal point, which is a union of two embedded Lagrangian 
disks intersecting transversally and in a single point, and then 
closing up the two boundary circles by gluing a Lagrangian annulus to the complement of the nodal neighborhood along the 
boundary circles. Compared with our $la$-disk surgery, the 
nodal neighborhood is precisely $E_+\cup E_-$, the gluing 
annulus can be either $Q$ or $M(Q)$, and the 
resulting Lagrangian surface (with one nodal point removed) is 
 $L$ (resp. $\eta_D(L)$) if $Q$ (resp. $M(Q)$) is used for 
 gluing. 
}
\end{rem} 

\noindent
{\bf Relative polarity.} \ 
It is easy to check that $D,D'$ are of the same type as $la$-disks of $L,L'$ respectively, This can be seen by observing 
that the isotopy $L_\h$ takes $C=\pa D$ to $C'=\pa D'$, and induces  isomorphisms 
$H_1(L,\Z)\cong H_1(L',\Z)$ and $\pi_1(L)\cong \pi_1(L')$. 
As for polarity, 
there is a way to compare the polarity of $D,D'$ which we describe as follows: 

Fix an orientation of the annulus $Q\subset L$. 
Also fix an orientation of 
$C\subset Q$. Since $Q$ and $Q':=M(Q)\subset L'=\eta_D(L)$ are tangent along $\pa Q=\pa Q'$, orientation of $Q$ induces 
an orientation of $Q'$ so the two orientations coincides 
on $T_{\pa Q}Q=T_{\pa Q'}Q'$. It is easy to see that $C'$ also 
inherit an orientation compatible with that of $C$. Indeed, the 
orientations of the pair $(Q',C')$ are obtained by 
transporting the orientations of $(Q,C)$ via the 
isotopy $M_\h$. 

Let $\nu_C\subset T_CQ$ be a normal 
vector field to $C$ in $Q\subset L$ so that the ordered pair 
$(\nu_C, \dot{C})$ is a positive basis of $T_CQ$. Here 
$\dot{C}$ denotes the tangent vector field of $C$ with 
respect to some parameterization compatible with the orientation 
of $C$. Likewise Let $\nu_{C'}\subset T_{C'}Q'$ be a normal 
vector field to $C'$ in $Q'\subset L'$ so that the ordered pair 
$( \nu_{C'}, \dot{C}')$ is a positive basis of $T_{C'}Q'$.
We also denote by $\nu$ the outward normal to $D$ along $C$, 
by $\nu'$ the outward normal to $D'$ along $C'$. 
Note that the symplectic normal bundle $N^\omega_C$ is spanned by $\nu$ and $\nu_C$, and similarly $N^\omega_{C'}$ by $\nu'$ and $\nu_{C'}$.

\begin{defn} \label{rel-polarity} 
{\rm 
We say that $D$ and $D'$ {\em have the same relative polarity} 
if $\omega (\nu,\nu_C)$ and $\omega(\nu',\nu_{C'})$ are of the 
same $\pm$ sign; and $D$ and $D'$ {\em have opposite relative polarities} 
if $\omega (\nu,\nu_C)$ and $\omega(\nu',\nu_{C'})$ are of  
different $\pm$ signs. 
The notion of relative polarity is independent of the choices of 
orientations of $Q$ and $C$. 
} 
\end{defn}

\begin{prop} 
Let $D\subset L$ and $D'\subset L'=\eta_D(L)$ be the 
corresponding $la$-disks in the $la$-disk surgery. Then 
$D$ and $D'$ have  opposite relative polarities. 
\end{prop} 

\begin{proof} 
Let $D=\{ x_1^2+x_2^2\leq (\sqrt{2}-1)^2r^2, \ y_1=0=y_2\}$ and $Q=Orb_\cG\gamma$ be as in the standard model. Then $D'=\{ y_1^2+y_2^2\leq (\sqrt{2}-1)^2r^2, \ x_1=0=x_2\}$, $Q'=Orb_\cG\gamma'$. Since $D,D',Q,Q'$ are $\cG$-invariant, we only 
need to compare $\omega (\nu,\nu_C)$ at the point 
$p:=(x_1=(\sqrt{2}-1)r, y_1=0, x_2=0, y_2=0)$ and $\omega(\nu',\nu_{C'})$ at the point $p'=M(p)=(x_1=0,y_1=(\sqrt{2}-1)r, x_2=0,y_2=0)$. 
Without loss of generality we may take 
\[ 
\nu=\pa_{x_1}, \quad \nu_C=\pa_{y_1} \quad \text{ at $p$}. 
\] 
Then by applying $M$, we get 
\[ 
\nu'=\pa_{y_1}, \quad \nu_{C'}=\pa_{x_1} \quad \text{ at $p'$}. 
\] 
Since $\omega=\sum_{j=1}^2dx_j\wedge dy_j$, 
\[ 
\omega_p( \nu, \nu_C)=1>0, \quad \omega_{p'}(\nu',\nu_{C'})=-1<0. 
\] 
So $D$ and $D'$ have opposite relative polarities. 
\end{proof}

 In particular, if $L\subset W$ is a nullhomologous torus in $W$ 
 with $H^1(W,\Z)=0$, 
 and if $D$ is a stable (resp. unstable) 
$la$-disk of $L$, then $D'$ is a unstable (resp. stable) $la$-disk of $L'$.

\begin{exam}
{\rm Let $L\subset \R^4$ be the Chekanov torus and $D$ 
as defined in Example \ref{cheka}. Then a $la$-disk surgery along 
$D$ changes $L$ to $L'=\eta_D(L)\subset \R^4$ a monotone Clifford 
torus monotone Lagrangian isotopic to the one in Example \ref{cliff}. 
} 
\end{exam} 

\begin{exam} 
{\rm Consider the cotangent bundle $T^*S^2$ and regard 
the zero section $S^2=D\cup D'$ as a union of two closed Lagrangian disks $D,D'$ with $\pa D=C=\pa D'$ as the 
equator. Let $U\subset T^*S^2$ be a cotangent neighborhood 
of $D$ and $L'\subset U$ be a monotone Clifford torus with 
$D$ as its unstable $la$-disk, and $D'$ as its stable $la$-disk. 
Then a $la$-disk along $D'$ turns $L'$ into a monotone torus 
$L''\subset T^*S^2$ monotone Lagrangian which, up to a 
scaling by a Liouville vector field, is  isotopic to the 
geodesic torus as defined in Example \ref{geode}. 
} 
\end{exam}

%
%
\subsection{Relation with generalized  Dehn twists} \label{ddtwist}

A generalized Dehn twist is defined at the presence of an 
embedded Lagrangian sphere. 
Below we describe the model generalized Dehn twist following 
\cite{Se}. 

{\bf Generalized Dehn twist.} \ Identify 
a small neighborhood $U_S$ of an embedded Lagrangian $2$-sphere $S$ symplectically  with a neighborhood $V_0$ of the 
$0$-section of the cotangent bundle $T^*S^2$, and $S$ identified with $S^2$. Use the model 
\[ 
T^*S^2=\{ (q,p)\in \R^3\times \R^3\mid |q|=1  \text{ and } 
 \langle q,p\rangle =0\}, 
\] 
in which $\omega =\sum_i dq_i\wedge dp_i$. For $x\in \R^3\setminus \{ 0\}$ and $t\in \R$, let $R^t(x)\in SO(3)$ be the 
rotation with axis $x/|x|$ and angle $t$. Define 
\[ 
\sigma^{it}(q,p)=(R^t(q\times p)q,R^t(q\times p)p). 
\] 
Take a function $C^\infty(\R,\R)$ such that $\psi(t)+\psi(-t)=2\pi$ for all $t$, $\psi(t)=0$ for $t\gg 0$, $\psi(t)=\pi$ for 
small $|t|$, and $\psi(|\xi|)=0\mod 2\pi$ for $\xi\not\in V_0$. Then the model generalized Dehn twist 
$\tau_S:W\to W$ is a symplectomorphism with compact support 
contained in $U_S$, and its restriction to $U_S\cong V_0$ 
is 
\[ 
\tau_S(\xi) :=\begin{cases} 
\sigma(e^{i\psi(|\xi|)})(\xi), & \xi\in V_0\setminus S^2, \\ 
A \xi, & \xi\in S^2,  
\end{cases} 
\] 
where $A$ is the antipodal map on $S^2$. For $n\in \Z$ the  $2n$-th power of $\tau_S$ is 
\[ 
\tau^{2n}_S=\begin{cases} 
\sigma(e^{2ni\psi(|\xi|)})(\xi), & \xi\in V_0\setminus S^2, \\ 
\xi, & \xi\in S^2. 
\end{cases}
\]

{\bf Elliptic $la$-disk.}  \ 
Let $D$ be an {\em elliptic}  $la$-disk of a Lagrangian surface $L$ with boundary 
$C=\partial D\subset L$ which bounds an embedded disk 
$\Delta\subset L$. 
Identify a neighborhood $U$ of $\Delta$ symplectically with an open domain $V\subset 
T^*\R^2$ so that under this identification 
\begin{align*} 
\Delta & =\{ x_1^2+x_2^2\leq 1, \ y_1=0=y_2\}, \\ 
U & =\{ x_1^2+x_2^2<1+\epsilon_1, \ y_1^2+y_2^2<\epsilon_2^2\}, 
\end{align*} 
where $(y_1,y_2)$ are the corresponding fiber coordinates 
for $T^*\Delta$, and 
$D\cap U$ is one of the following two sets according to the 
polarity of $D$: 
\begin{itemize} 
\item $A_+:=\{ x^2=1, \ x_1y_2=x_2y_1, \ x_1y_1\geq 0\}\cap U$, 
\item  $A_-:=\{ x^2=1, \ x_1y_2=x_2y_1, \ x_1y_1\leq 0\}\cap U$. 
\end{itemize}

\begin{rem} 
{\rm 
Note that $(A_+ \setminus C) \sqcup (A_- \setminus C)$ 
is isomorphic to the quotient $\R$-bundle $(TC)^\omega/
T_C\Delta$ with the zero section $C$ deleted. 
The polarization of   $((TC)^\omega/T_C\Delta) \setminus C$ 
by assigning 
$\pm$ signs to each of the two connected components as defined above is independent of the choice 
of a coordinate system for $\Delta$. 
}
\end{rem} 

\begin{rem} 
{\rm 
The above sign assignment is even 
independent of the choice of $\Delta\subset L$ with 
$\partial \Delta=C$. Indeed, 
suppose there is another embedded disk $\hat{\Delta}\subset L$ with $\partial \hat{\Delta}=C$, and this happens precisely when $L$ is a sphere. Parameterize $\hat{\Delta}$ with coordinates $(\hat{x}_1,\hat{x}_2)\in \R^2$ so that $\hat{\Delta}=\{ (\hat{x}_1)^2+(\hat{x}_2)^2\leq 1\}$. Let $(\hat{y}_1,\hat{y}_2)$ be the corresponding fiber coordinates for $T^*\hat{\Delta}$. We may assume that along $C=\partial D=\partial \hat{\Delta}$ 
\[  
x_1=-\hat{x}_1, \quad x_2=\hat{x}_2. 
\] 

Let $\pi:T^*L\to L$ denote the canonical projection. 
For $p\in\pi^{-1}(C)$, 
its $(x,y)$ coordinates and $(\hat{x},\hat{y})$ coordinates are 
related by the equations 
\[ 
x_1=-\hat{x}_1, \ x_2=\hat{x}_2, \ y_1=-\hat{y}_1, \ y_2=\hat{y}_2. 
\] 
The equation $x_1y_2=x_2y_1$ is then equivalent to 
$\hat{x}_1\hat{y}_2=\hat{x}_2\hat{y}_1$. In addition $x_1y_1=\hat{x}_1\hat{y}_1$, 
so the sign assignment is independent of the choice of 
$\Delta\subset L$. 
}
\end{rem}

We proceed to analyze the surgery $\eta_D$ applied to $L$.

\begin{prop}  \label{dehn} 
Let $L\subset (W,\omega)$ be an embedded oriented Lagrangian surface.  Let $D$ be 
an elliptic $la$-disk to $L$ with $\partial D=C$. 
Let $\Delta\subset L$ be an embedded disk with $\pa\Delta=C$. 
Then the union 
$D\cup \Delta$ associates an embedded Lagrangian sphere 
$S$, $S$ is unique up to Hamiltonian isotopy,  such that $S$ intersects with $L$ 
transversally and in a single point.  Let $\nu$, $\nu'$ denote 
respectively the outward normal of $D$ and $\Delta$ along $C$. 
Let $\eta_D(L)$ denote the Lagrangian surface obtained by applying to $L$ the Lagrangian 
surgery on $L$ via $D$. 
We have the following conclusions: 
\begin{enumerate} 
\item Assume that $\omega(\nu,\nu')>0$, then $\eta_D(L)$ is Hamiltonian isotopic to 
$\tau^2_S(L)$, where $\tau_S$ is the positive generalized double Dehn twist along $S$. 
\item  Assume that $\omega(\nu,\nu')<0$, then $\eta_D(L)$ is Hamiltonian isotopic to 
$\tau^{-2}_S(L)$, the squared negative generalized double Dehn twist along $S$. 
\end{enumerate} 
Moreover, $\eta_D(L)\pitchfork S$ and $\eta_D(L)\cap S=L\cap S$. 
\end{prop}

\begin{proof}

\noindent
{\bf Case 1: \ $\omega(\nu,\nu')>0$. } \  
 In this case we have $D\cap U=A_+$.
We can construct an embedded Lagrangian sphere $S$  which intersects transversally 
with $L$ and in a single point. The construction is done 
by properly smoothing out the 
corner curve $C$ of the union $D\cup \Delta$ as follows. 

Let $f:[-1,1]\subset \R_{x_1}\to \R$ be a continuous function 
satisfying the following conditions: 
\begin{itemize} 
\item $f(-x_1)=-f(x_1)$, $0\leq |f|\leq \epsilon_2/2$,   
\item $f$ is smooth on $(-1,0)\cup(0,1)$, $f'(x_1)>0$ on $(-1,0)\cup(0,1)$,
\item  $\lim_{x_1\to 0}f'(x_1)=\infty$, $\lim_{x_1\to -1^+}f'(x_1)=\lim_{x_1\to 1^-} f '(x_1)=\infty$. 
\end{itemize} 
Let $\gamma_f\subset \R^2_{x_1y_1}$ denote the graph of $f$, and 
\[ 
S:=D_1\cup \Delta_f , 
\] 
where 
\begin{itemize} 
\item $D_1:=D - \{ y_1^2+y_2^2\leq |f(\pm1)|\}$, and 
\item $\Delta_f:=Orb_\cG(\gamma_f)$. 
\end{itemize} 
Then $S$ is an embedded Lagrangian sphere, and the 
conditions on $f$ ensure that  
\[ 
S\cap L=\{0\}\in \Delta \quad  \text{and} \quad S\pitchfork L.  
\] 
 Note that $0\neq [L]\in H_2(W,\Z)$. 
We can choose coordinates $(x_1,x_2)$ on $\Delta$ 
so that $dx_1\wedge dx_2$ is an area form. 

We can choose $f$ so that 
the projection of $S\cap U$ to the $(y_1,y_2)$-coordinate 
plane gives a coordinate system of $S\cap U$. One sees that $\{ -\pa_{x_1}, 
-\pa_{x_2}\}$ is the corresponding dual basis for $T^*_0S$. 

Recall $\tau^2_S$ the square of the positive generalized 
Dehn twist along $S$. We may assume that $\tau^2_S$ is supported in a cotangent 
neighborhood of $S$, and $\tau^2_S=id$ near $S$. 

Since on $U$ the sets $S\cap U$, $D\cap U$, 
$\Delta$ and the symplectic form $\omega|_U$ are invariant under the Hamiltonian 
$\cG$-action, we may assume 
that $\tau^2_S$ is also $S^1$-invariant when restricted to 
the intersection of a cotangent neighborhood of $S$ with $U$. Thus we can describe the effect of $\tau^2_S$ on $L$ 
by looking at the corresponding picture in the 
$(x_1,y_1)$-coordinate plane $E$. 

We may assume that $\text{Supp}(\tau^2_S)\cap L=\{ 0<\delta_1\leq x_1^2+x_2^2\leq \delta_2, y_1=0=y_2\}$. Here $\text{Supp}(\tau^2_S)$ denotes the support of $\tau^2_S$. 
Now for each $\theta\in[0,2\pi]$,  $\tau^2_S$ sends the oriented line segment 
$\ell_\theta(s)=(x_1=s\cos\theta, x_2=s\sin\theta, y_1=0,y_2=0)\subset\text{Supp}(\tau^2_S)\cap L $, 
$\delta_1\leq |s|\leq \delta_2$,  to a curve 
which projects to the simple geodesic circle in $S$ passing through the north pole 
$0\in \Delta$ and is oriented by the vector $\dfrac{d\ell}{ds}(\delta_1)$. Note that $\tau^2_S$ is independent of the orientation of $S$. 

With the $S_1$-symmetry associated with $\cG$ we can depict the $E$-slice of $R:=\tau^2_S(L)\cap U$ 
as the bold black curve in the right picture of Figure \ref{tau2S}. 
Let us denote by $R_E$ the $E$-slice of $R$. 

\begin{figure}[th] 
 \begin{center} 
 \includegraphics[scale=0.4]{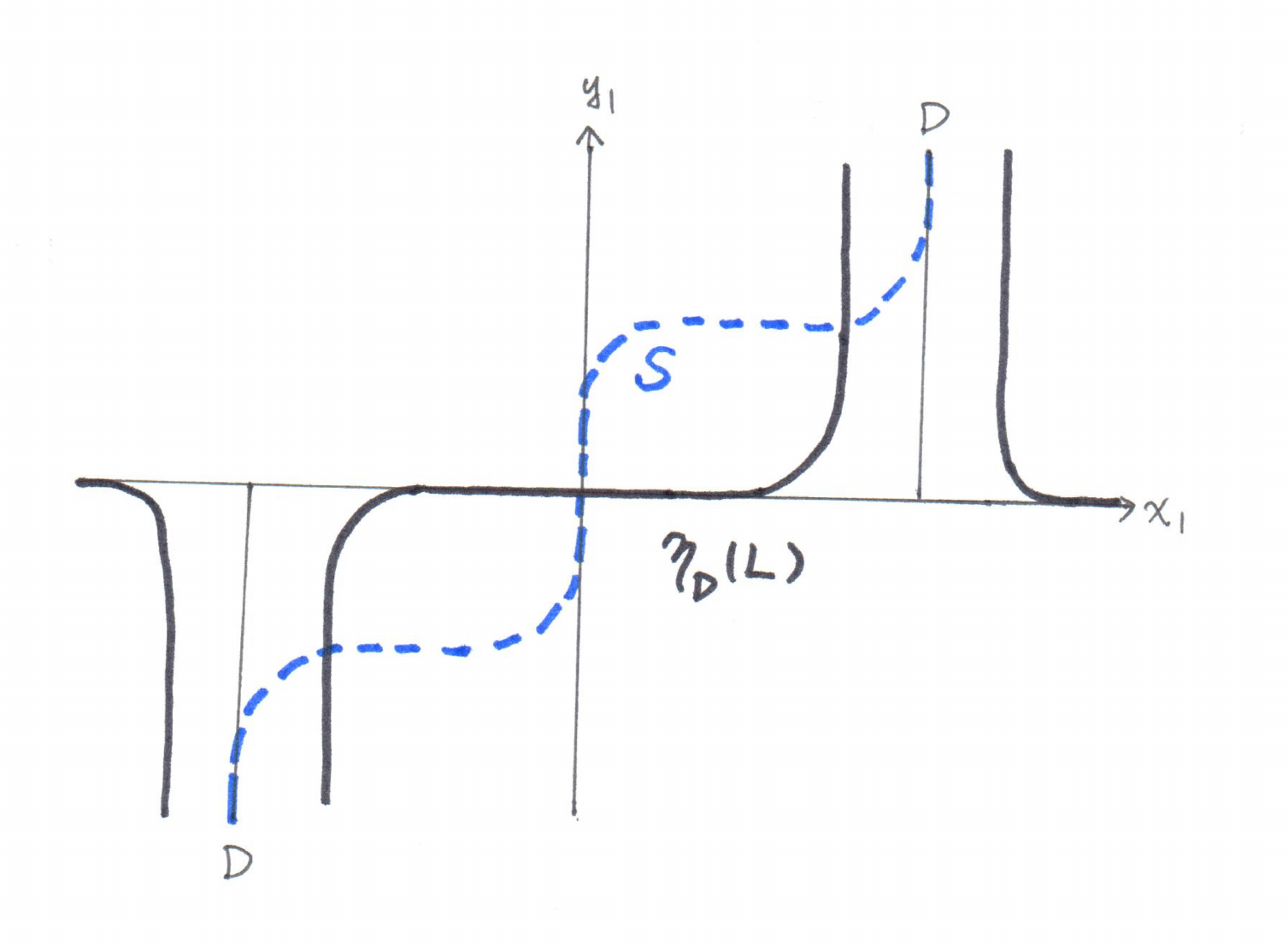}
\includegraphics[scale=0.4]{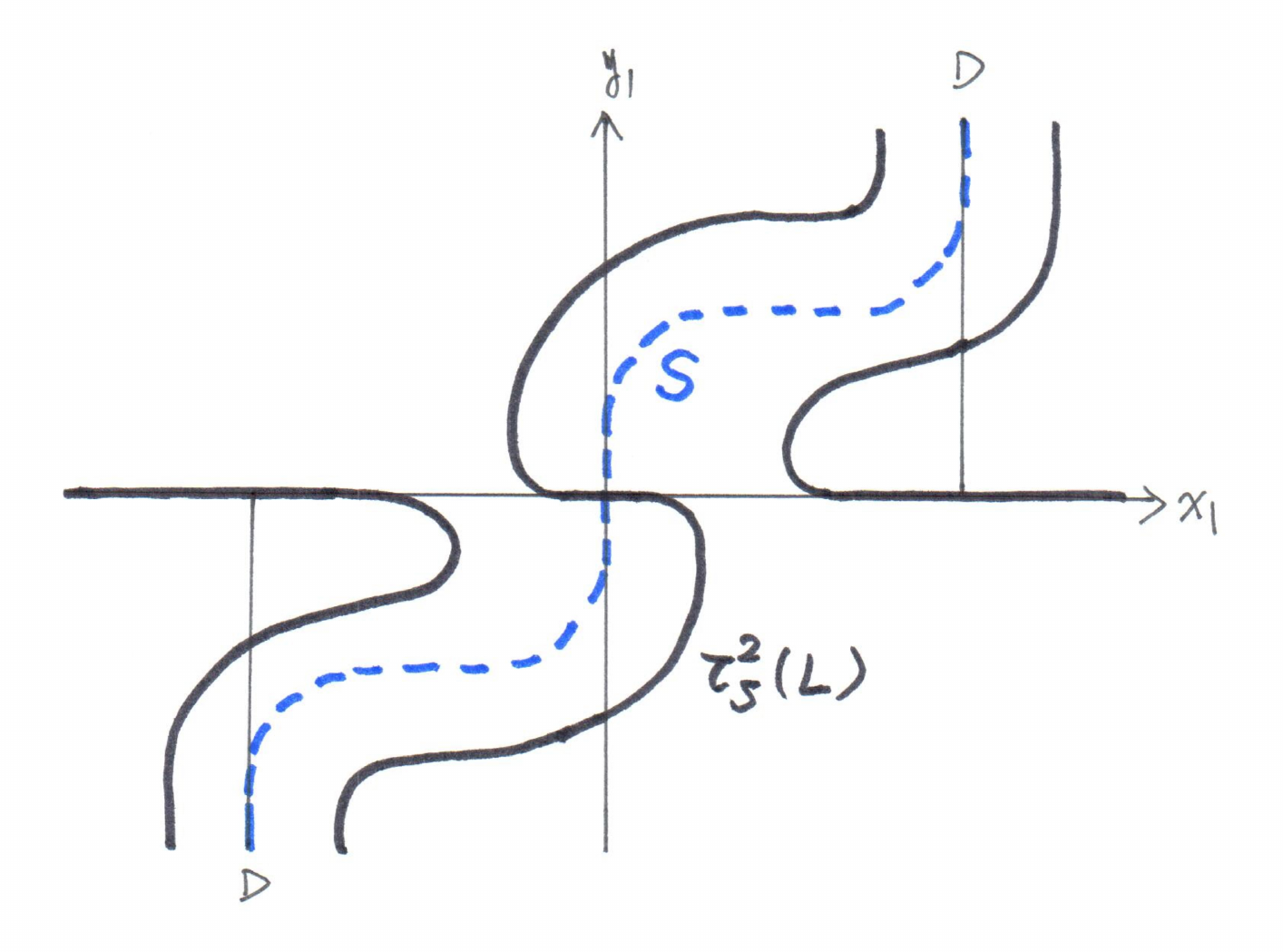} 
\caption{$\eta_D(L)$ and $\tau_S^2(L)$.  \label{tau2S}} 
\end{center} 
\end{figure}

Since $S=\Delta_f\cup D_1$ and we treat $0\in \Delta_f$ as the north pole of $S$, we can view $D_1$ as the southern hemisphere of $S$. Let $p_s\subset D_1$ denote 
the south pole of $S$. we can also parameterize $D$ so that $p_s$ is the origin 
of the 2-disk $D$. 
Note that the Lagrangian 
disk surgery $\eta_D$  replaces a collar neighborhood $U_C$ 
of $C=\pa D\subset L$ by a Lagrangian annulus symplectomorphic to 
the total space in $T^*D\cong TD$ of of oriented geodesics 
with a fixed constant speed 
in $D$ (with respect to the standard Euclidean metric) passing through 
$p_s\in D$. 
The annulus  $U_C\subset L$ can be symplectically identified with a neighborhood of the zero section of a 1-dimensional subbundle $V$ of $T^*_CL$ over $C$. Then $U_C$ has a fibration over $C$ with fiber $\ell_x\subset V_x$ over $x\in C$, and up to a 
Hamiltonian isotopy with $\partial U$ fixed, $\eta_D(\ell_x)\subset T^*_CD\cong T_CD$ is a curve that 
projects to an oriented geodesic line in $D$, passing through $p_s$, and with orientation 
uniquely specified by $x\in C$. In particular, $\ell_x$ and $\ell_{-x}$ correspond to the 
same geodesic line but with different orientations. Note that since $C\subset L$ is 
homologically trivial, up to Hamiltonian isotopy the resulting Lagrangian surface $\eta_D(L)$ 
does not depend on the precise surgery size of $\eta_D$. Thus we may assume that 
\begin{itemize}  
\item $\Gamma:=\eta_D(L)\cap U=Orb_\cG(\Gamma_E)$ is the orbit of $\Gamma_E$ 
(the bold black curve in the left picture of Figure \ref{tau2S}
) under the Hamiltonian $\cG$-action, 
$\Gamma_E$ is invariant under the $180^\circ$-rotation of $E=\pa_{x_1}\wedge\pa_{y_1}$ centered at 
$0\in \Delta$,   and $\pa \Gamma_E=\pa R_E$; 
\end{itemize}

   On $E$ there is a Hamiltonian isotopy $\phi_t$ supported in 
a compact set in $U\cap E$, $\phi_0=id$, $\phi_1(R)=\Gamma$, such that $\phi_t$ commutes 
with the said $180^\circ$-rotation and $\phi_t$ fixes $\partial \Gamma$ for all $t$. 
Observe that $\phi_t$ can be extended to a Hamiltonian isotopy on $W$, compactly 
supported in $U$, commuting with the Hamiltonian $S^1$-action (so that 
all $R_t:=\phi_t(R)$ are $S^1$-invariant), such that 
$\phi_1(R)=\Gamma$ and $\phi_t=id$ near $\pa R$ for all $t$. 
We obtain that $\eta_D(L)$ is Hamiltonian isotopic to 
$\tau^2_S(L)$. 
Hence up to Hamiltonian isotopy $\eta_D=\tau^2_S$ as surgeries on $L$.

{\bf Case 2: \  $\omega(\nu,\nu')<0$. } \  
The discussion goes almost parallel to Case 1, with several changes: 
 
 \begin{enumerate} 
 \item $D\cap U=A_-$. 
 \item For the construction of $S$ replace $f$ by $-f$.  
 \item $\tau^2_S$ is replaced by $\tau^{-2}_S$ the square of the {\em negative} 
 generalized Dehn twist along $S$. 
 
 \item The conclusion is that $\eta_D(L)$ and $\tau^{-2}_S(L)$ are Hamiltonian isotopic. 
 Hence $\eta_D|_L \cong \tau^{-2}_S|_L$.  
 \end{enumerate} 
 \end{proof}

\begin{rem} \label{L=S2} 
{\rm 
When $L=S^2$ is a Lagrangian sphere, the attaching circle 
$C=\pa D$ bounds two distinct 
disks $\Delta$ and $\hat{\Delta}$ in $L=\Delta\cup -\hat{\Delta}$. Let $S,\hat{S}$ be the Lagrangian sphere 
associated to $D\cup -\Delta$ and $D\cup \hat{\Delta}$ respectively. 
Then up to a choice of orientation, the homology classes of 
$S$ and $\hat{S}$ differ by the homology class of $L$. 
Nevertheless $\tau^2_{\hat{S}}(L)$ and $\tau^2_{S}(L)$ are 
Hamiltonian isotopic. 
} 
\end{rem}

 The polarity of an elliptic $la$-disk can be described in terms 
 of the $\pm$-sign of $\omega(\nu,\nu')$ as defined in 
 Proposition \ref{dehn}:

\begin{defn}[{\bf Polarity of an elliptic $la$-disk}] 
{\rm 
We say that an elliptic $la$-disk $D$ of $L$ is {\em positive} if it satisfies $\omega(\nu,\nu')>0$ as in Proposition \ref{dehn}(i), {\em negative} if 
it satisfies $\omega(\nu,\nu')<0$ instead. 
} 
\end{defn}

\begin{lem}[\bf Elliptic  pair] \label{Epair}  
Let $D$ be an elliptic $la$-disk of $L$ with 
$C:=\partial D$ bounds an embedded disk $\Delta\subset L$. 
Then there is another elliptic $la$-disk 
$D'$ of $L$ with $\pa D'=C$ such that 
\begin{enumerate} 
\item $D'$ and $D$ have opposite polarity; 
\item $\eta_{D'}(L)$ is Hamiltonian isotopic to 
$\tau^{-2}_S(L)$, where $\tau_S$ is the generalized 
double Dehn twist associated to $D$.  
\end{enumerate} 
\end{lem}  

\begin{proof} 
Assume first that $D$ is positive. 
The negative 
elliptic $la$-disk $D'$ with $\pa D'=C$ can the obtained from $D$ as indicated in Figure \ref{DD'}. 

\begin{figure}[th] 
 \begin{center} 
 \includegraphics[scale=0.4]{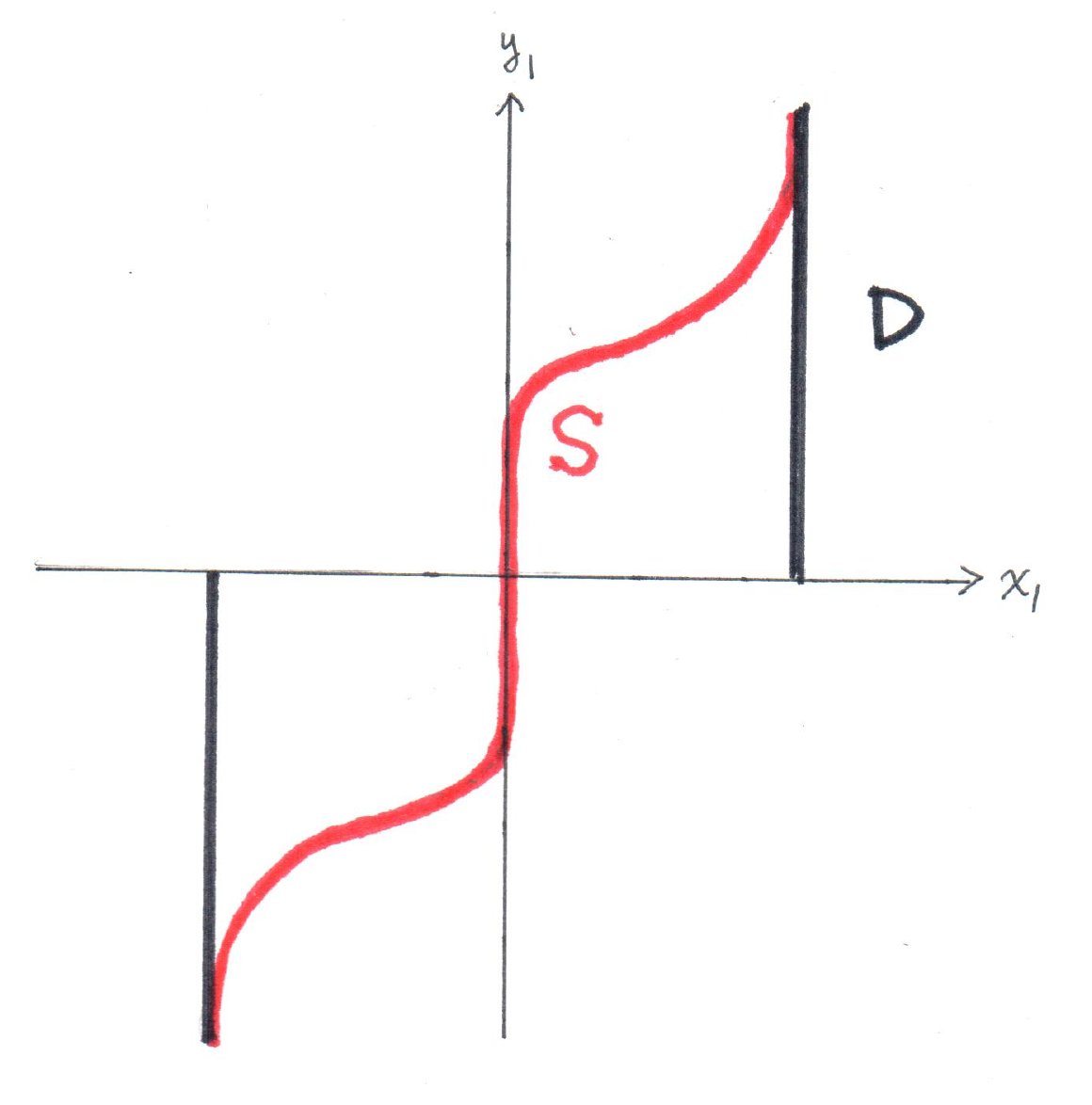}
\includegraphics[scale=0.4]{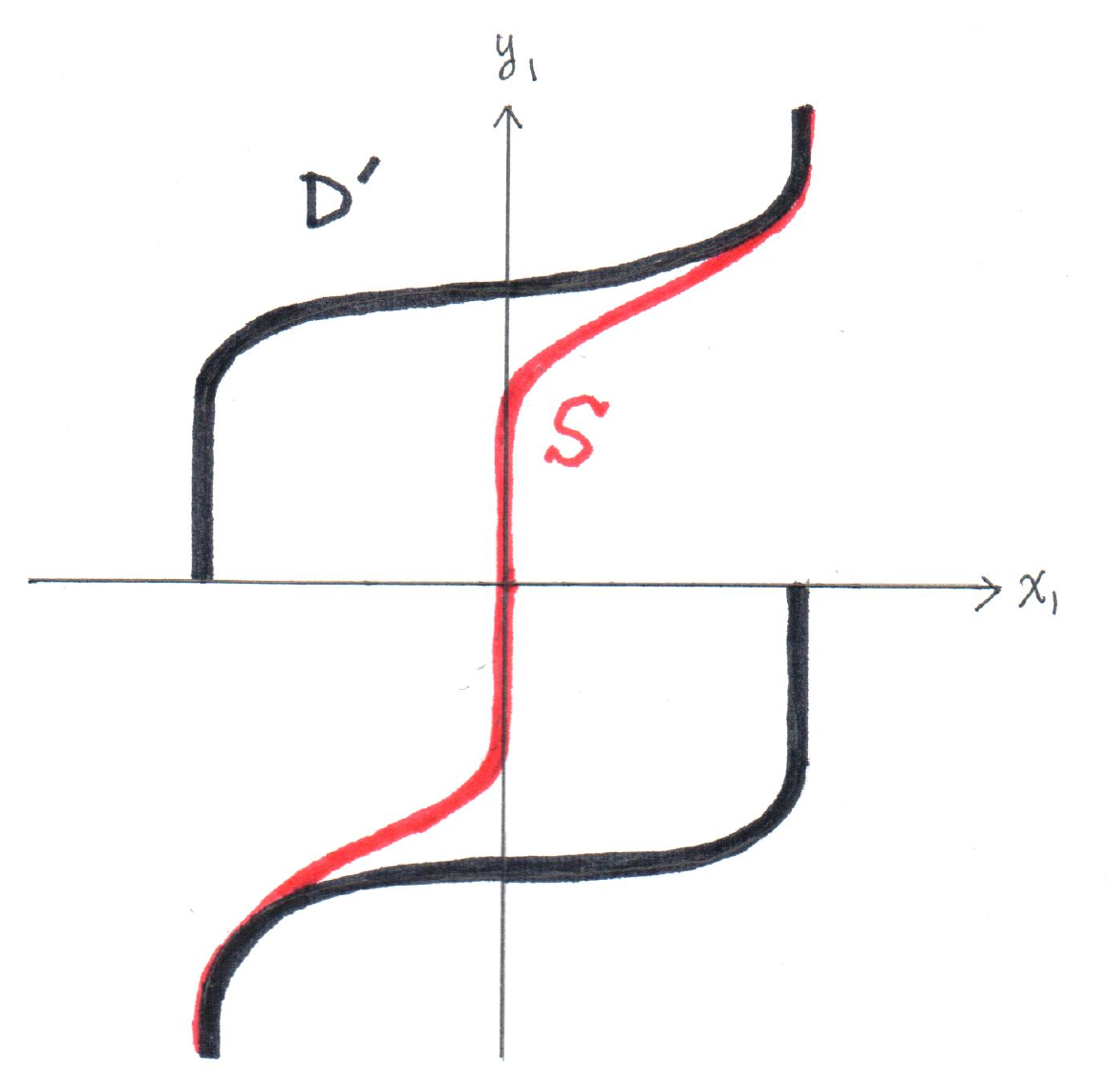} 
\caption{Positive $D$ and negative $D'$.  \label{DD'}} 
\end{center} 
\end{figure}

It is easy to see that the Lagrangian sphere  associated 
to $D'$ as depicted in Figure \ref{DD'} is Hamiltonian isotopic 
to the Lagrangian sphere $S$ associated to $D$.  That 
$\eta_{D'}(L)$ and $\tau^{-2}_S(L)$ being Hamiltonian isotopic 
follows from Proposition \ref{dehn}(ii). 

The negative $D$ case can be verified in a similar way. 
\end{proof}

\begin{cor} [\bf Infinite order] \label{inf-order} 
Let $D, D',L$ be as defined in Lemma \ref{Epair}. 
Assume that $D$ is positive and $D'$ is negative.  
Then for each $n\in\N$,  $\eta^n_D(L)$  is defined and is Hamiltonian isotopic to $\tau^{2n}_S(L)$.  
Similarly, $\eta_{D'}^n(L)$ is defined and is Hamiltonian isotopic 
to $\tau^{-2n}_S(L)$, $n\in \N$. 
\end{cor} 

\begin{proof} 
Let $L^1:=\eta_D(L)$. Up to Hamiltonian isotopy $L^1=\tau^2_S(L)$ as depicted in the right-hand picture of Figure \ref{tau2S}. There one observes that $\tau^2_S(L)$ contains a smaller disk $\Delta_1\subset \Delta$, and $C_1:=\pa \Delta_1$ 
is also the boundary of a positive elliptic $la$-disk $D_1$ of $\tau^2_S(L)$ such that $D_1=D$ outside $U$ (recall $\Delta$ and $U$ from the proof of Proposition \ref{dehn}). 
We can define $L^2:=\eta^2_D(L):=\eta_{D_1}\circ\eta_D(L)=
\eta_{D_1}(L^1)$. Since $\eta_{D_1}(L^1)$ is Hamiltonian isotopic to $\tau^2_S(L^1)$, we have that $L^2:=\eta^2_D(L)$ is Hamiltonian isotopic to $\tau^2_S(L^1)=\tau^4_S(L)$. 
Repeat  the process with  $L^1$ replaced by $L^2$ and so on  
so forth  we get infinitely many 
smoothly isotopic Lagrangian surfaces $L^n:=\eta_D^n(L)=
\eta_{D_{n-1}}(L^{n-1})$, $D_0=D$, $L^0=L$, $n\in \N$, 
and $L^n$ is Hamiltonian isotopic to $\tau^{2n}_S(L)$. 

The $\eta_{D'}$ case can be proved by exactly the same type 
of argument. This completes the proof. 
\end{proof} 

%
%
\subsection{Properties of $la$-disk surgery} \label{prop-surg}

Below we discuss the effect of $la$-disk surgery on the minimal 
Maslov number, the Liouville class, and the monotonicity of 
Lagrangian surfaces (see \cite{O1} for detailed definitions). For simplicity we assume that $c_1(W)=0$ 
and $\pi_1(W)=0$. We also assume that $L$ is orientable. Then the Maslov class $\mu_L$ of $L$ 
is an element of $H^1(L,\Z)$  and the Liouville class $\alpha_L$ 
is in $H^1(L,\R)$. If $\mu_L\neq 0$ then we define 
the minimal Maslov number of $L$ to be 
\[ 
m_L:=\min \{ \mu_L(\tau)\mid \tau\in H_1(L,\Z), \ \mu(\tau)>0\}, 
\] 
and we set $m_L=0$ if $\mu_L=0$. We say that $L$ is 
{\em monotone} if there is a constant $c>0$ such that 
\[ 
\alpha_L=c\cdot \mu_L. 
\] 

Let $L'=\eta_D(L)$ be obtained from $L$ by a $la$-disk surgery
 along a $la$-disk $D$ of $L$. 
If $L=S^2$ then both $\mu_L$ and $\alpha_L$ vanish, so are 
$\mu_{L'}$ and $\alpha_{L'}$. We may assume that the 
genus of $L$ is $g>0$. Then the intersection pairing 
\[ 
H_1(L;\Z)\times H_1(L;\Z)\to \Z , \quad (\s, \eta)\to \s\cdot\eta,  
\] 
is unimodular. For $\s\in H_1(L;\Z)$ we denote 
\[ 
\s^\perp:=\{ \zeta\in H_1(L;\Z)\mid \sigma \cdot \zeta =0\}. 
\]

\begin{prop} \label{L'monotone} 
Let $L,L'=\eta_D(L)$ be as in Proposition \ref{smooth-iso}. 
Assume in addition that $c_1(W)=0$ and $\pi_1(W)=0$. Then $L, L'$ have  the same 
minimal Maslov number.
If in addition that $L$ is orientable and monotone, then 
$L'$ is monotone iff the kernel  of the Liouville class 
$\alpha_L:H_1(L;\Z)\to \R$ satisfies  
\[ 
\ker(\alpha_L)\subset [C]^\perp.  
\] 
In particular, if $L$ is a monotone Lagrangian torus then so 
is $L'$. 
Moreover, if $[C]=0\in H_1(L,\Z)$, then $L'$ is monotone and 
$\alpha_{L'}$ is isomorphic to $\alpha_L$ as cohomology classes. 
\end{prop} 

\begin{proof} 
Let $C:=\pa D$. Let $U_C\subset L$ be a collar neighborhood 
of $C\subset L$ such that $L\setminus U_C\subset L'$. 
If $\sigma\in 
H_1(L,\Z)$ satisfies $\sigma\cdot[C]=0$ then $\sigma 
=[\beta]$ for some curve $\beta\subset L$ disjoint from $U_C$. 
 Hence $\beta\subset L'$ and $\mu_{L'}([\beta])=\mu_{L}([\beta])$. 
 Similarly we have $\alpha_{L'}([\beta])=\alpha_L([\beta])$. 

Now let $\sigma$ be a generator of the quotient group of 
$H_1(L,\Z)$ by $[C]^\perp$. We may 
assume that $\sigma\cdot[C]=1$. Represent $2\sigma$ by two disjoint embedded curves $\beta_\pm
\subset L$ such that each of $\beta_\pm$ intersects 
with $C$ transversally at a single point, and 
$\gamma_\pm:=\beta_\pm
\cap U_D$ is contained in the $x_1y_1$-coordinate plane 
so that $\gamma_+=\gamma$ as defined in (\ref{gamma}), $\gamma_-\subset \{ 
x_1<0\}$, and $\gamma_+\cup \gamma_-$ is invariant under 
the $180^\circ$-rotation of the $x_1y_1$-plane with $(0,0)
\in\R^2_{x_1,y_1}$ as the center point. 

Since  $\pi_1(W)=0$, there exist smooth maps  
$\psi_\pm:(D^2,\partial D^2)\to  (W,\beta_\pm)$ of disks 
with boundary circles mapped to $\beta_\pm$ respectively.  
Let $D_\pm:=\psi_\pm(D^2)$. We may assume that 
$B_\pm:=D_\pm\cap U_D$ is contained in the $x_1y_1$-plane 
such that $B_+\subset \{ x_1>0\}$, $B_-\subset \{ 
x_1<0\}$, and $B_+\cup B_-$ is also invariant under 
the $180^\circ$-rotation described above. 
We may also assume that 
$\psi_\pm$ is a smooth embedding when 
restricted to $V_\pm:=\psi^{-1}_\pm(B_\pm)$.

Recall the anti-symplectic rotation $M$ as defined in (\ref{M}). 
Let $\gamma'_\pm:=M(\gamma_\pm$). Let $\beta'$ denote the 
closure of the union of $(\beta_\pm\setminus \gamma_\pm)
\cup \gamma'_\pm$. Then $\beta'\subset L'$ is a simple 
closed curve representing a class $\sigma'\in H_1(L',\Z)$ 
satisfying $\sigma'\cdot [C']=2$. Let $Z\subset \R^2_{x_1y_1}$  
denote the closed region 
bounded by $\gamma_\pm$ and $\gamma'_\pm$. 
Let $D':=D_-\cup Z\cup D_+$, then $D'$ is a disk with $\pa D'=\beta'$. 

Pick a symplectic trivialization 
$\Psi_\pm$ of $\psi_\pm^*TW$ so that $\Psi_\pm|_{V_\pm}$ 
is the standard symplectic trivialization of $TW$ on 
$B_\pm$. It is easy to see that $D'$ is the image of 
a smooth map $\psi':(D^2,\partial D^2)\to (W,\gamma')$ 
such that $\psi'$ is an embedding when restricted to $V':=
\psi^{-1}(B)$ where $B=B_-\cup Z\cup B_+$. Choose 
a symplectic trivialization $\Psi'$ of $\psi'^*TW$  so that 
$\Psi'|_{V'}$ is just the standard symplectic trivialization of  $TW$ on $B$ and $\Psi'=\Psi_\pm$ when restricted to 
the preimages of $D\setminus B_\pm$ respectively. 
Now the comparison of $\mu_L(\beta_-\cup \beta_+)$ and $\mu_{L'}(\beta')$ is reduced to the calculation of 
the Maslov angles of $\gamma_-\cup\gamma_+$ and $\gamma'_-\cup
\gamma'_+$. An easy computation shows the two angles are 
equal. Thus we have 
\[ 
\mu_{L'}(\sigma')=\mu_{L}(2\sigma). 
\] 
Note that $\sigma$ is primitive and $\sigma'$ is twice of a 
primitive class. We conclude that $\mu_L$ and $\mu_{L'}$ 
have the same minimal Maslov number. 

Also we have 
\[ 
\alpha_{L'}(\sigma')=2\cdot \alpha_L(\sigma)+\int_Z\omega, 
\] 
here the orientation of $Z$ is determined by the orientation of its 
boundary $\pa Z=\gamma'_+\cup(-\gamma_-)\cup\gamma'_-\cup(-\gamma_+)$. Here $-\gamma_-$ denotes $\gamma_-$ but 
with its orientation reversed, and $\bar{\gamma}_+$ is defined 
in a similar way. 
So $\int_Z\omega>0$  if $\pa Z$ is oriented counterclockwise, 
$\int_Z\omega<0$ if otherwise. 

Thus if $L$ is monotone then $L'$ is monotone iff 
$\ker(\alpha_L)\subset [C]^\perp$. In particular, this condition is 
met when $[C]=0\in H_1(L,\Z)$. On the other hand, if $L$ is a 
torus (and monotone) then the condition  $\ker(\alpha_L)\subset [C]^\perp$ is automatically satisfied, even though $[C]\neq 0$. 
So the monotonicity of a Lagrangian torus is preserved under 
$la$-disk surgery. 
\end{proof}

%

%
%

%
%
\section{Lagrangian Grassmannian}

%
%
\subsection{Decomposition and group action}  \label{LG}

In this section we review Lagrangian Grassmannian of 
a $4$-dimensional symplectic vector space $V$. Since $V$ is 
linearly isomorphic  to the standard symplectic $4$-space 
$(\R^4,\omega=\sum_{j=1}^2dx_j\wedge y_j)$, we will identify 
$V$ with $(\R^4,\omega)$ without further notice. Denote by 
$\Lambda^+$  the space of all oriented Lagrangian planes in 
$V$.

{\bf Factorization of $\Lambda^+$.} \ 
Let $J:V\to V$ be a complex structure compatible with $\omega$, i.e., $J$ is a linear map with $J^2=-Id$, and 
the composition $\omega\circ (Id\times J):V\times V\to \R$ is 
positive definite and symmetric. 
The pair $(\omega, J)$ associate  a unique inner product 
$g:=\omega\circ (Id\times J)$ on $V$, and the structure group 
is reduced to the unitary group $U(2)$ associated to $(\omega, J)$. 
For notational convenience, 
we will often represent an  oriented 2-dimensional subspace 
$E\subset V$ as a $2$-vector $E=v_1\wedge v_2$ formed by an 
oriented basis $\{ v_1,v_2\}$ of $E$. We will also denote 
by $E^\perp$ the oriented  2-dimensional subspace 
in $V$ $g$-orthogonal to $E$ such that the orientation of $E\wedge E^\perp$ is that of their ambient  symplectic vector space  $V=\R^4$.

Fix a unitary basis $( u,v)$. There is a unique $g$-orthogonal 
complex structure $K'$ on $V$ such that $u\wedge v$ is 
$K'$-complex. i.e., $v=K'u$. Let $K'':=JK'$. We have 
$K'K''=J$, $K''J=K'$. The triple 
$(J,K',K'')$  generate 
an $S^2$-family of $g$-orthogonal complex structures 
\[ 
J_{a,b,c}:=aJ+bK'+cK'', \quad a^2+b^2+c^2=1. 
\] 
Let 
\[
K_t:=\cos tK'+\sin tK'', \quad t\in \R/2\pi\Z. 
\] 
Note that  
$K_t$-complex planes are Lagrangian planes. We have the 
following decomposition of $\Lambda^+$: 
\begin{equation} \label{Lambda+} 
\Lambda^+ =\bigsqcup _{t\in\R/2\pi\Z}\dP(K_t), 
\end{equation} 
where $\dP(K_t)\cong \C P^1\cong S^2$ is the Grassmannian 
of $K_t$-complex 2-dimensional subspaces of $V$. 
We also denote by $\dP(J)$ the Grassmannian of $J$-complex 
planes. We 
call (\ref{Lambda+}) a {\em $J$-decomposition of $\Lambda^+$}. 

\begin{rem} 
{\rm 
A different choice of a unitary basis $( u,v)$ amounts to 
changing the parameter $t$ in (\ref{Lambda+}) by 
adding a constant. In addition, 
the space of $\omega$-compatible 
complex structures on $V$ is contractible, hence any two 
$J$-decompositions of $\Lambda^+$ are homotopic. 
} 
\end{rem}

For example, if $J$ is the standard complex structure on 
$\C^2\cong \R^4$ defined 
by 
\[ 
J\pa_{x_j}=\pa_{y_j}, \quad j=1,2, 
\] 
then $g$ is the Euclidean metric 
on $\R^4$. Pick the unitary basis 
\[ 
u:=\pa_{x_1}, \quad v=\pa_{x_2}, 
\] 
then $K', K''$ are  defined by 
\begin{eqnarray}
K'\pa_{x_1}  = \pa_{x_2},  & \quad  K'\pa_{y_1} = -\pa_{y_2} 
\label{K'}\\ 
K''\pa_{x_1} = \pa_{y_2},  & \quad  K''\pa_{y_1} = -\pa_{x_2}.  
\label{K''} \end{eqnarray}

Although the choice of $u,v$ play no big role in the decomposition of $\Lambda^+$ as in (\ref{Lambda+}), the 
oriented Lagrangian plane $u\wedge v$ associates a unique 
oriented $S^1$-subgroup of $SU(2)$ as well as a unique 
oriented $S^1$-family of $J$-complex planes, leading to  
a parameterization of $\Lambda^+$ which will be described below. First review some basic facts about the $U(2)$-action 
on $\Lambda^+$.

{\bf Action of $U(2)$ on $\Lambda^+$.} \ 
The unitary group $U(2)$ associated to $(\omega, J)$ acts on 
$\Lambda^+$. Let $\cC\subset U(2)$ denote the subgroup 
of centralizers of $U(2)$. With respect to any 
unitary basis (e.g., $\{ u,v\}$), the matrix representative of $\cC$ is 
\begin{equation} \label{C} 
\cC  =\Big\{ c_t=\begin{pmatrix} e^{it} & 0\\ 0 & e^{it}\end{pmatrix}\mid t\in\R/2\pi\Z\Big\}.  
\end{equation} 
$\cC$ acts on $\Lambda^+$ by rotations: 
\begin{equation} \label{ctau} 
c_\tau (\dP (K_t))=\dP (K_{t+2\tau}), 
\end{equation} 
with $c_\pi=-Id$ acts trivially on 
$\Lambda^+$. Note that $\cC$ acts trivially on $\dP (J)$, $c_t$ rotates the total space of each element $E\in\dP(J)$ by an angle of $t$-radians with respect to the orientation of $E$. 

On the other hand, the special unitary subgroup 
$SU(2)\subset U(2)$ acts on each of $\dP(K_t)$ as well as on 
$\dP(J)$ by rotations, with its centralizer subgroup $\{ \pm Id\}$ 
acts as the isotropy subgroup of the action. Indeed, the action 
of $SU(2)/\{ \pm Id\}$ on $\dP(J)$ and on each of $\dP(K_t)$ 
can be identified with the canonical action of $SO(3)$ on the 
unit 2-sphere $S^2\subset \R^3$. Also $SU(2)$ commutes 
with all $J_{a,b,c}$, and in particular $K_t$ for $t\in\R/2\pi\Z$.

As a homogeneous space 
$\dP(K_t)\cong SU(2)/S^1\cong S^2$ is endowed with an 
$SU(2)$-equivariant metric unique up to scaling. 
We take the one with which the area of $\dP(K')$ is $\pi$, then 
$\dP(K_t)$ is a standard sphere with diameter 1. We also  
endow $\dP(J_{a,b,c})$ with the same kind of metric for 
$(a,b,c)\in S^2$. 

An $S^1$-subgroup $\cH\subset SU(2)$ acts on each 
of $\dP(J_{a,b,c})$ by standard rotations. There is a unique 
$\cH$-orbit in $\dP(J_{a,b,c})$ which is a great circle with 
respect the $SU(2)$-equivariant metric. We call this special 
orbit the {\em geodesic $\cH$-orbit in $\dP(J_{a,b,c})$}.

%
\subsection{Intersection of complex and Lagrangian planes} \label{intCL}

\begin{defn}[{\bf Complex locus}] 
{\rm 
Let $Z\subset V\cong \R^4$ be an oriented two dimensional 
subspace. The {\em complex locus of $Z$}  is defined 
to be 
\[ 
\cE_Z:=\{ E\in \dP(J)\mid E\cap Z\neq \{ 0\}\}. 
\] 
} 
\end{defn} 

\begin{prop} 
Let $(u,v)$ be a positive orthonormal basis of $Z$. Then 
\[ 
\cE_Z=\{ E_\h:=u_\h\wedge Ju_\h\mid \h\in \R/\pi\Z\}, 
\] 
where $u_\h:=u\cos\h+v\sin \h$. 
So $\cE_Z$ is a (possibly degenerated) circle in $\dP(J)$. 
And $\cE_Z$ consists of a single point iff $Z\in \dP(J)$ or 
$Z\in\dP(-J)$. 
\end{prop} 

It is easy to see that $\cE_Z=\{ Z\}$ if $Z$ is $J$-complex, and 
$\cE_Z=\{ -Z\}$ if $Z$ is $(-J)$-complex. 

In general $Z$ is $J_{a,b,c}$-complex for some unique 
$(a,b,c)\in S^2$. $SU(2)$ acts on the $J_{a,b,c}$-complex 
Grassmannian $\dP(J_{a,b,c})\cong S^2$ by rotations with 
$\{ \pm Id\}$ as the isotropy subgroup. 
There is a unique $S^1$-subgroup $\cH\subset SU(2)$ fixing 
the pair $Z,Z^\perp\in \dP(J_{a,b,c})$. 
$\cH$ acts as rotations on the total 
spaces of $Z$ and $Z^\perp$ respectively. We can orient 
$\cH=\{ h_s\mid s\in \R/2\pi\Z\}$ with $h_0=Id$ 
so that $h_s$ rotates $Z$ by an angle of $s$-radians, whilst 
it rotates $Z^\perp$ by an angle of $(-s)$-radians, with respect 
to the orientation associated with the complex structure 
$J_{a,b,c}$. By definition, $\cH$ also acts on $\cE_Z$ by 
rotations. Combining with the proposition above we have the 
following lemma.  

\begin{lem} 
The complex locus $\cE_Z$ is a connected  
$\cH$-orbit and hence a latitude  in $\dP(J)$ 
with respect to the fixed points of $\cH$ in $\dP(J)$, 
where $\cH\subset  SU(2)$ is the stabilizer subgroup of $Z$. 
\end{lem}

\begin{prop} \label{ZZ} 
The two complex loci $\cE_Z$ and $\cE_{Z^\perp}$ are 
either disjoint or equal. Moreover, 
$\cE_Z=\cE_{Z^\perp}$ iff $Z$ is Lagrangian. 
\end{prop} 

\begin{proof} 
Observe that $Z$ and $Z^\perp$ have the same stabilizer 
subgroup $\cH\subset SU(2)$, hence both $\cE_Z$ and 
$\cE_{Z^\perp}$ are connected $\cH$-orbit in $\dP(J)$. 
So $\cE_Z\cap \cE_{Z^\perp}=\emptyset$ if $\cE_Z\neq \cE_{Z^\perp}$. 

To verify the second statement we 
may assume that $Z$ is 
 neither $J$-complex nor $(-J)$-complex 
without loss of generality. 
Write $Z=u\wedge J_{a,b,c}u$ with $u$ unitary and $a\neq 0$, 
Then 
\[ 
\cE_Z=\{ E_\h:=(u\cos \h+J_{a,b,c}u\sin \h)\wedge 
J(u\cos \h+J_{a,b,c}u\sin \h) \mid \h\in\R/\Z\}. 
\] 
Assume that $\cE_Z=\cE_{Z^\perp}$. Then 
for any $\h\in\R/\pi\Z$, we have 
 $\dim E_\h\cap Z=1=\dim E_\h
\cap Z^\perp$, and $(E_\h\cap Z)\perp Z^\perp$, $(E_\h\cap 
Z^\perp)\perp Z$ since $Z\perp Z^\perp$. This implies that 
$E_\h\cap Z^\perp$ is spanned by $J(u\cos\h+J_{a,b,c}u\sin\h)$. Same conclusion holds if we replace $\h$ by $\h^\perp:=\h+\frac{\pi}{2}$. Thus we must have  $Z^\perp =-JZ$, hence $Z$ is Lagrangian. 

Conversely, assume that $Z$ is Lagrangian, then $J_{a,b,c}=K_t$ for some $t\in \R/2\pi\Z$. Since the centralizer subgroup 
$\cC\subset U(2)$ commutes with $SU(2)$ and $c_{\frac{-t}{2}}(Z)\in \dP(K')$, we may assume that $Z\in \dP(K')$ 
without loss of generality. 
Then we can write $Z=u\wedge K'u$, and $Z^\perp=K''u\wedge Ju=JK'u\wedge Ju$. That $\cE_Z=\cE_{Z^\perp}$ 
can be easily verified by direct computation. This completes 
the proof. 
\end{proof} 

\begin{prop} \label{georbit} 
The complex locus $\cE_Z$ is a great circle in $\dP(J)$ iff 
$Z$ is Lagrangian. 
\end{prop} 

\begin{proof} 
Let $\cH\subset SU(2)$ be the stabilizer subgroup of $Z$, 
and $\cN(\cH)$  the group of normalizers of $\cH$ in 
$SU(2)$. It is well known that  $\cN(\cH)/\cH\cong \Z_2$. 
Let $\rho\in \cN(\cH)$ be an element representing the 
nontrivial element of $\cN(\cH)/\cH$. Then $\rho^2=Id$, 
$\rho h_s\rho^{-1}=h_{-s}=h^{-1}_s$ for any $h_s\in \cH$. So $\rho$ preserves the orbit space of $\cH$ but reverses 
the orientation of $\cH$. Since $\rho$ maps the stabilizer subgroup 
$\cH_Z=\cH$ of $Z$ to the stabilizer subgroup $\cH_{Z^\perp}=\cH^{-1}$ of $Z^\perp$ we have $\rho(\cE_Z)=\cE_{Z^\perp}$.

By a direct computation one can 
show that $\rho$ acts on $\dP(J)$ as a $180^\circ$ rotation 
with respect to a pair of antipodal points lying on the 
geodesic $\cH$-orbit which is a great circle in $\dP(J)$. 
In particular, the geodesic $\cH$-orbit is the unique $\cH$-orbit 
preserved by $\rho$. 
So $\cE_Z$ is a great circle in $\dP(J)$ iff $\rho(\cE_Z)=\cE_Z$. Since $\rho(\cE_Z)=\cE_{Z^\perp}$ we conclude that 
$\cE_Z$ is a great circle in $\dP(J)$ iff $Z$ is Lagrangian 
by Proposition \ref{ZZ}. 
\end{proof}

\begin{cor} 
Given $Z,Z'\in \Lambda^+$, then 
\[ 
\cE_Z\cap \cE_{Z'}=
\begin{cases} 
\cE_Z  & \text{ if $Z'\in Orb_\cC(\{ Z,Z^\perp\})\subset \Lambda^+$},  \\ 
\text{ 2 points} & \text{ else}. 
\end{cases} 
\] 
\end{cor}

\begin{defn}[{\bf Lagrangian locus}]
{\rm 
For every $J$-complex plane $E\subset \dP(J)$ we define 
the {\em Lagrangian locus} $\Lambda^+_{E,t}$ of $E$ in $\dP(K_t)$ to be the 
set of all $K_t$-complex complex planes which intersects 
nontrivially with $E$, i.e., 
\[ 
\Lambda^+_{E,t}:=\{ F\in \dP(K_t)\mid \dim F\cap E=1\}. 
\] 
We also define the {\em total Lagrangian locus} $\Lambda^+_E$ 
of $E$ to be 
\[ 
\Lambda^+_E:=\cup_t\Lambda^+_{E,t}. 
\] 
} 
\end{defn}

\begin{prop}  \label{LambdaE} 
\begin{enumerate} 
\item 
Let $E\in\dP(J)$, then for each $t$, $\Lambda^+_{E,t}\cong S^1$ is a 
great circle in $\dP(K_t)$, and $\cC$ acts on $\Lambda^+_E$: 
\[ 
c_{\frac{t'}{2}}(\Lambda^+_{E,t})=\Lambda^+_{E,t+t'}. 
\] 
\item 
For $E,E'\in\dP(J)$ with $E\neq E'$, 
\[ 
\Lambda^+_{E,t}\cap \Lambda^+_{E',t}=
\begin{cases} 
\text{ 2 points} & \text{ if $E'\neq E^\perp$} \\ 
\Lambda^+_{E,t}=\Lambda^+_{E',t} & \text{ if $E'=E^\perp$}. 
\end{cases} 
\] 
\end{enumerate} 
\end{prop} 

\begin{proof} 
For $t\in \R/2\pi\Z$ 
 let $\omega_t:=g\circ (K_t\oplus Id)$ denote  the nondegenerate  anti-symmetric bilinear form (i.e. a linear symplectic form) on 
 $V=\R^4$ with respect to the complex structure $K_t$. 
 Elements of $\dP(J)$ are $\omega_t$-Lagrangian for all $t$. 
 Then (i) and (ii) follow easily from Propositions \ref{ZZ} and 
 \ref{georbit}, and the action of $\cC$ on $\Lambda^+$ as discussed in Section \ref{LG}. 
 \end{proof}

\begin{lem}
Given $E_0\in \dP(J)$ and an $S^1$-subgroup $\cG$ of $SU(2)$ we denote by $\cE:=Orb_\cG(E_0)$ 
the $\cG$-orbit of $E_0$ in $\dP(J)$. Then for any $t\in\R/2\pi\Z$,  
\[ 
\dP(K_t)=\underset{E\in \cE}{\bigcup} \Lambda^+_{E,t} 
\quad \Longleftrightarrow \quad \text{ $\cE$ is a great circle}.  
\] 
\end{lem} 

\begin{proof} 
For any $Z\in \dP(K_t)$ we have 
\[ 
\begin{split} 
 Z\in \underset{E\in \cE}{\bigcup} \Lambda^+_{E,t} \  
& \Longleftrightarrow  \text{ $\dim (Z\cap E)=1$ for some 
$E\in \cE$}  \\ 
& \Longleftrightarrow  \text{ $E\in \cE_Z$ for some $E\in \cE$}. 
\end{split} 
\] 
The map 
\[ 
\dP(K_t)\to \{ \text{ great circles in $\dP(J)$} \}
\] 
 by sending $Z\in\dP(K_t)$ to $\cE_Z$ is a 2:1 surjective map. 
 If $\cE$ is a great circle, then it will intersects with 
 every great circle in $\dP(J)$ at least twice, which implies 
 that $\dP(K_t)=\underset{E\in\cE}{\bigcup} \Lambda^+_{E,t}$. On the other hand, if 
 $\cE$ is not a great circle, then it is contained in some open hemisphere $D$ of  $\dP(J)$ and hence misses 
 at least one (in fact, infinitely may)  great circle: the boundary 
 $\pa D$ of $D$. Since $\pa D=\cE_Z$ for some $Z\in\dP(K_t)$
 we conclude that $\dP(K_t)\supsetneqq\underset{E\in\cE}{\bigcup} \Lambda^+_{E,t}$. This completes the proof. 
\end{proof}

Applying Proposition \ref{georbit} we have the following result.

 \begin{cor} 
For any $Z\in \dP(K_t)$, we have 
\[ 
\dP(K_t)=\underset{E\in \cE_Z}{\bigcup}\Lambda^+_{E,t}, \quad \text{and} 
\] 
\[ 
\Lambda^+=\underset{t\in\R/2\pi\Z}{\bigcup}(\underset{E\in \cE_Z}{\cup}\Lambda^+_{E,t}). 
\] 
\end{cor}

\begin{rem} 
{\rm 
For $E\in \dP(J)$ and $a\neq 0$ one can also define the 
{\em intersection locus} of $E$ in $\dP(J_{a,b,c})$ to be 
$\dP(J_{a,b,c})_E:=\{ Z\in \dP(J_{a,b,c})\mid \dim Z\cap E >0\}$. Then $\dP(J_{a,b,c})_E$ is a connected orbit of the 
stabilizer subgroup $\cH_E\subset SU(2)$ of $E$, but it is not a great circle in  in $\dP(J_{a,b,c})$. This follows from the 
observation that $E$ is not Lagrangian with respect to the 
symplectic form $\omega_{a,b,c}:=g\circ(J_{a,b,c}\oplus Id )$ 
associated to $J_{a,b,c}$, provided that $a\neq 0$. Likewise, 
in contrast to the Lagrangian case, for any $S^1$-subgroup 
$\cH\subset SU(2)$ and any $E_0\in \dP(J)$, the union of the 
intersection loci $\cup_{E\in Orb_\cH(E_0)}\dP(J_{a,b,c})_E$ 
will never cover $\dP(J_{a,b,c})$ if $a\neq 0$, even if 
$Orb_\cH(E_0)$ is a great circle in $\dP(J)$. 
This covering property of  intersection loci set Lagrangian 
planes apart from the totally real ones, and will enable us 
to device a new invariant for Lagrangian surfaces with 
stronger rigidity than some of the classical ones.  
} 
\end{rem}

%
%
\subsection{Spherical coordinates adapted to $( u,v)$} \label{sph}

Fix a unitary basis $( u,v)$ and let $Z:=u\wedge v\in \dP(K')$. 
Let $\cG\subset SU(2)$ denote the stabilizer subgroup of $Z$.
The matrix representation of $\cG$ 
with respect to the unitary basis $u,v$ is 
\[ 
\cG  =\Big\{ g_\h=\begin{pmatrix} \cos\h & -\sin\h\\ \sin\h & \cos\h\end{pmatrix}\mid \h\in\R/2\pi\Z\Big\}.  
\]
Here the parameter $\h$ is chosen so that $\eg_\h$ rotates 
the $K'$-complex plane $Z$ by an angle of $\h$-radians, with 
respect to the $K'$-complex orientation of $Z$. Simultaneously 
$g_\h$ rotates $Z^\perp$ by an angle of $(-\h)$-radians, also 
with respect to the $K'$-complex orientation of $Z^\perp$. 

Recall the complex locus $\cE_Z$. We parameterize 
$\cE_Z$ by $\h\in \R/\pi\Z$: 
\[ 
\cE_Z=\{ E_\h\in \dP(J)\mid  E_\h=u_\h\wedge Ju_\h\}, 
\]  
where $u_\h:=u\cos\h+v\sin\h$.

For each $\h\in\R/\pi\Z$, we denote by $\lambda_\h$ 
the Lagrangian locus $\Lambda^+_{E_\h,0}$ of $E_\h$ in 
$\dP(K')=\dP(K_0)$: 
\[
\lambda_\h:=\{ \xi\in\dP(K')\mid \dim (\xi\cap E_\h)=1\}. 
\] 
Each of $\lambda_\h$ is a great circle in 
$\dP(K')$ passing through $Z$ and $Z^\perp$. Note that 
(recall $\h^\perp=\h+\frac{\pi}{2}$) 
\[ 
\lambda_{\h^\perp}=\lambda_\h, \quad \h\in\R/\pi\Z. 
\] 
For each $\h$, we choose $u_\h$ as the basis for $E_\h$. 
Denote $\xi\in\dP(K')$ as 
\[ 
\xi=\lambda_\h(s) \quad \h,s\in\R/\pi\Z 
\] 
\[ 
\text{if }\  \xi\in\lambda_\h \ \  \text{and} \ \  \xi\cap E_\h=\text{Span}\{ u_\h\cos s +Ju_\h\sin s\} . 
\] 
We have 
\begin{enumerate} 
\item $\lambda_{\h^\perp}(s)=\lambda_\h(\pi-s)=\lambda_\h(-s)=(\lambda_\h(\frac{\pi}{2}-s))^\perp$ 
for $\h,s\in\R/\pi\Z$, and 
\item $Z=\lambda_\h(0)$, $Z^\perp=\lambda_\h(\frac{\pi}{2})$ 
for $\h\in \R/\pi\Z$. 
\end{enumerate} 
The parameter $s$ not only corresponds to the angle 
of rotation of the intersection subspace 
$\text{Span}\{ u_{\h,s}\}$ in $E_\h$, but also parameterize the orbit 
space of $\cG$ in $\dP(K')$ if $s$ is restricted to either $[0,\frac{\pi}{2}]$ or $[\frac{\pi}{2},\pi]$. Then 
\[ 
\Phi_+:\R/\pi\Z\times [0,\pi/2]\to \dP(K'), \quad \Phi_+(\h,s):=\lambda_\h(s)
\] 
 is a  
modified spherical coordinate system for $\dP(K')\cong S^2$.  
To compare $\Phi_+$ with the homogeneous coordinates of 
$\dP(K')$ we may take $u,v$ to be 
\[ 
u=\pa_{x_1}, \quad v=\pa_{x_2} 
\] 
without loss of generality. 
Let $x=x_1+ix_2$ and $y=y_1-iy_2$ be the $K'$-complex coordinates. Then $\dP(K')$ is parameterized by the homogeneous coordinates $[x:y]$: 
\begin{gather*}
\dP(K')=\{ [x:y]\mid (x,y)\in \C^2\setminus \{ (0,0)\}\}, \\   
[x:y]=[x':y']   \Leftrightarrow  (x',y')=(\lambda x,\lambda y) \text{ for some $\lambda\in \C^*$}. 
\end{gather*} 
In particular the points 
\[ 
\xi_0:=[1:0] \ \text{ and } \ \xi_\infty:=[0:1]
\] 
 represent $Z=\pa_{x_1}\wedge\pa_{x_2}$  
and $Z^\perp=\pa_{y_1}\wedge (-\pa_{y_2})$ respectively. 
The tangent plane $T_\xi\dP(K')$ can 
be identified with the total space of $\xi^\perp$. 
A direct computation yields the identity 
\[ 
\lambda_\h(s)=[e^{i\theta}\cos s: e^{-i\theta}\sin s], \quad 
\h,s\in\R/\pi\Z. 
\] 
Moreover, 
the induced  orientation on $\dP(K')$ via $\Phi_+$ coincides with the 
orientation on $\dP(K')$ inherited from its $K'$-complex 
structure.

 Note that the other parameterization 
 \begin{gather*} 
\Phi_-:\R/\pi\Z\times [\frac{\pi}{2},\pi]\to \dP(K'), \\ 
\Phi_-(\h,s):=\lambda_\h(s), 
\end{gather*} 
induces the opposite orientation on $\dP(K')$.

The action of $c_t\in\cC$ on $\Lambda^+$ carries the spherical 
trivialization $\Phi_+$ on $\dP(K')=\dP(K_0)$ over $\dP(K_{2t})$. 
Observe  that $\cC$ commutes with $\cG$. In fact, $\cC$ and 
$\cG$ together generate a maximal torus $\cT_\cG$ of $U(2)$. 
Accordingly we obtain an $\cT_\cG$-equivariant trivialization 
\begin{gather} 
\tilde{\Phi}:\R/2\pi\Z\times \R/\pi\Z\times [0,\frac{\pi}{2}]\to 
\Lambda^+  \nonumber \\ 
\tilde{\Phi}(t,\h,s):=c_{\frac{t}{2}}(\lambda_\h(s)). \label{triv}
\end{gather} 
With $\tilde{\Phi}$ we can identify $\Lambda^+$ with 
$S^1\times S^2$, so that $\{ t\}\times S^2\cong \dP(K_t)$, 
and the projection $S^1\times S^2\to S^2$ corresponds to 
the {\em central projection}  
\begin{align} 
\pi' & :\Lambda^+\to \dP(K'), \nonumber \\ 
\pi'(\xi) & := c_{-\frac{t}{2}}(\xi) \quad 
\text{ if $\xi\in\dP(K_t)$}. \label{pi'} 
\end{align}

\begin{rem} \label{theta} 
{\rm 
The parameter 
$\h\in \R/\pi\Z$ , which parameterizes complex planes 
$E_\h$, 
increases clockwise around $\xi_0$ and counterclockwise 
around $\xi_\infty$. 
} 
\end{rem}


{\bf Orientation of $\lambda_\h$.}  \ 
Observe that $\lambda_\h=\lambda_{\h^\perp}$ comes 
with two different orientations. Denote by $\lambda_\h^+$ 
as $\lambda_\h$ with the orientation induced $E_\h$, 
and by $\lambda_\h^-$ as $\lambda_\h$ but with the orientation 
induced by $-E_\h$. We have 
\[ 
\lambda_\h^+=\lambda_{\h^\perp}^- , \quad 
\lambda_\h^-=\lambda_{\h^\perp}^+. 
\] 

{\bf Relative $E_\h$-phase.} \ 
As we trace out $\lambda^+_\h$ once, the intersection 
subspace $\lambda^+(s)\cap E_\h$ rotates in $E_\h$ by an 
angle of $\pi$-radians, whilst $\lambda^+_\h\cap E_{\h^\perp}$ 
rotates in $E_{\h^\perp}$ by an angle of $(-\pi)$-radians, 
we call the former minus the latter, denoted as 
$(\Delta\varphi)_{\lambda^+_\h}$, the {\em relative phase} 
of $E_\h$ along 
$\lambda^+_\h$, which is 
\begin{equation} \label{Delta_lambda} 
(\Delta\varphi )_{\lambda^+_\h}=2\pi. 
\end{equation}

An alternative description of the orientations of $\lambda_\h$ 
is in order: The complement $\dP(K')\setminus\lambda_\h$ 
consists of two disjoint open disks: 
\[ 
\dP(K') \setminus \lambda_\h=D_\h \sqcup D_{\h^\perp}, 
\] 
where 
\begin{align} 
D_\h & :=\{ \xi\in \dP(K')\mid \xi\pitchfork E_\h, \ \xi\wedge E_\h>0\}  \label{Dh} \\  \nonumber 
 & = \{ \lambda_{\h'}(s)\in \dP(K')\mid  \h-\frac{\pi}{2}<\h'<\h \mod \pi, \ 0<s<\frac{\pi}{2}\}, \\  
D_{\h^\perp} & :=\{ \xi\in \dP(K')\mid \xi\pitchfork E_{\h^\perp}, \ \xi\wedge E_{\h^\perp}>0\}     \label{Dhp} \\ 
 & =\{ \lambda_{\h'}(s)\in \dP(K')\mid  \h<\h'<\h^\perp \mod \pi, \ 0<s<\frac{\pi}{2}\}. \nonumber
\end{align}

Note that for $\xi\in\dP(K')$ 
\[ 
\xi\wedge E_\h>0 \ \Longleftrightarrow \  \xi\wedge E_{\h^\perp}<0. 
\] 
Then 
\[ 
\lambda_\h^+=\lambda_{\h^\perp}^-=\pa D_\h, \quad \lambda_{\h^\perp}^+=\lambda_\h^-=\pa D_{\h^\perp}
\] 
as oriented boundaries. 
Similar conclusions hold straightforwardly for complex loci $\lambda^t_\h:=\Lambda^+_{E_\h,t}\in \dP(K_t)$ by applying 
the rotation $c_{\frac{t}{2}}$.

\subsection{Maps into $\dP(K')$} \label{PK'}

Recall that $\xi_0:=u\wedge v\in \dP(K')$ denotes the oriented 
Lagrangian plane which corresponds to the south pole 
$[1:0]$ of 
$\dP(K')=\{ [e^{i\h}\cos s:e^{-i\h}\sin s]\mid \h\in\R/\pi\Z, \ 
s\in [0,\frac{\pi}{2}]\}$.

A suitable neighborhood $U_{\xi_0}$ of $\xi_0$ in $\Lambda^+$ can be identified with the space of symmetric $2\times 2$ 
matrices 
\[ 
S=\begin{pmatrix} a+c & b \\ b & -a+c \end{pmatrix}, \quad  
a,b,c\in \R, 
\] 
so that with respect to the orthonormal basis $\{ u, v,Ju, Jv\}$ the column vectors of the $4\times 2$ matrix 
$\begin{pmatrix} I \\ S\end{pmatrix}$ form a basis of the 
corresponding Lagrangian plane $\xi\in U_{\xi_0}$. In particular, $\xi\in 
\dP(K')$ iff the trace of $S$ is $\text{tr}S:=2c=0$. In this case $[1: a-ib]$ is 
the homogeneous coordinate of $\xi\in \dP(K')$, with 
\[ 
\tan 2\h=\frac{b}{a}, \quad \tan s=\sqrt{a^2+b^2}. 
\]  
Indeed the map 
\[ 
[e^{i\h}\cos s:e^{-i\h}\sin s]\to (a=\cos2\h \tan s, -b=-\sin 2\h
\tan s) 
\] 
is the stereographic projection map of  $\dP(K')\setminus\{ [0:1]\}$ from 
its north pole $\xi_\infty=[0:1]$ onto $\R^2$. 

Let $L$ be an oriented surface and $g_L:L\to \Lambda^+$ a smooth map. Composing $g_L$ with the projection $\pi':
\Lambda^+\to \dP(K')$ we get 
\[ 
g'_L:=\pi'\circ g_L:L\to \dP(K').
\] 

Given $q_0\in L$ and 
 assume that $g'_L(q_0)=\xi_\infty$ for the moment. Let $U\subset \R^2$ be a coordinate neighborhood of $q_0$ with coordinates $(x_1,x_2)$ 
such that $g'_L(U)\subset \dP(K')\setminus\{ \xi_\infty\}$. 
Then near $q_0$ the map $g'_L$ can be expressed as 
\[ 
g'_L(x_1,x_2)=(a(x_1,x_2), -b(x_1,x_2)). 
\] 
 Then by a direct computation 
we find for $q\in U$ 
\[ 
b(q)\cos 2\h -a(q)\sin 2\h 
\begin{cases} 
>0 & \text{ iff $g'_L(q)\in D_\h$}, \\ 
=0 & \text{ iff $g'_L(q)\in \lambda_\h$}, \\ 
<0 & \text{ iff $g'_L(q)\in D_{\h^\perp}$}. 
\end{cases} 
\] 

We denote $a_i:=\frac{\pa a}{\pa x_i}$ and similarly 
$b_j:=\frac{\pa b}{\pa x_j}$. The differential $dg'_L$ can 
be expresses as a matrix 
\[ 
dg'_L(q)=\begin{pmatrix} a_1(q) & a_2(q) \\ 
-b_1(q) & -b_2(q)\end{pmatrix}.  
\] 
Then $q$ is a singular point of $q'_L$ iff $\det dg'_L(q)=0$, 
i.e., if the gradient vectors $\nabla a$ and $\nabla b$ are 
linearly dependent at $q$. 

%
We will use the following notations: 
\begin{align*} 
L^+ & :=\{ q\in L\mid \det dg'_L(q)>0\} \\  
L^o & :=\{ q\in L\mid \det dg'_L(q)=0\} \\ 
L^- & :=\{ q\in L\mid \det dg'_L(q)<0\} \\ 
V_\h &:=(g'_L)^{-1}(D_\h) \\ 
\Gamma_\h &:=(g'_L)^{-1}(\lambda_\h)=(g'_L)^{-1}(\pa D_\h)
\end{align*} 

Recall that $\Gamma_\h=\Gamma_{\h^\perp}$. Generically each of $\Gamma_\h\subset L$ is a $1$-dimensional skeleton together with a finite number of isolated points.  Since 
\[ 
\Gamma_\h=\{ q\in L\mid \dim E_\h\cap g'_L(q)=1\}=\{ 
q\in L\mid \dim E_\h\cap g_L(q)=1\}
\] 
we call $\Gamma_\h$ the {\em $E_\h$-locus} of $g'_L$. 
Note that 
\[ 
(g'_L)^{-1}(\{ \xi_0,\xi_\infty\})\subset \Gamma_\h, \quad \forall \h\in \R/\pi\Z,  
\] 
and for $\h,\h'\in [0,\frac{\pi}{2})$ with $\h\neq \h'$, 
\[ 
\Gamma_\h\cap \Gamma_{\h'}=(g'_L)^{-1}(\{ \xi_0,\xi_\infty\}). 
\]

%
%

{\bf Curves in $L^o$.} \ 
Generically the set $L^o$ of singular points of $g'_L$ is a 
$1$-dimensional skeleton, the union of a finite number of 
immersed closed curves and a finite number of isolated points. 

Let $\sigma\subset L^o$ be a connected immersed closed curve 
with a finite number of self-intersection points. 
Recall that if $g'_L(\sigma)$ misses the point $\xi_\infty$ 
and $dg'_L|_q\neq 0$ at $q\in \sigma$, then the kernel 
of $dg'_L|_q$ is tangent to level sets $\{ a=a(q)\}$ and 
$\{ b=b(q)\}$ at $q$.

\begin{defn}[{\bf Sign-changing}] \label{sign-changing} 
{\rm 
We say $\sigma$ is 
{\em sign-changing} if the determinant $\det dg'_L$ changes its 
$\pm$-signs at $\sigma$, i.e., if $\sigma$ is contained in the closure of $L^+$ as well as the closure of $L^-$. 
} 
\end{defn} 

The sign-changing property is persistent under a small perturbation of $g'_L$. If $\sigma$ is not sign-changing then it may disappear 
or split into a pair of sign-changing curves under a  small perturbation $g'_L$.

\begin{defn}[{\bf Ordinary folding curve}]  
{\rm 
We say that a sign-changing curve $\sigma$ is an {\em ordinary  folding 
curve} if its tangent  
$\dot{\sigma}(q)\not\in\ker(dg'_L)$ for all but a  finite 
number of points $q\in \sigma$. 
}
\end{defn} 

The image $g'_L(\sigma)$ is $1$-dimensional if $\sigma$ is an 
ordinary folding curve. On the other hand, there may exist a 
curve in $L^o$ whose image in $\dP(K')$ is a point.

\begin{defn}[{\bf $p$-curve}] \label{pcurve}
{\rm 
Let $\xi$ be a singular value of $g'_L$. Then a 
1-dimensional connected component 
$\gamma\subset L^o$ of 
$(g'_L)^{-1}(\xi)$ is called a $p$-curve. 
}
\end{defn} 

Let $\xi$ be a singular value of $g'_L$. Suppose that 
$(g'_L)^{-1}(\xi)$ contains some 1-dimensional connected components. Let $\gamma$ be one of the connected components.  Assume for the moment 
that $g'_L(\gamma)=\xi\neq \xi_\infty$. Then by composing with 
the stereographic projection from $\dP(K')\setminus\{ \xi_\infty\}$ to 
$\R^2$, we can easily see that $\gamma$ is a common level 
curve to both $a$ and $b$, i.e., $\gamma\subset a^{-1}(a_0)$ and 
$\gamma\subset b^{-1}(b_0)$ for some $a_0,b_0\in \R$. 
In general $a_0,b_0$ are not regular values of 
$a$ and $b$ respectively. Generically along 
$a^{-1}(a_0)$ the gradient $\nabla a$ 
of $a$ vanishes only at a finite set $S_a\subset a^{-1}(a_0)$. 
Similarly, $\nabla b\neq 0$ along $b^{-1}(b_0)$ except at 
a finite set $S_b\subset b^{-1}(b_0)$. Since $\gamma$ 
is of dimension one, {\em $S_a\cap S_b\cap \gamma$ is empty 
generically}. This means that along $\gamma$, the normal 
bundle $N_{\gamma/L}$ of $\gamma$ is fiberwise spanned 
by $\nabla a$ and $\nabla b$, which implies that 
{\em $\gamma$ is smoothly embedded}, and along $\gamma$ the differential $dg'_L$  is of rank one and $\ker(dg'_L)$ is spanned by 
$\dot{\gamma}$ the tangent of $\gamma$. This also applies to the case 
when $g'_L(\gamma)=\xi_\infty$. 
Note that it is possible that $\gamma$ is still embedded even though $\nabla a$ and $\nabla b$ together need not span 
$N_{\gamma/L}$ along $\gamma$.

{\bf Terminology alert:} \  Since $p$-curves of $g'_L$ are (smoothly) embedded for generic $g'_L$, {\em from now on all $p$-curves are assumed to be embedded unless otherwise mentioned}. A $p$-curve which is not embedded will be called 
a {\em singular $p$-curve}. 
The same abuse of language also applied to {\em folding $p$-curves} and {\em crossing $p$-curves} which will be defined 
below.

{\bf Folding v.s. crossing.} 
Let $\gamma$ be a  $p$-curve   and $U=U_\gamma\subset L$ 
be a small tubular neighborhood of $\gamma$ so that 
$\xi=g'_L(\gamma)\not\in g'_L(U\setminus\gamma)$. 
We parametrize $U$ as $\R/2\pi\Z\times (-\epsilon,\epsilon)$ 
with coordinates $(x_1,x_2)$ so that $\gamma=\{ x_2=0\}$. 
Let $U^+:=U\cap \{ x_2>0\}$ and $U^-:=U\cap \{ x_2<0\}$. 
Denote coordinate curves $\gamma_s:=\{ x_2=s\}$, 
$\ell_\h:=\{ x_1=\h\}$. Also let $\ell^+_\h:=\ell_\h\cap\{ x_2\geq 0\}$, $\ell^-_\h:=\ell_\h\cap\{ x_2\leq 0\}$. 

Let $D=D_\xi\subset\dP(K')$ be a neighborhood of $\xi$ 
diffeomorphic to an open disk. Identify $D$ with a disk 
of radius $\delta>0$ with center $\xi$, 
and let $(\rho,t)\in[0,\delta)\times\R/2\pi\Z$ be polar coordinates on $D$ with $\xi=\{ \rho=0\}$.  Identify $\gamma_s$ as 
$\R/2\pi\Z$, then  $t|_{\gamma_s}$, $0<s<\epsilon$, 
is a smooth family of  maps which extends continuously over 
$s=0$. Indeed $t(x_1,s)$ is the angle (oriented 
counterclockwise) from the polar axis of 
$\xi$ to the secant line  connecting $\xi$ and $g'_L(x_1,s)$ 
and {\em pointing away from $\xi$}, 
then $t(x_1,0)$ is defined to be the angle from the polar axis 
to the oriented tangent line of the image curve 
$g'_L(\ell^+_{x_1})$ 
at $\xi$, also pointing away from $\xi$. Put together 
we get a continuous  map $t^+:U^+\cup\gamma\to S^1$ which is smooth on $U^+$. 
The same holds true for the $s\leq 0$ case. We denote 
the corresponding map as $t^-:U^-\cup\gamma\to S^1$. 
Since $g'_L(\ell^+_{x_1})$ and $g'_L(\ell^-_{x_1})$ have the same unoriented tangent line at $\xi$ and $g'_L(U^\pm)$ are disjoint 
from $\xi$, by comparing the oriented tangent lines 
of $g'_L(\ell_{x_1}^\pm)$ at $\xi$ and by continuity of $t^\pm$ on $x_1$ exactly one 
of the followings will be satisfied: 

\begin{description}
\item[(F).] $t^-(x_1,0)=t^+(x_1,0) \mod 2\pi$, 
\item[(C).] $t^-(x_1,0)=t^+(x_1,0)+\pi \mod 2\pi$. 
\end{description}

%

In Case ({\bf F}) the image $g'_L(\ell_{x_1})$ has $\xi$ as its cusp 
point since the oriented tangent lines of $g'_L(\ell^\pm_{x_1})$ 
at $\xi$ point to the same direction. Thus the images of 
all $\ell_{x_1}$ "fold back" at $\xi$. Also in this case, $t^-$ and 
$t^+$ together form a continuous function on $U$, or equivalently, the composition of $g'_L$ with the coordinate function $t$ of $D_\xi$ is a continuous function on $U$. 
Note that here $\rho\geq 0$ also lifts to a smooth function 
on $U$. 

In Case ({\bf C}) the oriented tangent line for the image of $\ell_{x_1}$ at $\xi$ 
is defined for each $x_1$, meaning that the image curves of $\ell_{x_1}$ all "cross" the point $\xi$ at $x_2=0$ as $x_2$ 
increases from negative to positive. The two functions $t^-$ and 
$t^+$ do not match at $x_2=0$. Hence $t$ does not lift to a 
continuous function on $U$. 
This however can be remedied at the expense of allowing 
$\rho$ to be negative, namely instead of $(\rho,t)$ we 
 consider the extended polar coordinates  
$(\tilde{\rho},\tilde{t})\in (-\delta,\delta)\times \R/2\pi\Z$ 
with the equivalence relation $(\tilde{\rho},\tilde{t}+\pi)\sim 
(-\tilde{\rho},\tilde{t})$. Then both 
$\tilde{\rho}$ and $\tilde{t}$ lift 
to functions continuous on $U$ and smooth when $x_2\neq 0$: 
\[ 
(\tilde{\rho},\tilde{t})=\begin{cases} 
(\rho,t^+) & \text{ when $x_2\geq 0$}, \\ 
(-\rho, t^--\pi) & \text{ when $x_2<0$}. \end{cases}
\]

\begin{defn}[{\bf Crossing v.s. folding}]  \label{cvsf} 
{\rm 
Let 
We say that a $p$-curve $\gamma$ is a 
{\em folding} $p$-curve if ({\bf F}) holds for $\gamma$; 
a {\em crossing} $p$-curve if ({\bf C}) holds instead. 
} 
\end{defn}

Below we give a different description about the folding/crossing 
dichotomy of $p$-curves. 
Let $\gamma$ be a $p$-curve, 
$\xi:=g'_L(\gamma)$, $D_\xi\subset \dP(K')$ an open 
disk with polar coordinates $(\rho, t)\in [0,\delta)\times \R/2\pi\Z$ centered at $\xi$. Then $g'_L(L^o)$ intersects transversally 
with level sets of $\rho$ on $D_\xi\setminus\{ \xi\}$ provided 
$\delta>0$ is small enough. This implies that there exist a 
tubular neighborhood $U=U_\gamma\subset L$ missing 
all $p$-curves except for $\gamma$, and 
coordinates $(x_1,x_2)\in \R/2\pi\Z\times (-\epsilon,\epsilon)$ 
for $U$ so that

%
%
%

%
%
\begin{description} 
\item[{\rm (c1).}] $\rho$ depends only on $x_2$, $\frac{\pa \rho}{\pa x_2}\neq 0$ on $U\setminus \gamma$,  
\item[{\rm (c2).}] $\rho^{-1}(0)=\gamma=\{ x_2=0\}$,  
\item[{\rm (c3).}] $dg'_L$ is of rank $1$ on $(L^o\setminus \gamma)\cap U$, 
\item[{\rm (c4).}] $\pa_{x_1}\in \ker (dg'_L)$ on 
$L^o\cap U$. 
 \end{description} 

Without loss of generality we may assume that $\xi:=g'_L(\gamma)\neq \xi_\infty$. 
Let $a_0,b_0\in \R$ be such that 
$\gamma \subset\{ a=a_0\}\cap \{ b=b_0\}$. 
Let $\bar{a}=a-a_0$, $\bar{b}=b-b_0$. By adding to $t$ a constant if necessary we may assume that 
$\bar{a}=\rho \cos t$, $\bar{b}=-\rho\sin t$. 

Observe that the $\pm$-sign of 
$\rho_2:=\frac{\pa \rho}{\pa x_2}$ 
changes as we cross  $\gamma=\{ x_2=0\}$. 
Replacing the coordinate $x_2$ by $-x_2$ if necessary we may assume that 
\[ 
\rho_2=
(\cos t) \bar{a}_2-(\sin t) \bar{b}_2 
\begin{cases}  <0 & \text{ if $x_2<0$},  \\ 
>0 & \text{ if $x_2>0$},  
\end{cases} 
\] 
where $\bar{a}_2:=\frac{\pa \bar{a}}{\pa x_2}$ and 
$\bar{b}_2:=\frac{\pa \bar{b}}{\pa x_2}$. We arrive at the 
following observations.

\begin{description} 
 
\item[(F')] If $\gamma$ is a folding $p$-curve,  $t$ is 
continuous  on $U$, then along each line $\ell_\h:=\{ x_1=\h\}$  both $\bar{a}_2$ and $\bar{b}_2$ 
change signs at $x_2=0$. Thus  for both functions $a$ and $b$, the surface $L$ "folds" at $\gamma$.  
  
\item[(C')] If $\gamma$ is a crossing $p$-curve then $t$ jumps 
  by the value $\pi \mod 2\pi$ at $\gamma$, hence 
  along each line $\ell_\h$ the signs of 
 $\bar{a}_2$ and $\bar{b}_2$ do not  change at $x_2=0$. 
 This in particular is the case when 
 $\nabla a=\nabla\bar{a}$ and $\nabla b=
\nabla\bar{b}$ span the normal 
bundle $N_{\gamma/L}$ along $\gamma$. 
\end{description}

\begin{prop} \label{sign} 
Let $\gamma$ be a $p$-curve, then $\gamma$ is sign-changing. 
\end{prop} 

\begin{proof} 
Let $\xi:=g'_L(\gamma)$. Let $U=U_\gamma$, $U^\pm$, $D_\xi$ be as defined above. 
Let $V\subset U$ be a connected component of $U\setminus 
\overline{L^o\setminus\gamma}$ with $V\cap \gamma\neq \emptyset$. Write $V\setminus \gamma=V^+\cup V^-$,  $V^+:=V\cap U^+$ and $V^-:=V\cap U^-$. $g'_L$ 
is nondegenerate on both $V^\pm$. Note 
 that the $\pm$ sign of $\frac{\pa \rho}{\pa x_2}$ on $V^+$ is opposite to  that on  $V^-$. Also   $\frac{\pa t}{\pa x_1}\neq 0$ 
 on $V^+\cup  V^-$ since $\rho$ depends only 
 on $x_2$. Moreover the signs of $\frac{\pa t}{\pa x_1}\neq 0$   on 
 $V^+$ and $V^-$ are the same since $t$ changes only by a 
 constant when crossing $\gamma$. So $\gamma$ is sign-changing since on $V$ the 
 sign of $\det(dg'_L)=\frac{\pa \rho}{\pa x_2}\frac{\pa t}{\pa x_1}$ changes at $\gamma$. 
 This completes the proof. 
\end{proof}

{\bf Domain-switching property.} \ 
Let $\gamma$ be a $p$-curve, and $U^\pm,V^\pm, D_\xi$ be as above. Recall the polar coordinates $(\rho,t)$ and  the extended 
polar coordinates $(\tilde{\rho},\tilde{\h})$ for $D_\xi$. Observe 
that $(\tilde{\rho},\tilde{t})|_{V^\pm}$ is a coordinate system 
for each of $V^\pm$.  Recall the coordinates $(x_1,x_2)$ for 
$U$ with $\gamma=\{ x_2=0\}$, $x_2\cdot \frac{\pa \rho}{\pa x_2}>0$ for $x_2\neq 0$. Then 
$\det g'_L=\frac{\pa \tilde{\rho}}{\pa x_2}\frac{\pa \tilde{t}}{\pa x_1}$ changes sign at $\gamma$. Moreover,

\begin{enumerate} 
\item if $\gamma$ is a folding $p$-curve, then 
$\frac{\pa \tilde{\rho}}{\pa x_2}$ changes sign at $\gamma$, but 
$\frac{\pa \tilde{t}}{\pa x_1}$ does not. 

\item if $\gamma$ is a crossing $p$-curve, then 
$\frac{\pa \tilde{\rho}}{\pa x_2}$ does change sign at 
$\gamma$, but 
$\frac{\pa \tilde{t}}{\pa x_1}$ does. 

\end{enumerate}   

Compare with the fact that $\frac{\pa \rho}{\pa x_2}$ changes 
sign at $\gamma$ for any $p$-curve $\gamma$, we observe 
that as a sign-changing curve, a crossing $p$-curve also 
has what we call the {\em domain-switching property}: 

\begin{defn} 
{\rm 
Let $\gamma\subset L$ be a closed curve with a finite number 
of self-intersections. We say that 
$\gamma$ has {\em domain-switching} property 
(with respect to the map $g'_L\to \dP(K')$) 
if for every $q\in\gamma$ which is not a self-intersection 
point of $\gamma$ there is a connected open neighborhood $U_q\subset L$ of $q$ such that 
$U_q\setminus\gamma=U^+_q\cup U^-_q$, with $U^+_q\neq \emptyset$ and $U^-_q\neq \emptyset$ on different sides of $\gamma$, satisfies the following condition 
\[ 
g'_L(U^+_q)\cap g'_L(U^-_q)=\emptyset. 
\] 
On the other hand, we say $\gamma$ is {\em domain-folding} 
if for every $q\in\gamma$ and any connected open neighborhood $U_q\subset L$, with $U^+_q\neq \emptyset$ and $U^-_q\neq \emptyset$ on different sides of $\gamma$, satisfies the following condition
\[ 
g'_L(U^+_q)\cap g'_L(U^-_q)\neq\emptyset. 
\] 
} 
\end{defn} 
Clearly folding curves are sign-changing, and curves which are 
not sign-changing are domain-switching. On the other hand, 
a crossing $p$-curve is both sign-changing and domain-switching.

\begin{prop} 
If $\gamma\subset L$ is an embedded sign-changing and 
domain switching closed curve with respect to $g'_L$, then 
$\gamma$ is a crossing $p$-curve. 
\end{prop} 

\begin{proof} 
Let $U$ be a tubular neighborhood of $\gamma$ with 
coordinates $(x_1,x_2)\in \R/2\pi\Z\times (-\epsilon,\epsilon)$ 
so that $\gamma=\{ x_2=0\}$. Let $\ell_\h=\{ x_1=\h\}$ be 
an $x_1$-coordinate line. Since $\gamma$ is domain-switching 
$g'_L|_{\ell_\h}:\ell_\h\to L$ is an embedding near $\ell_\h\cap \gamma$. Then  the sign-changing property of $\gamma$ 
forces $dg'_L(\pa_{x_1})=0$ along $\gamma$, i.e., 
$\gamma$ is tangent to $\ker(dg'_L)$. So $\gamma$ is a 
crossing $p$-curve.  
\end{proof}

The sign-switching property and domain-changing property of 
a curve are rigid under small deformation of $g'_L$, which 
implies the following rigidity of a 
a crossing $p$-curve $\gamma$.

\begin{prop}[{\bf Rigidity of a crossing $p$-curve}] \label{Vswitch} 
Let $\gamma$ be a  crossing $p$-curve. 
Then $\gamma$  is  topologically rigid under small perturbations 
of $g'_L$ compactly supported in a small tubular 
neighborhood of $\gamma\subset L$. 
\end{prop}

\begin{rem} \label{ordinary}
{\rm 
Unlike its crossing counterpart, a folding $p$-curve may turn 
to an {\em ordinary folding curve} (i.e., the image under $g'_L$ is 1-dimensional) under a small perturbation. 
}
\end{rem}

The following two lemmas demonstrates the difference 
between a crossing $p$-curve and a folding curve via 
$\Gamma_\h$ and $V_\h$.

\begin{lem}[\bf {$\Gamma_\h$ and $p$-curves}] \label{Dswitch1}
Let $\gamma\subset \Gamma_\h$ be a  $p$-curve. Suppose  
there exists a curve $\sigma\subset \Gamma_\h$  
intersecting transversally with $\gamma$ at a point $q$. Assume that $q$ is the only intersection point of 
$\sigma$ with $L^o$ in a small neighborhood $U_q\subset L$ of 
$q$. Also assume that $L^o\cap U_q=\gamma\cap U_q$. Identity $U_q$  with an 
open domain in $\R^2$ with $q$ corresponding to the origin, and 
$\gamma$ the $x$-axis, $\sigma$ the $y$-axis. 

\begin{enumerate} 
\item If $\gamma$ is folding then either $\{x<0, y\neq 0\} 
\subset V_\h$ and $\{ x>0, y\neq 0\}\subset V_{\h^\perp}$, or 
the other way around.

\item If $\gamma$ is  crossing  then 
either $\{ xy>0\}\subset V_\h$ and 
$\{ xy<0\}\subset V_{\h^\perp}$, or the other way around. 
\end{enumerate} 

In other words,upon passing the intersection point $q$ when move along $\sigma$ in either direction, 
the two domains $V_\h$ and $V_{\h^\perp}$ 
\begin{enumerate} 
\item stay on their own sides of $\sigma$ if $\gamma$ is folding, 
\item switch to their 
opposite sides across $\gamma$ simultaneously if $\gamma$ is crossing. 
\end{enumerate} 
\end{lem}

\begin{proof}
 Assume that $\xi:=g'_L(\gamma)
\neq \xi_\infty$ for the moment. Then by the stereographic 
projection $\dP(K')\setminus\{ \xi_\infty\} \to \R^2$ one can 
see that $\sigma\subset \Gamma_\h$ intersects with $\gamma$ 
iff $\lambda_\h$ is the line passing through $\xi_0=(0,0)$ and 
$\xi=(a_\xi,b_\xi)$. Recall the local coordinate system 
$(\bar{a},\bar{b})$ and the corresponding polar coordinate 
system $(\rho,t)$ around $\xi$. By 
rotating the local coordinate system 
$(\bar{a},\bar{b})$ (and hence $(\rho, t)$) if necessary we may 
assume that  $t=2\h_0$ along one 
branch of $\lambda_{\h_0}\setminus\{ \xi\}$, and $t=2\h_0^\perp$ along 
the other branch of $\lambda_{\h_0}-\xi$, where $\h_0\in \R/\frac{\pi}{2}\Z$ is a value such that $\xi\subset \lambda_{\h_0}$  
($\h_0$ is unique $\mod \pi$ if $\xi\neq \xi_0$ or $\xi_\infty$). 
Recall the definitions of  $D_\h$ as in (\ref{Dh}) and (\ref{Dhp}).
If $\gamma$ is folding then the angle $t$ along $\sigma$ 
does not change after $\sigma$ meeting with $\gamma$, so 
$D_{\h_0}$ remain on the same side of $\sigma$. If $\gamma$ 
is crossing, then the angle $t$ along $\sigma$ changes by 
an amount of $\pm \pi$ after $\sigma$ meeting with 
$\gamma$ instead. Then $\h_0$ 
changes to $\h_0^\perp$, hence both $D_{\h_0}$ and 
$D_{\h_0^\perp}$ 
switch to their opposite sides simultaneously after $\sigma$ 
meeting with $\gamma$. 

The case $\xi=\xi_\infty$ is the same as the case $\xi_0$. 
This completes the proof. 
\end{proof}

A similar result holds for the case when an embedded $p$-curve 
$\gamma\subset \Gamma_\h$ is a connected component of 
$\Gamma_\h$. 

\begin{lem} \label{Dswitch2} 
Suppose that a $p$-curve $\gamma\subset \Gamma_\h$ is 
a connected component of $\Gamma_\h$. Let $U$ be a tubular 
neighborhood of $\gamma$ and $U^+$, $U^-$ denote the 
two connected components of $U\setminus\gamma$. Assume 
that $U$ is small enough, then 
\begin{enumerate} 
\item either $U^\pm\subset V_\h$ or $U^\pm\subset V_{\h^\perp}$  if $\gamma$ is folding, 
\item one of $U^+,U^-$ is in $V_\h$ and the other is in 
$V_{\h^\perp}$  if $\gamma$ is crossing. 
\end{enumerate} 
\end{lem} 

\begin{rem} 
{\rm 
Lemmas \ref{Dswitch1} and \ref{Dswitch2} demonstrate the 
{\em domain switching} ($V_\h\leftrightarrow V_{\h^\perp}$) 
property as one crosses a 1-dimensional 
 is held exactly only for crossing $p$-curves. This property sets 
 crossing $p$-curves apart from other types of curves in $L^o$. 
 The {\em domain-switching} and {\em sign-changing} properties 
 of crossing $p$-curves are rigid and in a suitable sense make the deformation of crossing $p$-curves independent from that  of 
 the rest of $L^o$. 
} 
\end{rem}

\begin{defn}[{\bf $p$-domain and crossing/folding domain}] \label{pdomain}
{\rm 
Given a smooth map $g'_L:L\to \dP(K')$, we say that 
$U\subset L$ is a {\em $p$-domain} if each of connected 
component of $\pa U$ is a $p$-curve. 
We say a $p$-domain $U$ is a {\em folding domain} if 
$\pa U$ consists of folding $p$-curves, and the interior 
of $U$ does not contain any crossing $p$-curve; and $U$ is a {\em 
crossing domain} if $\pa U$ consists of crossing $p$-curves, and the interior 
of $U$ does not contain any crossing $p$-curve. 
} 
\end{defn} 

\begin{defn}[{\bf PLG-degree of a $p$-domain}] 
{\rm 
Let $U\subset L$ be a $p$-domain. Then $U$ associates a 
closed surface $\hat{U}$ which is obtained by collapsing each 
connected component of $\pa U$ to a point. 
Let $\overline{U}$ denote the closure of $U$, and 
$r:\overline{U}\to \hat{U}$ the corresponding collapsing map. Then $g'_L$ 
induces a map, denoted by $g'_U$, such that 
\[ 
g'_U(q)=\begin{cases} 
g'_L(q) & \text{ if $q\in U$}, \\ 
g'_L(\gamma) & \text{if $r^{-1}(q)=\gamma\subset \pa U$} 
\end{cases} 
\] 
The $g'_L$-degree of $U$ is defined to be the degree of the map 
$g'_U:\hat{U}\to \dP(K')$. 
} 
\end{defn}

%
%
\section{Invariants of Lagrangian surfaces}

In this section we define a new invariant, called $y$-index, for orientable Lagrangian surfaces immersed in a symplectic manifold $(W,\omega)$. For technical simplicity 
we assume that $(W,\omega)$ is parallelizable.

%
%
\subsection{Parallelizable symplectic 4-manifold} \label{paraM} 

Let $(W,\omega)$ be a symplectic 4-manifold 
and $J$ a $\omega$-compatible complex structure on $W$. It is well known that the set of all $\omega$-compatible almost complex structures is contractible, hence the Chern classes 
$c_i(W,J)$ are independent of the choice of $J$. Omitting $J$ we simply write the Chern classes of $(W,\omega)$ as 
$c_i(W)$. From now on we assume that $(W,\omega)$ satisfies 
the following condition: 
 
\begin{cond}[{\bf Parallelizability}]
\[ 
c_1(W)=0 \  \text{ and } \  c_2(W)=0.  \label{cond-c}
\] 
\end{cond}

Often it is convenient for computation if $W$ also satisfies 
\begin{cond} 
\[ 
H_1(W,\Z)=0 \ \text{ and } \ H_3(W,\Z)=0.   \label{cond-t} 
\] 
\end{cond} 

Condition (\ref{cond-c}) implies that $TW$ together with an $\omega$-compatible almost complex structure is a trivial complex vector 
bundle, hence $W$ is parallelizable. Condition (\ref{cond-t}) 
on cohomologies  ensures  that the complex trivialization of 
$TW$ is unique up to homotopy. 
For example, a $1$-connected Stein surface with an associated symplectic structure and vanishing first Chern class 
satisfies these conditions.

Fix a $\omega$-compatible almost complex structure $J$ on 
$(W,\omega)$ and let $g:=\omega\circ(Id\oplus J)$ denote the 
corresponding Riemannian metric on $W$. 
Since $c_1(W)=c_2(W)=0$, $TW$ is a trivial $J$-complex vector 
bundle. We fix a pair of unitary sections $u,v$ of $TW\to W$ 
so that pointwise $u,v$ form a unitary basis of $T_pW$, $p\in W$. We call $( u,v)$ a {\em unitary framing} or {\em unitary basis}  (with 
respect to $J$).

With $( u,v)$ we associate a unique  pair of $g$-orthogonal almost complex 
structures $K',K''$ defined by 
\begin{gather*}
K'u = v,   \quad K'Ju  =  -Jv; \\ 
K''u =  Jv,   \quad  K''Ju  = v.  
\end{gather*} 
The triple $(J,K',K'')$ satisfy 
\[ 
JK'=K'', \quad  K'K''=J, \quad K''J=K'; 
\] 
and  generate 
an $S^2$-family of $g$-orthogonal almost complex structures 
\[ 
J_{a,b,c}:=aJ+bK'+cK'', \quad a^2+b^2+c^2=1. 
\] 

Let $K_t:=(\cos t) K'+(\sin t)K''$, $t\in \R/2\pi\Z$. 
Then 
$K_t$-complex planes are Lagrangian planes, and we have the 
decomposition of oriented Lagrangian planes over $p\in W$: 
\[ 
\Lambda^+ _p=\bigsqcup _{t\in\R/2\pi\Z}\dP_p (K_t), 
\] 
where $\dP_p(K_t)\cong \C P^1\cong S^2$ is the Grassmannian 
of $K_t$-complex 2-dimensional subspaces of $T_pW$ for $p\in W$. 

\begin{rem} \label{uv}
{\rm 
 The choice of 
$u$ is unique up to homotopy since $H_3(W,\Z)=0$, and the 
choice of $v$, which amounts to a framing of the $J$-complex 
plane bundle $(u\wedge Ju)^\perp$ orthogonal to 
$u\wedge Ju$, 
is also unique up to homotopy due to $H_1(W,\Z)=0$. 
}
\end{rem} 

With respect to the $J$-complex unitary basis $(u,v)$, $TW$ is bundle 
isomorphic to $W\times \C^2$, and we identify $T_pW$ with 
the vector space $\C^2$ via the projection map 
\begin{gather*} 
\Psi:TW\to \C^2, \\ 
\Psi (a_1u+a_2Ju+b_1K'u+b_2K''u):=a_1\pa_{x_1}+
a_2\pa_{y_1}+b_1\pa_{x_2}+b_2\pa_{y_2}. 
\end{gather*}  
Then the complex structures $J$, $K_t$ are identified 
pointwise with 
their counterpart in $\R^4$ as studied in Sections \ref{LG} and 
\ref{intCL}, and relevant constructions like $\cC$, $\cG$ and 
$E_\h$ there can be extended  over $TW$ straightforwardly. 
For notational simplicity we will use the same notations 
$\cC$, $\cG$, $E_\h$ for their pullback over $TW$ by $\Psi$.  

%
%


The orbit space of $\cG$ contains a unique great circle 
subbundle $\cE_\cG$ of $\dP(J)$ 
formed by elements $E$ of $\dP(J)$ which have 
nontrivial intersection with $u\wedge v$, i.e., 
$\dim E\cap (u\wedge v)=1$. More precisely, 
let $u_\theta:=\cos\theta u+\sin\theta v$ and define 
\[ 
E_\theta:=u_\theta\wedge Ju_\theta, \quad \theta\in\R/2\pi\Z, 
\] 
Then $\cE_\cG=\{ E_\theta\mid 0<\theta<\pi\}$. 
Note that $E_\theta=E_{\theta+\pi}$, and $E_{\theta+\frac{\pi}{2}}=E_\theta^\perp$.

%

\begin{notn} 
{\rm 
Let $\cF^\omega$ denote the set of triples $(J,u,v)$ over $W$ 
where $J$ is an $\omega$-compatible almost complex structure, 
and $\{ u,v\}$ is a $J$-complex unitary framing of $TW$. Here the 
norms of $u,v$ are determined with respect to the Riemannian 
metric $g:=\omega \circ (Id\times J)$. 
Let $\cF^\omega_J:=\{ (J',u',v')\in \cF^\omega\mid J'=J\}$. 
}
\end{notn} 

\begin{rem} 
{\rm 
The set $\cF^\omega$ is path-connected if $H_1(W,\Z)=0=W_3(W,\Z)$. 
}
\end{rem}

%
%
\subsection{The $\mu_2$-index} \label{mu2index}

Fix $\ff=(J,u,v)\in \cF^\omega$, the corresponding 
unitary trivialization 
\[ 
\Psi:TW\cong W\times \C^2\to \C^2\cong \R^4
\] 
 induces a trivialization of 
the associated bundle $\cL^+$ of oriented Lagrangian Grassmannian over $W$: 
\[ 
\cL^+\cong W\times \Lambda^+. 
\] 
Consider the associated projection onto $\Lambda^+$: 
\[ 
\hat{\pi}:\cL^+\cong W\times \Lambda^+\to \Lambda^+
\]  
Also  recall the projection $\pi': \Lambda^+\to \dP(K')\cong S^2$. 

Given an immersed oriented Lagrangian surface $L\subset W$ 
we can associate to it the {\em projected Lagrangian Gauss map}   
(PLG-map) 
\begin{align*} 
g'_L: L  & \to \dP(K') \\ 
p & \to g'_L(p):=\pi'\circ \hat{\pi}  (T_pL). 
\end{align*}  
\begin{defn} \label{mu2} 
{\rm 
Let $L$, $W$, $g'_L$ and $\dP(K')$ be as above. We define 
the $\mu_2$-index of $L$ (oriented) with respect to $\ff\in 
\cF^\omega$ to be 
\[ 
\mu_2(L;\ff):=\deg (g'_L) \in \Z,  
\] 
the PLG-degree of  $g'_L$ from $L$ to 
$\dP(K')$. The Grassmannian $\dP(K')$ is oriented by its 
 $K'$-complex structure.
 }
\end{defn}

\begin{prop}   \label{mu2ori}
The number $\mu_2(L;\ff)$ is independent of the choice of the 
orientation of $L$, i.e., $\mu_2(-L;\ff)=\mu_2(L;\ff)$ where 
$-L$ denote $L$ with the opposite orientation. 
\end{prop} 

\begin{proof} 
Let $r:-L\to L$ denote the orientation reversing map 
defined by $r(p)=p$ for $p\in -L$. 
For $p\in -L$, 
\[ 
g'_{-L}(p)=\pi'\circ\hat{\pi}(T_p(-L))=\pi'\circ\hat{\pi}(-T_{r(p)}L).  
\] 
 Let 
 $v_1,v_2$ be a positive orthonormal basis of $T_pL$, so 
 $T_pL=v_1\wedge v_2$. Then 
 \[ 
 T_p(-L)=v_2\wedge v_1=(Jv_1\wedge Jv_2)^\perp =
 (c_{\frac{\pi}{2}}(v_1\wedge v_2))^\perp .  
 \] 
 Note that the map $A:\Lambda^+\to \Lambda^+$ defined by 
\[ 
A(\xi):=\xi^\perp, \quad \xi\in\Lambda^+
\] 
commutes with the action of $\cC$, preserving  each of $\dP(K_t)$ and acting on which as an antipodal map. So we 
have 
\begin{equation} \label{g-L}  
g_{-L}=A\circ c_\frac{\pi}{2}\circ g_L \circ r =c_\frac{\pi}{2}\circ A\circ g_L \circ r
\end{equation} 
and hence 
\[ 
g'_{-L}=A\circ g'_L\circ r, 
\] 
with  $A$ viewed  as a map from $\dP(K')$ to $\dP(K')$. 
Then 
\[ 
\deg (g'_{-L})=\deg (A)\cdot \deg (g'_L)\cdot \deg (r)= 
\deg (g'_L)
\] 
since $\deg(A)=-1=\deg(r)$. So 
\[ 
\mu_2(-L;\ff):=\deg(g'_{-L})=\deg(g'_L)=\mu_2(L;\ff). 
\] 

\end{proof}

\begin{prop} \label{mu2-indep}   
Let $L$ be an oriented compact Lagrangian surface immersed in $W$. Given $\ff=(J,u,v),\ff'=(J',u',v')\in \cF^\omega$, then 
\[ 
\mu_2(L;\ff)=\mu_2(L;\ff')
\]  provided that 
$\ff$ and $\ff'$ are in the same connected component 
of $\cF^\omega$. 
\end{prop}

\begin{proof} 
Let $(J_t,u_t,v_t)$, $t\in[0,1]$, be a path in $\cF^\omega$ 
connecting $(J,u,v)=(J_0,u_0,v_0)$ and $(J',u',v')=(J_1, u_1,v_1)$. Each of $(J_t,u_t,v_t)$ induces a trivialization 
of $\cL^+|_L\overset{\Psi_t}{\underset{\cong}{\longrightarrow}} L\times S^1\times S^2$. 
\[ 
\phi_{t,p}:=\Psi_0\circ\Psi_t^{-1}|_{\{ p\}\times S^1\times S^2}:S^1\times S^2\to S^1\times S^2
\] 
is a smooth family of diffeomorphisms parameterized by 
$(t,p)\in[0,1]\times W$ with $\phi_{0,p}=id$ for all $p\in W$. 

Let $g_{t,L}:L\to S^1\times S^2$ be the corresponding Lagrangian Gauss map. Then $g_{t,L}(p)=\phi_{t,p}\circ g_{0,L}(p)$ for $p\in L$. So $\deg g'_{t,L}=\deg g'_{0,L}$ for 
$t\in [0,1]$. Hence $\mu_2(L;\ff)=\mu_2(L;\ff')$. 
\end{proof}

For $\ff=(J,u,v)\in \cF^\omega$ we denote by $[\ff]$ the 
equivalence of $\ff$ so that $[\ff]=[\ff']$ iff $\ff$ and 
$\ff'$ are in the same connected component of $\cF^\omega$. 

\begin{rem} 
{\rm 
$\mu_2(L,[\ff])$ is invariant under smooth regular homotopy of 
$L$ in the space of compact oriented Lagrangian surfaces immersed in $W$. 
} 
\end{rem}

\begin{exam} 
{\rm 
In the standard symplectic $\R^4$, $\mu_2(L,\ff)$ is 
independent of $\ff\in \cF^\omega$ since 
$H_1(\R^4)=0=H_3(\R^4)$. By direct computation one gets that 
$\mu_2(L)=2$ for Lagrangian Whitney spheres  
and $\mu_2=0$ for  both Chekanov tori and Clifford tori. See 
Section \ref{R4} for the detail. 
}
\end{exam}

{\bf Crossing domain decomposition.}  \ 
Recall the definitions of (crossing) $p$-curves and crossing domains from Section \ref{PK'}. Let 
\begin{align*} 
\Sigma & : \text{ the union of all $p$-curves 
of $g'_L$}, \\ 
\Sigma^c & : \text{ the union of all 
crossing $p$-curves of $g'_L$}.
\end{align*}
We assume that every crossing $p$-curve is embedded in $L$, 
which is true for generic $g'_L$. 

The complement $L\setminus\Sigma^c$ consists of a finite number 
of connected open subdomains denoted by $L_i$. 
Note that this decomposition is independent of the orientation 
of $L$.

Let $\bar{L}_i$ denote the closure of $L_i$. $\bar{L}_i$ is a 
compact surface with boundary. Let 
$\hat{L}_i$ denote the closed surface obtained by collapsing each 
of the boundary components of $\bar{L}$ to a point. 
Recall that   $g'_L$ 
induces for each $\hat{L}_i$ a map 
\[
g'_i:\hat{L}_i\to \dP(K'). 
\] 
Let 
\[  d_i:=\deg(g'_i). 
\] 
The degree of $g'_L:L\to \dP(K')$ is then 
\[ 
\mu_2(L;\ff)=d=\sum_i d_i. 
\] 
We all $d_i$ the {\em PLG-degree} of $L_i$. 

\subsection{The $y$-index} \label{y} 


Let $L\subset (W,\omega) $ be an immersed oriented 
Lagrangian suface 
and $g'_L:L\to \dP(K')$ the PLG-map 
with respect to the framing $\ff\in\cF^\omega$. 
Recall the critical set $L^o\subset L$ of $g'_L$. 
Generically $L^o$ is a codimension 1 subset of $L$, and 
the set of critical values $g'_L(L^o)$ is a codimension 1 subset 
of $\dP(K')$. Both $L^o$ and $g'_L(L^o)$ may contain 
a finite number of 0-dimensional connected components 
in addition to 1-dimensional ones. 

We assume that $L^o$ and $g'_L(L^o)$  satisfy the following conditions: 

\begin{enumerate} 
\item Both $L^o$ and $g'_L(L^o)$ are  disjoint unions of a finite 
number of 1-dimensional connected components and a finite 
number of isolated points. 


\item $g'_L$ has a finite number of crossing $p$-curves, all 
embedded.

%
%
%

\item The 1-dimensional components of $\overline{L^o\setminus\Sigma^c}$ are 
smooth curves with transversal self-intersections. 


Recall  $\Sigma\subset L^o$ the union of all $p$-curves of $L$.

\item For each $\h\in \R/\pi\Z$, $\lambda_\h\cap g'_L(L^o)$ 
is a finite set. Also, the intersection of $(g'_L)^{-1}(\lambda_\h)$ 
with $L^o\setminus\Sigma$ is finite, $\forall \h\in \R/\pi\Z$. 

\end{enumerate} 
A generic $g'_L$ will satisfy the conditions listed above.

{\bf Proper $E_\h$-locus.} \ 
Recall the $E_\h$-locus $\Gamma_\h:=(g'_L)^{-1}(\lambda_\h)$ 
for $\h\in \R/\pi\Z$, 
and the {\em proper $E_\h$-locus}  
\[ 
\check{\Gamma}_\h:=\overline{\Gamma_\h\setminus\Sigma}.  
\] 

Ignoring its 0-dimensional components (a finite number of points 
if not empty), each of $\check{\Gamma}_\h$ is a finite 
disjoint union of embedded curves except for a finite number 
of $\h$'s. For such exceptional $\h$, $\check{\Gamma}^+_\h$ 
has a finite number of self-intersection points which can be 
self-intersection points of $L^o\setminus\Sigma$ or isolated 
singular points of $g'_L$.

%
%
%
%


Recall the oriented geodesics $\lambda^+_\h$ as the 
boundary of $D_\h$, with $\lambda^+_{\h^\perp}=\lambda^-_\h$. 
We will show that there is a uniform way of orienting 
$\check{\Gamma}_\h$ so that the oriented 1-dimensional cycles 
$\check{\Gamma}^+_\h$ satisfy $\check{\Gamma}^+_{\h^\perp} =\check{\Gamma}^-_\h$, where $\check{\Gamma}^-_\h=-\check{\Gamma}^+_\h$, and the degree 
of $g'_{L,\h}:=\Gamma^+_\h\to \lambda^+_\h$ is independent 
of $\h$. 

Along $\check{\Gamma}_\h$ with $\h$ not exceptional, the Jacobian of $dg'_L$ changes 
its sign at a finite number of points on $\check{\Gamma}_\h$. 
These points are either {\em folding points} (intersection points with folding curves) or {\em crossing points} (intersection points with crossing $p$-curves) of 
$\check{\Gamma}_\h$.

{\bf Orientation of $\check{\Gamma}_\h\cap L_i$.} \ 
In each crossing domain $L_i$, there are no interior crossing points, and $V_\h$ stay on the the same side of $\check{\Gamma}_\h$, so there is a uniform way of orienting $\check{\Gamma}_\h\cap L_i$ for all $\h$: 
Observe that 
$\check{\Gamma}_\h\cap L_i$ 
are disjoint from $\pa L_i$ and are smooth for all but a 
finite number of $\h$. In $L_i$ we orient $\check{\Gamma}_\h$ 
so that $V_\h$ is on the {\em left hand side} of $\check{\Gamma}_\h$. 
The orientation of $\check{\Gamma}_{\h'}$ for exceptional $\h'$ 
can be determined by observing that $\check{\Gamma}_{\h'}$ is 
contained in the limit set of $\check{\Gamma}_\h$ as $\h\to \h'$. 
It is easy to see that with the orientation assigned to $\check{\Gamma}_\h\cap L_i$, $V_\h$ stay on the left 
hand side of $\check{\Gamma}_\h\cap L_i$ for all $\h$. 
\begin{notn} 
{\rm 
We denote by $\check{\Gamma}_{\h,i}$ the oriented set $\check{\Gamma}_\h\cap L_i$ so that $V_\h\cap L_i$ is on the left hand side of $\check{\Gamma}_{\h,i}$. 
} 
\end{notn}

\begin{cor} 
The degree of the restricted map $g'_L: \check{\Gamma}_{\h,i} 
\to \lambda^+_\h$ is $d_i$. 
\end{cor}

{\bf Graph $\Lambda_L$.} \ 
Note that {\em the above regional orientations for 
$\check{\Gamma}_\h$ do not match at crossing points}. We 
need to adjust the orientations of $\check{\Gamma}_{\h,i}$.  
First of all  to the decomposition  $(\cup_{i\in I} L_i)\cup 
\Sigma^c$ of $L$ into the disjoint union of crossing domains and crossing $p$-curves we associate 
a graph $\Lambda_L$ of which the vertices are indexed by $I$ the index 
set of $\cup L_i$ so that the vertex $v_i$ corresponds to $L_i$, 
and two vertices $v_i,v_j$ are connected by an edge if 
$L_i$ and $L_j$ has a common boundary (crossing $p$-curve) 
$\gamma\subset \Sigma^c$. 
Note that $\Lambda_L$ is connected since $L$ is.

\begin{prop} \label{evenloop}
Any loop embedded in $\Lambda_L$ is even, i.e.,  consisting of an even number of distinct edges. 
\end{prop} 

\begin{proof} 
Suppose that $\sigma\subset \Lambda_L$ is an embedded 
loop of odd type. We fix a cyclic ordering of its vertices 
$v_{i_1},v_{i_2},...,v_{i_n}, v_{i_{n+1}}=v_{i_1}$ with $n\in \N$ 
an odd number. 
$\sigma$ corresponds to a loop of crossing 
domains $L_{i_1},L_{i_2},...,L_{i_n},L_{i_{n+1}}=L_{i_1}$, 
and a sequence $\h_i\in \R/\pi\Z$ so that $L_{i_k}$ and 
$L_{i_{k+1}}$ share a common boundary $C_k\subset \Gamma_{\h_k}$. We can form a closed curve $\gamma$  embedded in $L$ such that 
\begin{itemize} 
\item $\gamma$ misses all points where $dg'_L=0$, 
\item $\gamma$ intersects orthogonally with  each of 
$C_k$ and in one point, 
\item $\gamma\pitchfork \sigma$ for all any 
folding curve $\sigma$ of $L$ if the intersection is not empty, and if $q\in \gamma \cap \sigma$, then $q$ is an embedded point of $\sigma$ such that $\sigma$ is not tangent to 
$\ker(dg'_L)$ at $q$, but $\gamma$ is. 
 \end{itemize} 


Let $U_\gamma\subset L$ be a small  tubular 
neighborhood of $\gamma$. We may identify $U_\gamma$ 
with the total space of the normal bundle $N_{\gamma/L}$ 
of $\gamma$ in $L$. Since $L$ is orientable, 
$U_\gamma$ is diffeomorphic to an annulus. Parameterize 
$U_\gamma$ as a strip $(t,s)\in [0,1]\times (-\epsilon,\epsilon)$ 
with the identification $(0,s)\sim (1,s)$ for $-\epsilon<s<\epsilon$, so that $\gamma$ is the curve $\{ s=0\}$, and 
the point with coordinates $(0,0)$ is a crossing point. 
Let $\ell$ denote the image of $\gamma$ under $g'_L$, 
$\ell$ is immersed except at folding points. 
 The normal vector field 
$v:=dg'_L(\pa _s)$ vanishes exactly at crossing points. 
Moreover, as one traces out $\gamma$ once the normal vector 
$v$ flips to the opposite side of $\ell$ exactly at crossing 
points. 
Since $n$, the number of crossing points, is odd we have that $dg'_L(\pa_s)_{(0,0)}
=-dg'_L(\pa_s)_{(1,0)}\neq 0$ which is impossible 
unless $U_\gamma$ is a M\"{o}bius band  but $U_\gamma$ 
is not. So all embedded loops of $\Lambda_L$ are of even type. 
\end{proof}

In fact, the evenness can be held for a wider class of closed 
curves in $\Lambda_L$. Let $\cV$ denote the set of all vertices 
of $\Lambda_L$. 
We say that a smooth map 
$\gamma:S^1\to \Lambda_L$ is an {\em immersed} loop in $\Lambda_L$ if (A) $\gamma^{-1}(\cV)$ is a compact $0$-dimensional set, and (B)  the restricted map $\gamma:\gamma^{-1}(\Lambda\setminus \cV)\to \Lambda_L$ is an immersion. 
An immersed loop of $\Lambda_L$ is endowed with the structure of a graph via the map $\gamma$. 

Similarly a map $\gamma':[0,1]\to \Lambda_L$ is called an {\em immersed 
path} in $\Lambda_L$ if $\{ 0,1\}\subset \gamma^{-1}(\cV)$ and 
$\gamma'$ satisfies the above two conditions (A), (B). 
An immersed path is endowed with the structure of a graph via 
$\gamma'$. We say that an immersed path is {\em even} if it 
has an even number of edges, {\em odd} if instead the number 
of its edges is odd.

\begin{cor} 
Any immersed loop $\gamma:S^1\to \Lambda_L$ is even 
(with respect to the induced graph structure. 
\end{cor} 

\begin{proof}
It is enough to show that $\gamma$ has an even number of edges, 
which follows from the condition (B) above and the evenness of 
every embedded loop in $\Lambda_L$. 
\end{proof}

\begin{cor} \label{path}
Let $v_i$ and $v_j$ be two vertices of $\Lambda_L$. 
Let $\sigma,\sigma':[0,1]\to \Lambda_L$ be two immersed paths in $\Lambda_L$ with $\sigma(0)=v_i=\sigma'(0)$ and 
$\sigma(1)=v_j=\sigma'(1)$. Let $n,n'$ be the number of edges of $\sigma$ 
and $\sigma'$ respectively. Then $n-n'\in 2\Z$. 
\end{cor}

{\bf Uniform orientation.} \ 
Define a function 
\begin{align*} 
\varepsilon=\varepsilon_L & :I\times I\to \{ \pm 1\}, \\ 
\varepsilon (i,j) & =\begin{cases} 
1 & \text{ if  $v_i,v_j$ are connected by an even path}, \\ 
-1 & \text{ if $v_i,v_j$ are connected by an odd path}. 
\end{cases} 
\end{align*} 
Corollary \ref{path} ensures that $\varepsilon(i,j)$ is independent 
of the choice of an immersed path connecting $v_i,v_j$ and 
hence is well-defined. 
The function $\varepsilon$ is symmetric, $\varepsilon(i,i)=1$ 
for $i\in I$, and more generally 
\[ 
\varepsilon(i,j)\varepsilon(j,k)=\varepsilon(i,k), \quad i,j,k\in I. 
\] 

Below we orient $\check{\Gamma}_\h$ 
for each $\h\in\R/\pi\Z$, so that the orientation of 
$\check{\Gamma}_{\h^\perp}$ is the opposite of the orientation 
of $\Gamma_\h$, and the orientation of $\check{\Gamma}_\h$ 
varies continuously as $\h$ varies. 

Firstly we fix a {\em reference crossing domain} say $L_{i_0}$ and 
We orient $\check{\Gamma}_\h$ so that the oriented $\check{\Gamma}_\h$, 
denoted as $\check{\Gamma}^+_\h$ is defined to be 
\[ 
\check{\Gamma}^+_\h \cap L_j=\varepsilon(i_0,j)\check{\Gamma}_{\h,j},  
\]
i.e., the orientation of $\check{\Gamma}^+_\h \cap L_j$ is equal 
to that of $\check{\Gamma}_{\h,j}$ iff $\varepsilon(i_0,j)$ is $+1$.

Note the orientation of $\check{\Gamma}^+_\h$ depends on the 
choice of a reference crossing domain $L_{i_0}$. If we choose 
a different reference crossing domain $L_{i'_0}$, then the new orientation 
for $\check{\Gamma}_\h$ will be different from the old one 
iff $\varepsilon(i_0,i'_0)=-1$.

\begin{defn}
{\rm 
Let $q\in L_{i_0}$ be a regular point of $g'_L$ with respect to 
$\ff\in \cF^\omega$, where $L_{i_0}$ is a crossing domain 
containing $q$. 
We define the {\em $y$-index of $L$}  relative to the framing $\ff$ 
and $q$  to 
be 
\[ 
y(L,q;\ff):=\sum_{i\in I} \epsilon(i_0,i)d_i.  
\] 
Then $y(L,q;\ff)=y(L,q';\ff)$ if $q\in L_{i_0}$ and $q'\in 
L_{i'_0}$ with $\varepsilon(i_0,i'_0)=1$. 

The {\em absolute $y$-index} relative to $\ff$ is defined to be 
\[ 
\bar{y}(L;\ff):=|y(L,q;\ff|. 
\] 
$\bar{y}(L;\ff)$ is independent of the choice of a reference 
crossing domain. 
} 
\end{defn} 

{\bf Relative $E_\h$-phase.} \  
Let $\alpha^+_\h$  denote the {\em $E_\h$-phase along $\check{\Gamma}^+_\h$}, which is defined to be the total angle of 
rotation of $T_qL\cap E_\h$ in $E_\h$ as $q$ traces out 
$\check{\Gamma}^+_\h$ once. Similarly let $\alpha^-_\h$  denote the {\em $E_{\h^\perp}$-phase along $\check{\Gamma}^+_\h$}, which is defined to be the total angle of 
rotation of $T_qL\cap E_{\h^\perp}$ in $E_{\h^\perp}$ as $q$ traces out 
$\check{\Gamma}^+_\h$ once. Also define the {\em relative 
$E_\h$-phase along $\check{\Gamma}^+_\h$} to be 
$\alpha_\h:=\alpha^+_\h-\alpha^-_\h$.

\begin{prop} 
 Let $\alpha^\pm_\h$ be as defined above. Then 
\[ 
y(L,q;\ff)=\frac{1}{2\pi}\alpha_\h=\frac{1}{2\pi}(\alpha^+_\h-\alpha^-_\h)
\] 
\end{prop} 

\begin{proof} 
Given a crossing domain $L_i$ of $g'_L$, 
et $\alpha_{\h,i}$ denote the relative $E_\h$-phase of 
$\check{\Gamma}_{\h,i}=\pa V_\h\cap L_i$. 
Recall the the relative $E_\h$-phase along $\lambda^+_\h=\pa D_\h\subset \dP(K')$ is $2\pi$. Then 
$\alpha_{\h,i}=2\pi d_i$, where $d_i$ is the PLG-degree of $L_i$. Note that the value of $\alpha_{\h,i}$ is independent of $\h$. 
Orientations of $\check{\Gamma}_{\h,i}$ and $\check{\Gamma}_{\h,j}$ match iff $\varepsilon(i,j)=1$. With 
$L_{i_0}$ as the reference domain we have $\check{\Gamma}^+_{\h}|_{L_{i_0}}=\check{\Gamma}_{\h,i_0}$, so 
\[ 
\alpha_\h=\sum_{i\in I} \varepsilon(i_0,i)\alpha_{\h,i} =2\pi 
\sum_{i\in I} \varepsilon(i_0,i)d_i, 
\] 
and 
\[ 
y(L,q;\ff)=\frac{1}{2\pi}\alpha_\h=\frac{1}{2\pi}(\alpha^+_\h-\alpha^-_\h). 
\] 
\end{proof}

\begin{prop} 
$y(L,q;\ff)$ is independent of the orientation of $L$, i.e., 
$y(-L,q;\ff)=y(L,q;\ff)$. 
\end{prop}

\begin{proof} 
Recall (\ref{g-L}): 
\[ 
g_{-L}=A\circ c_\frac{\pi}{2}\circ g_L \circ r =c_\frac{\pi}{2}\circ A\circ g_L \circ r, 
\] 
where $A:\Lambda^+\to \Lambda^+$, 
$A(\xi)=\xi^\perp\in \dP(K_t)$ for $\xi\in\dP(K_t)$. 
Recall also 
\[ 
g'_{-L}=A\circ g'_L\circ r. 
\] 
The map $r$ preserves each of the crossing domains $L_i$ pointwise except the orientation. It also preserves 
unoriented 
$\Gamma_\h=\Gamma_{\h^\perp}$ and hence the crossing $p$-curves of $L$. So $g'_{-L}|_{-L_i}=A\circ g'_L\circ r|_{-L_i}$
 and $g'_{-L}(-L_i)=A\circ g'_L(L_i)$. 
Since $A:\dP(K')\to \dP(K')$ is the antipodal map, 
the crossing domain decomposition of $L$ and that of $-L$ 
are identical (except the orientation). And the maps 
$g'_{-L_i}:-\hat{L}_i\to \dP(K')$ and $g'_{L_i}:\hat{L}_i\to 
\dP(K')$ have the same degree: $\deg(g'_{-L_i})=\deg(g'_{L_i}) 
=d_i$. By definition of $y(L;\ff)$ we have 
$y(-L,q;\ff)=y(L,q;\ff)$. This completes the proof. 
\end{proof}


{\bf Degeneration of crossing $p$-curves.} \ 
For a generic framing  $\ff$ the map 
$g'_L:L\to \dP(K')$ contains only a finite number of crossing 
$p$-curves, and all are embedded. Though their sign-changing 
and domain-switching properties are 
rigid under variations of $\ff$ in $\cF^\omega$, 
along a genetic math $\ff_t$ of 
$\omega$-compatible framings embedded crossing 
$p$-curves may become singular (called {\em singular crossing 
$p$-curves}) and 
undergo certain topological changes 
as typical level curves of a smooth 
function may have when the function varies with respect to $t$. 

The actual deformation pattern  can be 
very complicated but, for a generic path $\ff_t$, the deformations of crossing $p$-curves can be decomposed as a finite combination of two basic types of degenerations: 
(I) merge-split and (II) birth-death.

{\bf Merge-split.} \ 
Recall that for a generic $\ff$, $\nabla a$ and $\nabla b$ do 
not vanish simultaneously at a point on a common level curve 
$\gamma$ if $\gamma$ is a crossing $p$-curve. However 
the opposite 
may happen during a deformation  $\ff_t$ with $t\in[0,1]$. 
The simplest case is when this happens at a single point 
$q\in \gamma$ at a critical moment $t_i\in(0,1)$, turning $\gamma$ into an immersed curve with 
one self-intersection point $q$, and then {\em split} at $q$ into two embedded crossing $p$-curves after $t=t_i$. Reversing 
the process we get two disjoint  embedded $p$-curves 
become connected at a point $q$ and then {\em merge} into a 
single embedded crossing $p$-curve. The {\em merge-split} process is indeed in line with the deformation of  a level set 
$\{ f=c\}$ of a function $c$ as $c$ passes through a critical 
value of $f$.

{\bf Birth-death.} \ 
There are two kinds of death/birth of  
crossing $p$-curves: 
the death/birth of (A) a single  crossing $p$-curve and (B) 
a pair of  crossing $p$-curves.

The death of a single embedded 
crossing $p$-curve $\gamma$ occurs 
when $\gamma$ degenerates to an  point $q\in L$ of $g'_L$. 
This can happen if $\gamma$ is the boundary of 
a crossing domain $L_0$ (diffeomorphic to a disk) whose PLG-degree is $0$. 
Reversing the above processes gives rise to the birth 
of $\gamma$.

Case (B) corresponds to the creation or cancellation of a pair of 
embedded crossing curves which bound an annulus $U$,  and 
have opposite normal flow orientations (with respect to both 
 $\nabla a$ and $\nabla b$). Without loss of generality we may assume $U$ is a crossing  domain with vanishing PLG-degree.

 The simplest model for $U$ is 
one such that the 
interior $int(U)$ of $U$ contains precisely one embedded folding curve $\sigma$ parallel to each connected components of $\pa U=\gamma_0\sqcup\gamma_1$. And 
up to a small perturbation this can assumed to be the case when 
$U$ deforms to a very narrow strip. 
As the closure of $U$ degenerates to a curve 
$\sigma_0$, the two crossing $p$-curves $\gamma_0$ and 
$\gamma_1$ become $\sigma_0$. Observe that  $\sigma_0$ is a folding $p$-curve, since it is still a common level set of $a$ and $b$, sign-changing, but not domain-switching. Reversing 
the deformation process we get the birth of a pair of crossing 
$p$-curves which form the boundary of an annular domain of 
vanishing PLG-degree.

{\bf Interaction with other types of singular locus.} \ 
%
Can birth/death phenomenon happen to a crossing-folding pair? 
Let $\gamma_0$ be a crossing $p$-curve, $\gamma_1$ a folding curve disjoint from $\gamma_0$, such that $\gamma_0\sqcup \gamma_1$ is the boundary of some annular domain $U$. Assume that the interior of $U$ does not contain 
any crossing $p$-curve. Suppose that 
the closure of $U$ is deformed to a curve $\sigma_0$, 
then $\sigma_0$ is a $p$-curve because $\gamma_0$ is. 
If there are no other folding curves in $U$ parallel to $\gamma_0$ and $\gamma_1$ then $\sigma_0$ is not sign-changing, which contradicts to the fact that any $p$-curve is 
sign-changing, with the only exception that  $\gamma_0$ 
and $\gamma_1$ together degenerate to a point upon the occurrence of cancellation. This brings us back to the Case (A) 
of the birth/death of $\gamma_0$. Now if there is exactly one folding curve 
$\sigma$ in the interior that is parallel to $\gamma_0$ and 
$\gamma_1$,   then as the closure of $U$ 
deforms to a curve $\sigma_0$, $\sigma_0$ 
will be a crossing $p$-curve. In other words, in this deformation 
$\gamma_1$ will cancel with a middle folding curve, and the crossing $p$-curve $\gamma_0$ survives this deformation and is the curve $\sigma_0$. So strictly speaking this is about the 
birth/death of a folding pair,  has nothing to do with a 
crossing $p$-curve. In general, if there are $k$ folding 
curves in $U$ parallel to $\gamma_0$ and $\gamma_1$, then 
generically the these middle curves will undergo their own 
cancelations at first, hence reduce to the $k=0,1$ cases. 
We conclude  that there are no nondegenerated birth/death of a 
crossing-folding pair.

\begin{prop} 
$y(L,q;\ff)=y(L,q;\ff')$ provided that $\ff,\ff'\in \cF^\omega$ 
are in the same connected component of $\cF^\omega$. 
\end{prop}

 \begin{proof} 
  The $y$-index is defined by a suitable counting of the 
 degrees of crossing domains. as the framing $\ff_t$ varies in 
 $\cF^\omega$ these crossing domains may undergo topological 
 changes due the deformation of its boundary crossing $p$-curves. As we investigated earlier, generic deformations of 
 crossing $p$-curves consists of two types of 
 elementary deformations: (I) merge-split, and (II) birth-death.

In Case (I), the topology of crossing domains 
are changed, so it the graph associated to the 
crossing domain decomposition.  Let $\gamma'$ and 
$\gamma''$ be two crossing $p$-curves.  Let $e'$ and $e''$ denote the corresponding  edges in the graph. Recall that any 
simple closed loop in the graph must be even, i.e., has an 
even number of edges.  

Note that for the elementary merge between $\gamma'$ and 
$\gamma''$ to happen, $\gamma'$ and $\gamma''$ have to 
be in the boundary of a common connected crossing domain 
say $L_i$, otherwise either $e'$ or $e''$ would have to 
collide with 
another boundary crossing $p$-curve before they actually meet each other. The vertex $v_i$ representing $L_i$ is then 
a common vertex of $e'$ and $e''$. $e'$ and $e''$ may or may 
not have a second common vertex. 

Suppose that $e'$ and $e''$ have another common vertex $v_j$, i.e., $\gamma'\cup \gamma''
\subset \pa L_j$ for some other crossing domain $L_j$. 
Then the only simple closed loop containing both $e'$ and $e''$ 
is the loop $\ell:=v_i\cup e'\cup v_j\cup e''\cup v_i$.  When $\gamma'$ and $\gamma''$ merge into a crossing $p$-curve 
$\gamma\in \pa L_i\cup \pa L_j$, $e'$ and $e''$ collide and 
the loop $\ell$ deforms 
to the union of the edge $e$ representing $\gamma$ and 
the two vertices $v_i$, $v_j$. Since there are no other 
simple closed loops containing both $e'$ and $e''$, the 
evenness of simple closed loops are preserved under the 
identification of $e'$ and $e''$ as $e$. Note the the 
PLG-degrees  $d_{L_i}$ and $d_{L_j}$ of $L_i$ and $L_j$ are 
preserved and 
so is the number $\varepsilon(i,j)$. This also holds for any 
pair of crossing domains. So the $y$-index is not changed 
upon the merging of two edges with the same end-vertices.

 Now suppose that $v_i$ is the unique common vertex of 
 $e'$ and $e''$. Let $v_{j'}$ and $v_{j''}$ be the other end-vertex 
 of $e'$ and $e'$ respectively. Let $L_{j'}$ and $L_{j''}$ denote the 
 crossing domains corresponding to $v_{j'}$ and $v_{j''}$ respectively. 
 Note that $\varepsilon(j',j'')=1$. Under the 
 merging of $e'$ and $e''$ to a single edge $e$ representing the 
 merged crossing $p$-curve $\gamma$,  $v_{j'}$ and $v_{j''}$ 
 also are merged into a single vertex $v_j$. Since $\varepsilon(j,j)=1$, $\varepsilon(j',j'')$ is preserved upon 
 the merging of $e'$ and $e''$. Also, $\varepsilon(k,l)$ is preserved under the merging for $k, l\not\in \{ j',j''\}$, and $\varepsilon(k,j')=\varepsilon(k,j'')$ 
 is equal to $\varepsilon(k,j)$ after the merging. 
  Let $L_j$ denote the 
 crossing domain obtained by the corresponding merging of $L_{j'}$ and $L_{j''}$. Then $d_{L_j}=d_{L_{j'}}+d_{L_{j''}}$.  
The PLG-degree $d_{L_i}$ is preserved, and so are those of other 
crossing domains not having either $\gamma'$ or $\gamma''$ 
as their boundary components. Put all together we find that 
the $y$-index is invariant under the merging of two crossing 
$p$-curves. By reversing the process of merging, we see that 
$y$-index is also invariant under the splitting of a crossing 
$p$-curve.

Move on to Case (II). Let $L_i$ be a crossing domain 
of PLG-degree $d_{L_i}=0$, and $\pa L_i=\gamma$ is connected. 
Assume that $\gamma$ is about to be deformed to a point 
in $L_i$. If the genus of $L_i$ is greater than $0$, then 
$\gamma$ has to go through a finite number of splitting and 
merging stages before actually shrinks to a point. So 
without loss of generality we may assume that $L_i$ is of genus 
$0$. Then $L_i$ corresponds to an end-vertex $v_i$ (i.e, a vertex with 
only one edge attached to it), and $\gamma$ an end-edge $e$ (an 
edge with one of its vertices being an end-vertex). The death 
of $\gamma$ and hence of $L_i$ corresponds to removing 
the "dangling" $(v_i,e)$ pair with $d_{L_i}=0$, which does not affect the $y$-index at all. Similarly, the birth of a single 
crossing $p$-curve $\gamma$ (and a crossing domain of degree $0$ bounded by $\gamma$) does not alter the $y$-index. 

Proceed to the death of a pair of crossing $p$-curves 
$\gamma',\gamma''$ whose union is the boundary of an 
annular crossing domain $L_i$. Let $e'$ and $e''$ denote the edge corresponding to $\gamma'$ and $\gamma''$ respectively, 
$v_i$ the vertex corresponding to $L_i$. Note that $v_i$ is a 
common vertex of $e'$ and $e''$.

Assume for the moment that $e'$ and $e''$ have another common vertex which we denote as $v_j$. Then $\ell:=v_j\cup 
e'\cup v_i\cup e''\cup v_j$ is a loop with $v_j$ as the only 
intersection point with the rest of the graph. There is no other 
simple closed loop containing both $e'$ and $e''$ except for 
$\ell$. Also, any loop not containing both $e'$ and $e''$ 
contains neither $e'$ nor $e''$. The death of the pair $\gamma', 
\gamma''$ corresponds to the removing of $e'$, $v_i$ and $e''$ from the graph. This will not effect the evenness of any simple 
closed curve in the remaining graph. It will not change 
$\varepsilon (k,l)$ provided $i\not\in\{ k,l\}$. The 
genus of $L_j$ will increases by $1$ (since it "absorbs" $L_i$) 
but its PLG-degree $d_{L_j}$ remains the same since the degree of 
$L_i$ is $d_{L_i}=0$. The degrees of other crossing domains 
are unaffected as well. So the $y$-index is preserved as well. 

Now consider the case where $v_i$ is the unique common end-vertex of $e'$ and $e''$. Let $v_{j'}$ and $v_{j''}$ denote the 
other end-vertex of $e'$ and $e''$ respectively. Then $\varepsilon(j',j'')=1$. Any embedded path connecting $v',v''$ is 
even, either containing both $e',e''$ or none of them. 
In this case 
a simple closed loop of the graph either (i) contains the subgraph 
$v_{j'}\cup e'\cup v_i\cup e''\cup v_{j''}$ or (ii) contains none 
of $e'$, $e''$ and $v_i$. The death of the $(\gamma',\gamma'')$ pair 
corresponds to removing $e'$, $e''$ and $v_i$, and identifying 
$v_{j'}$ and $v_{j''}$ as a single vertex $v_j$ in the new graph. 
Since any simple closed curve satisfies either condition (i) or condition (ii) above, the evenness of simple closed loops is 
preserved for the new graph, so preserved is $\varepsilon (k,l)$ for $k,l\not\in\{ i, j',j''\}$, 
and we have $\varepsilon(k,j')=\varepsilon(k,j)=\varepsilon(k,j'')$, $\varepsilon(j',j'')=\varepsilon(j,j)$. Let $L_j$, $L_{j'}$ and $L_{j''}$ denote the crossing domains corresponding to $v_j$, $v_{j'}$ and $v_{j''}$ 
respectively. Then topologically $L_j$ is the union of $L_{j'}$, 
$L_i$ and $L_{j''}$, with its PLG-degree  $d_{L_j}=d_{L_{j'}}+d_{L_{j''}}$ since $d_{L_i}=0$. We conclude the $y$-index is invariant 
under the death as well as the birth of a pair of crossing $p$-curves. The proof is now complete.  
 \end{proof} 
 
\begin{cor} 
$y(L,q;\ff)$ is invariant under symplectic isotopies of $L$. 
\end{cor} 

Because there is no canonical choice of a reference crossing 
domain, it is desirable to consider a relative version of the 
$y$-index, which will be defined for a pair of Lagrangian surfaces provided they are identical near a regular point. 
Note that this overlapping condition can always be achieved by 
applying to one of the two Lagrangian surfaces a suitable Hamiltonian 
isotopy. 

Given two oriented Lagrangian surfaces $L,L'$, and suppose that $L$ and $L'$ overlap on an open  neighborhood $U$ of $p$, having the same orientation on $U$, the image $g'_{L}(U)=g'_{L'}(U)\subset \dP(K')$ is an open domain in $\dP(K')$. and there are crossing domains $L_{i_0}$ and 
$L'_{i_0}$ of $L$ and $L'$ respectively such that 
$U\subset L_{i_0}$ and $U\subset L'_{i_0}$. 
We can define a relative $y$-index 
\[ 
y(L',L,q;\ff):=\Big(\sum_{i\in I'} \epsilon(i_0,i)d'_i\Big)-
\Big(\sum_{j\in I} \epsilon(i_0,j)d_j\Big)
\] 
It is easy to see that $y(L',L, q;\ff)$ is independent 
of the choice of the common orientation of $L'$ and $L$ at $q$. On the other hand, a different choice of $q$ may change 
the $\pm$ sign of $y(L',L,q;\ff)$.

\begin{cor} 
$y(L,L',q;\ff_0)=y(L,L',q;\ff_1)$ if $\ff_0$ and $\ff_1$ are included 
in a continuous family $\ff_t$, $t\in [0,1]$, so that $q$ is 
regular (hence disjoint from any crossing $p$-curve) 
for $t\in[0,1]$. In particular, if is invariant under symplectic 
isotopies of each of $L,L'$ provided the overlapping condition 
is preserved. 
\end{cor}

%
%
\subsection{Effect of a $la$-disk surgery} \label{la-effect}

\begin{prop} 
The $la$-disk surgery preserves $\mu_2(L;\ff)$, i.e., 
\[ 
\mu_2(\eta_D(L);\ff)=\mu_2(L;\ff). 
\] 
\end{prop} 

\begin{proof} 
%

Without loss of generality we may adapt the standard model 
of $L$ near $D$ as well as the relevant notations as described in Section \ref{model}. 
Let 
\[ 
\gamma'(s):=M(\gamma(s))=(x_1=r\sin s, y_1=\sqrt{2}r+r\cos s), \quad 
\frac{3\pi}{4}\leq s\leq \frac{5\pi}{4}, 
\] 

Let $L':=\eta_D(L)$. Then 
\[ 
L'=(L\setminus Orb_\cG(\gamma))\cup Orb_\cG(\gamma').  
\] 
Denote 
\[ 
L_\gamma:=Orb_\cG(\gamma)\quad  \text{ and } \quad  L_{\gamma'}:=Orb_\cG(\gamma')
\] 
 following the notational convention in Section 
\ref{p-domain}. Both $L_\gamma$ and $L_\gamma'$ 
are $p$-domains (see Section \ref{p-domain}), hence $\mu_2(L_\gamma)$ and $\mu_2(L_{\gamma'})$ are defined. 

Up to a homotopy we may assume that $\ff=(J,\pa_{x_1},\pa_{x_2})$ near $D$, where $J$ is the standard complex 
structure on $\R^4\cong \C^2$. 

%
%

To show that $\mu_2$ is preserved under a $la$-disk surgery, one observes that $M$, which interchanges $L_\gamma$ and 
$L_{\gamma'}$, is a linear map (an element of $SO(4)$) 
satisfying the equality 
\[ 
M=-K''\circ g_\frac{\pi}{2} =K_{-\frac{\pi}{2}}\circ g_\frac{\pi}{2}
=g_\frac{\pi}{2}\circ K_{-\frac{\pi}{2}}. 
\] 
Indeed $M$ is anti-symplectic, $M^*J:=MJM^{-1}=MJM=-J$, 
$M^*K_t=K_{\pi-t}$. $M$ commutes with the action of $\cG$  and anti-commutes with the action of 
the  centralizers $\cC$, $M\circ c_\h=c_{-\h}\circ M$. 
Recall that $\cC$ and $\cG$ induces a trivialization of 
$\Lambda^+\cong S^1\times S^2$. Since $M$ preserves 
$\dP(K_{\pm\frac{\pi}{2}})$ and acts on which as 
$180^\circ$-rotations ($M=g_\frac{\pi}{2}$ on $\dP(K_{\pm\frac{\pi}{2}})$), we conclude that when acting on $\Lambda^+$, 
$M$ is orientation reversing on the $S^1$-factor, but orientation 
preserving on the $S^2$-factor. Since 
\[ 
g'_{L_{\gamma'}}=Mg'_{L_\gamma}, 
\] 
the maps $g'_{L_{\gamma'}}$ and $g'_{L_\gamma}$ have the 
same degree, i.e., 
\[ 
\mu_2(L_{\gamma'};\ff)=\mu_2(L_\gamma;\ff), \quad \text{hence} 
\] 
\[ 
\mu_2(L';\ff)=\mu_2(L;\ff). 
\]

\end{proof}

We turn to the effect of a $la$-disk surgery on $y(L,q;\ff)$. 
Some of the computational details are left in Section 
\ref{p-domain}. 

Recall from Section \ref{surg-la} that $L\cap U_D=L_\gamma$, $\gamma':=M(\gamma)$, and 
 $L'\cap U_D=L_{\gamma'}=M(L_\gamma)$. 
 
 Note that $\gamma$ and $\gamma'$ are polar curves in $\R^2_{x_1y_1}\setminus\{ (0,0)\}$ as considered in Section \ref{p-domain}.   Let $\check{\Gamma}_{L,\tau}$ and 
 $\check{\Gamma}_{L',\tau}$ denote the proper $E_\tau$-locus 
 of $L$ and $L'$ respectively.

 \begin{prop} \label{effect} 
Let $L$, $L'$, $\ff$ be as above. Fix a regular point $q\in L\setminus L_\gamma=L'\setminus L_{\gamma'}$.
Then 

\begin{equation} 
y(L',L,q;\ff) =\pm 4 
\end{equation} 
So $L'=\eta_D(L)$ and $L$ are not Hamiltonian isotopic. 
\end{prop}

 \begin{proof} 
 The choice of $q$ 
 determines a uniform orientation for each  of 
 $\check{\Gamma}^+_{L,\tau}$ and $\check{\Gamma}^+_{L',\tau}$, which also coincide on $L\setminus L_\gamma$ as oriented 
 $E_\tau$-loci.  View $\gamma\subset \check{\Gamma}^+_0$ 
 as a  curve oriented as a connected component of 
 $\Xi :=\check{\Gamma}^+_{L,0}\cap L_\gamma$, and similarly $\gamma'$ 
 is oriented as a connected component of $\Xi':=\check{\Gamma}_{L',0}\cap L_{\gamma'}$. Observe that the map $M|_\gamma:\gamma\to \gamma'$ is orientation preserving. Let $(\Delta\varphi)_\gamma$ and $(\Delta\varphi)_{\gamma'}$ denote the 
 relative $E_0$-phase of $\gamma$ and $\gamma'$ respectively. 
 Since $M(E_\tau)=E_{\tau^\perp}$ we have 
 \[ 
 (\Delta\varphi)_{\gamma'}=-(\Delta\varphi)_\gamma. 
 \] 
 From Section \ref{p-domain} we have 
 \[ 
 |(\Delta\varphi)_\gamma|=\pi, 
 \] 
 and the relative $E_0$-phases of $\Xi$ and $\Xi'$ are 
 \[ 
  (\Delta\varphi)_\Xi=4 (\Delta\varphi)_\gamma, \quad 
   (\Delta\varphi)_{\Xi'}=4 (\Delta\varphi)_{\gamma'}. 
 \] 
 The relative $y$-index $y(L',L,q;\ff)$ is then 
 \begin{equation} 
  \begin{split} 
 y(L',L,q;\ff) & =\frac{1}{2\pi}( (\Delta\varphi)_{\Xi'}-(\Delta\varphi)_\Xi)   = -\frac{4}{\pi} (\Delta\varphi)_\gamma \\ 
 & = \begin{cases} 
 -4 & \text{ if } (\Delta\varphi)_\gamma=\pi, \\ 
 4 & \text{ if } (\Delta\varphi)_\gamma=-\pi. 
 \end{cases} 
 \end{split}
 \end{equation} 
 This completes the proof. 
 \end{proof}

\section{Examples} 

\subsection{Lagrangian surfaces in $\R^4$} \label{R4}

Let $\R^4$ be endowed with the standard symplectic form 
$\omega=\sum_{i=1}^2dx_i\wedge dy_i$, and the standard 
complex structure $J$. We take $u:=\pa_{x_1}$ and $v:=\pa_{x_2}$ to be the unitary framing of $T\R^4$. Then the associated 
$E_\tau$, $\tau\in \R/\pi\Z$ are $E_\tau=(\cos\tau\pa_{x_1}+
\sin\tau\pa_{x_2})\wedge (\cos\tau\pa_{y_1}+\sin\tau\pa_{y_2})$. 
Note that all $\omega$-compatible unitary framings on $\R^4$ are 
homotopic.  The $\mu_2$ and $y$-indexes of $L$ are independent of 
the choice of a framing $\ff$, and will be denoted as 
$\mu_2(L)$ and $y(L,q)$.

\subsubsection{$S^1$-invariant $p$-domain} \label{p-domain}

We consider the following symmetric model of a Lagrangian 
annulus which will be viewed as a part of a closed Lagrangian surface. 

Let $\gamma:[a,b]\to \R^2_{x_1,y_1}\setminus\{ (0,0)\}$ be a parameterized immersed curve with parameter $s$. 
We assume that $\gamma$ 
also satisfies the following condition: there exists a finite sequence of numbers 
$a=s_0<s_1<s_2<\cdots <s_{k-1}<s_k=b$ such that 
\begin{enumerate} 
\item $\dot{\gamma}(s_i):=\frac{d\gamma}{ds}(s_i)$ is tangent to the position vector 
$\gamma(s_i)$ for $i=0,1,...,k$; 
\item $\dot{\gamma}(s)$ is transversal to $\gamma(s)$ if 
$s\neq s_i$, $i=0,1,...,k$. 
\end{enumerate} 
Let $\cG$ be the $S^1$-subgroup of $SU(2)$ as defined in Section \ref{surg-la}. Recall that  $g_\h\in\cG$ is the time 
$\h $-map  (mod $2\pi$) of the flow of the Hamiltonian 
vector field $X_G=x_1\pa_{x_2}-x_2\pa_{x_1}+y_1\pa_{y_2}-y_2\pa_{y_1}$. 
Then the orbit space  $L_\gamma:=Orb_\cG(\gamma)$ is an  immersed 
Lagrangian surface in $\R^4$. 

Parametrize $L_\gamma$ so that $(s,\h)$ are the 
coordinates of $g_\h(\gamma(s))$. 
 Also orient $L_\gamma$ so that $\dot{\gamma}(s)$ and $X_G(\gamma(s))$ form a positive basis.

Consider two smooth functions associated to $\gamma$: 
\begin{enumerate} 
\item 
$t=t(s)$ is the smooth function of $s$ such that 
$t(s)$ modulo $2\pi$ is the angle of the counterclockwise rotation from $\pa_{x_1}$ to $\gamma(s)$ or, equivalently, 
$t(s)$ (modulo $2\pi$) is the angular function of the position vector $\gamma$. 

\item 
$\varphi=\varphi(s)$ is the smooth function such that 
$\varphi(s)$ $\mod 2\pi$ is the angle of counterclockwise 
rotation from the position vector $\gamma(s)$ to the tangent vector $\dot{\gamma}(s)$. 

\end{enumerate} 
Note that 
\[ 
\varphi(s)\equiv 0 \mod \pi \quad \Leftrightarrow \quad 
\gamma(s)\parallel \dot{\gamma}(s). 
\] 

 We also extend $t$ and $\varphi$ 
$\cG$-invariantly over $L_\gamma$.

Along the coordinate curve $\h=0$, the tangent space $T_{\gamma(s)}L_\gamma$ is spanned by the ordered orthonormal vectors 
\[ 
v_1(s):=\dot{\gamma}(s)/|\dot{\gamma}(s)|, \quad 
v_2(s):=X_G(\gamma(s))/|X_G(\gamma(s))|
\] 
which, expressed 
as column vectors with respect to the basis $\{ \pa_{x_1},\pa_{x_2},\pa_{y_1},\pa_{y_2}\}$, are 
\[ 
\begin{split} 
\begin{pmatrix} v_1 & v_2\end{pmatrix} & = \begin{pmatrix}   \cos(t+\varphi)  & 0 \\ 0 & \cos t  \\ \sin(t+\varphi) & 0 \\ 0 
& \sin t
\end{pmatrix} \\ 
& = \begin{pmatrix} \cos\rho & 0 & 
-\sin \rho & 0 \\ 0 & \cos \rho & 0 & -\sin \rho \\ \sin \rho & 0 & \cos \rho & 0 \\ 
0 & \sin \rho & 0 & \cos \rho \end{pmatrix} 
\begin{pmatrix} 
 \cos(\frac{\varphi}{2}) & 0 \\ 0 & \cos(\frac{-\varphi}{2})  \\ 
\sin(\frac{\varphi}{2}) & 0 \\ 0 & \sin(\frac{-\varphi}{2}) 
\end{pmatrix}  \\ 
& = c_\rho \begin{pmatrix} v'_1 & v'_2\end{pmatrix},  
\end{split} 
\] 
where $\rho=\rho(s):=t(s)+\frac{\varphi(s)}{2}$.

Then for each $s$ 
\[ 
Orb_\cG(\dot{\gamma}(s)\wedge X_G(\gamma(s))) 
\subset \Lambda^+
\]  
is  a (possibly degenerated) 
circle  with multiplicity 2.  Moreover, it degenerates to a point 
precisely when $\gamma(s)\parallel \dot{\gamma}(s)$, i.e., 
when $s=s_i$ for some $0\leq i\leq k$.  A direct computation 
yields the following lemma.

\begin{lem} 
The closed curve $Orb_{\cG}(\gamma(s_i))$ is a 
crossing $p$-curve if $\gamma$ does not change its concavity 
at $\gamma(s_i)$, a folding $p$-curve if $\gamma$ changes 
its concavity at $\gamma(s_i)$ (i.e., $\gamma(s_i)$ is an 
infection point of $\gamma$). 
\end{lem}

Let $V_\tau:=(g'_{L_r})^{-1}(D_\tau)$, where $D_\tau\subset \dP(K')$ is the open hemisphere with $\pa D=\lambda^+_\tau$. 
To determine $V_\tau$ we compute the determinant of the 
$4\times 4$ matrix with column vectors $\dot{\gamma}_\h(s)$, $X_G(\gamma_\h(s))$, $\cos\tau \pa_{x_1}+\sin\tau\pa_{x_2}$, 
and $\cos\tau \pa_{y_1}+\sin\tau\pa_{y_2}$, where $\gamma_\h (s):=
g_\h(\gamma(s))$ and $\dot{\gamma}_\h(s):=\frac{\pa \gamma_\h}{\pa s}(s)$. The determinant is 
\begin{equation} \label{Dx1x2} 
\begin{split} 
\fD & :=\begin{vmatrix}   
\cos(t+\varphi)\cos\h & -\cos t\sin\h & 
\cos\tau & 0 \\ \cos(t+\varphi)\sin\h & \cos t\cos\h & \sin\tau & 0 \\ 
\sin(t+\varphi)\cos\h & -\sin t\sin\h & 0 & \cos\tau \\ 
\sin(t+\varphi)\sin\h & \sin t\cos\h & 0 & \sin\tau  
\end{vmatrix} \\ 
& =\frac{1}{2}\sin\varphi\sin(2(\h-\tau)). 
\end{split} 
\end{equation}  
Note that reversing the orientation of $\gamma$ (and 
hence the orientation of $L_\gamma$) results in  
an addition of $\pi$ to $\varphi$, which changes the sign of 
$\fD$. 

$\fD>0$ at a point $q\in L_\gamma$ iff $q\in V_\tau$. 
Also $q\in \Gamma_\tau$, $\tau\in\R/\pi\Z$ iff $\varphi=0 \mod \pi$ or $\h=\tau+\frac{(k-1)\pi}{2}$, $k=1,2,3,4$. So on  $L_\gamma$  
the proper $E_\h$-locus $\check{\Gamma}_\tau$ consists of four embedded 
arcs, the union of $\check{\Gamma}_\tau$ form a smooth 1-dimensional 
foliation of $L_\gamma$. Note that  $L_\gamma$ is $\cG$-invariant, $g_\h(\check{\Gamma}_\tau)=\check{\Gamma}_{\tau+\h}$.

Take $\tau=0$ and then $\fD>0$ iff $\sin\varphi\sin 2\h>0$, 
i.e., iff one of the two conditions holds: 
\begin{enumerate} 
\item $\varphi\in (0,\pi)$ and $\h\in(0,\frac{\pi}{2})\cup (\pi,\frac{3\pi}{2}) \mod 2\pi$; 
\item   $\varphi\in(\pi,2\pi)$ and 
$\h\in(\frac{\pi}{2},\pi)\cup(\frac{3\pi}{2},2\pi) \mod 2\pi$. 
\end{enumerate}

{\bf Example of Case (i):} \ Let $(r,t)$ be polar coordinates of $\R^2$ so that $x_1=r\cos t$, $y_1=r\sin t$.
Take $\gamma$ to be a 
polar curve $r(t)>0$ with $a\leq t\leq b$ and  parametrize 
$\gamma$ by $s=t$, so that $\gamma(s)$ is 
 tangent to $\dot{\gamma}(s)=\frac{d\gamma}{ds}(s)$ iff $s=a$ or $b$. In particular, $L_\gamma$ is a p-domain. For 
 such $\gamma$ we have 
 $\varphi\in (0,\pi) \mod 2\pi$ except at endpoints of $\gamma$, and $\Delta\varphi=\pm\pi$ or $0$. 
 Let $V_{0,k}:=\{ (t,\h)\mid a<t<b, \ \frac{(k-1)\pi}{2}<\h<\frac{k\pi}{2}\}$, $k=1,2,3,4$. 
 Let $\gamma_\h:=g_\h(\gamma)$ 
with the induced orientation, and  $-\gamma_\h$ denote 
$\gamma_\h$ but with the reversed orientation.

Here $V_0=V_{0,1}\cup V_{0,3}$. Then $\check{\Gamma}_0\subset 
 \pa V_0$ with the boundary orientation is 
  \begin{equation} \label{Gamma1} 
 \check{\Gamma}^+_0=\gamma \cup -\gamma_\frac{\pi}{2} \cup \gamma_\pi \cup -\gamma_\frac{3\pi}{2}. 
 \end{equation}  
 That $g_\h(E_\tau)=E_{\tau+\h}$ ensures that the 
 four components of $\check{\Gamma}^+_0$ have the same 
 relative $E_0$-phase (the angle 
of rotation in $E_0$ minus the angle of rotation in $E_\frac{\pi}{2}$) which is $\Delta\varphi$, so 
$\alpha_{0,1}=2\Delta\varphi=\alpha_{0,3}$, and hence 
\begin{equation}  \label{mu2Lgamma1}
\mu_2(L_\gamma):=d_{L\gamma}=\frac{1}{2\pi}(\alpha_{0,1}+\alpha_{0,3})=
\frac{1}{2\pi}(4\Delta\varphi)=\frac{2}{\pi}\Delta\varphi. 
\end{equation} 
See Figure \ref{Delta1}  for examples with $\Delta\varphi=\pm\pi$ or $0$. 


\begin{figure}[th] 
 \begin{center} 
 \includegraphics[scale=0.65]{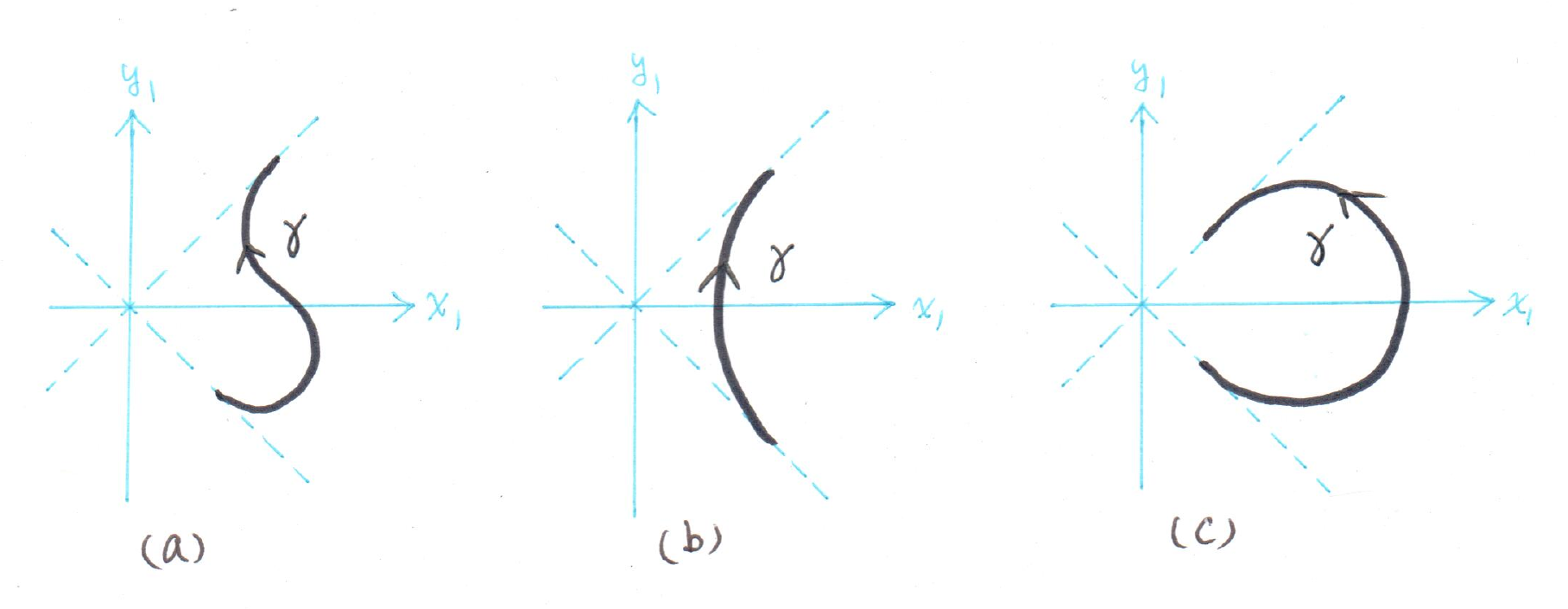} 
\caption{From (a) to (c): $\Delta\varphi=0,-\pi,\pi$; $d_{L_\gamma}=0, -2,2$.   \label{Delta1} }
\end{center} 
\end{figure} 

Note that $c_\tau(L_\gamma)$ and $L_\gamma$ have the same 
$\mu_2$-index (in fact they are Hamiltonian isotopic) and the 
precise value of $\mu_2$-index depends only on the variation 
$\Delta\varphi:=\varphi(b)-\varphi(a)$ but not on $\Delta t:=
t(b)-t(a)$.

\begin{figure}[th] 
 \begin{center} 
 \includegraphics[scale=0.65]{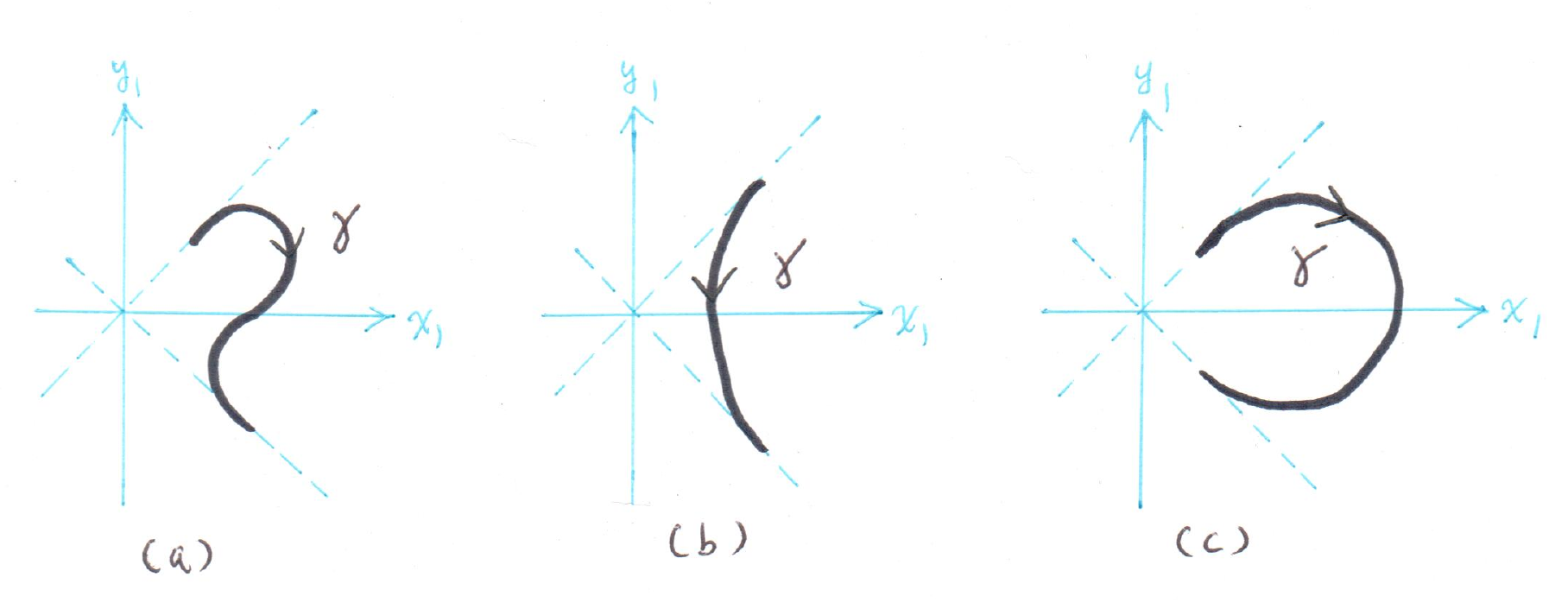} 
\caption{From (a) to (c): $\Delta\varphi=0,\pi,-\pi$; $d_{L_\gamma}=0, -2,2$.   \label{Delta2} }
\end{center} 
\end{figure}

{\bf Example for Case (ii):} \  Let $\gamma$ be as in the above 
example for Case (i) but {\em with the opposite orientation}, then 
 $\varphi\in (\pi,2\pi) \mod 2\pi$ except at endpoints of $\gamma$. Here 
 $V_0=V_{0,2}\cup V_{0,4}$ and 
 $\check{\Gamma}_0\subset 
 \pa V_0$ with the boundary orientation is 
  \begin{equation} \label{Gamma2} 
 \check{\Gamma}^+_0=-\gamma \cup \gamma_\frac{\pi}{2} \cup -\gamma_\pi \cup \gamma_\frac{3\pi}{2}. 
 \end{equation} 
 The relative $E_0$-phase along each of the 
 four components of $\check{\Gamma}^+_0$ is 
 $-\Delta\varphi$, 
$\alpha_{0,2}=-2\Delta\varphi=\alpha_{0,4}$, and hence 
\begin{equation} \label{mu2Lgamma2}
\mu_2(L_\gamma):=d_{L\gamma}=\frac{1}{2\pi}(\alpha_{0,2}+\alpha_{0,4})=
\frac{1}{2\pi}(-4\Delta\varphi)=-\frac{2}{\pi}\Delta\varphi. 
\end{equation} 

See Figure \ref{Delta2}  for examples with $\Delta\varphi=\pm\pi$ or $0$.

\begin{rem} 
{\rm
Note that in the above examples of $L_\gamma$, the orientation 
of $\check{\Gamma}^+_0$ (and hence of all $\check{\Gamma}^+_\tau$) does not depend on the orientation of $\gamma$. 
} 
\end{rem}

\subsubsection{Whitney sphere and tori} \label{WhTo}

{\bf Lagrangian Whitney sphere.} \ 
Consider the  polar curve $\gamma$  in $\R^2_{x_1,y_1}$ defined by 
\[ 
y_1^2-\frac{x_1^2}{2}+\frac{x_1^4}{4}=0, \quad x_1\geq 0. 
\] 
We orient $\gamma$ counterclockwise and parameterize it 
by its angular coordinate $t$ with $-\arctan \frac{1}{\sqrt{2}}\leq 
t\leq \arctan\frac{1}{\sqrt{2}}$. 
The immersed Lagrangian surface $L_\gamma:=Orb_\cG(\gamma)$ is a Whitney sphere, an immersed sphere with one 
transversal positive self-intersection. We orient $L_\gamma$ 
so that $\dot{\gamma}$ and $X_G(\gamma)$ form a positive 
basis of $T_\gamma L_\gamma$ for $t\neq \pm\arctan\frac{1}{\sqrt{2}}$. 

Note that $\gamma$ resembles the curve in Figure \ref{Delta1}(c) except that here the endpoints of $\gamma$ meet at 
$(0,0)$, Still we have $\Delta \varphi =\pi$ along $\gamma$. 
There are no crossing $p$-curves in $L_\gamma$. Take $q$ to be 
any regular point of $L_\gamma$, then by (\ref{mu2Lgamma1}) 
\[ 
\mu_2(L_\gamma)=2=y(L_\gamma,q). 
\] 

%
%

{\bf Chekanov torus.} \ 
Consider the closed curve $\upsilon:=\sigma\cup \gamma$ in $\R^2_{x_1,y_1}$: 
\begin{align*} 
\sigma(s) & :=(x_1=\sqrt{2}r+r\cos s, y_1=r\sin s), \quad 
\frac{-3\pi}{4}\leq s\leq \frac{3\pi}{4} , \\ 
\gamma(s) & :=(x_1=\sqrt{2}r+r\cos s, y_1=r\sin s), \quad 
\frac{3\pi}{4}\leq s\leq \frac{5\pi}{4},  
\end{align*} 
where $r>0$ is a constant. 
The orbit 
\[ 
L_{\upsilon}:=Orb_\cG(\upsilon) 
\] 
is a Chekanov torus. 

Observe that $\sigma$ is as in Figure \ref{Delta1}(c) 
with $(\Delta\varphi)_\sigma=\pi$, and 
$\gamma$ is as in Figure \ref{Delta2}(b) with $(\Delta\varphi)_\gamma=\pi$. Both $L_\sigma$ and 
$L_\gamma$ are crossing domains of $L_\upsilon$. Let 
$q$ be a regular point in $L_\sigma$, then 
\begin{align*} 
\mu_2(L_\upsilon) & = d_{L_\sigma}+d_{L_\gamma}= 
\frac{2}{\pi}( (\Delta\varphi)_\sigma-(\Delta\varphi)_\gamma))=
\frac{2}{\pi}(\pi-\pi)=0,  \\ 
y(L_\upsilon, q) & = d_{L_\sigma}-d_{L_\gamma}= 
\frac{2}{\pi}( (\Delta\varphi)_\sigma+(\Delta\varphi)_\gamma))=
\frac{2}{\pi}(\pi+\pi)=4.  
\end{align*}

%
%
%
%

{\bf Monotone Clifford torus.} \ 
Let $\gamma':=M(\gamma)$,  $\upsilon':=\sigma\cup\gamma'$. Then $L_{\upsilon'}=L_\sigma\cup L_{\gamma'}\subset \R^4$ is a monotone 
Clifford torus. Observe that, up to a rotation by some 
$c_\tau\in\cC$, $\gamma'$ resembles the model 
curve in Figure \ref{Delta1}(b) with $(\Delta\varphi)_{\gamma'}=-\pi$. Note that both $L_\sigma$ and $L_{\gamma'}$ are folding 
domains of $L_{\upsilon'}$. Let $q$ be a regular point in $L_\sigma$, then 
\[ 
y(L_{\upsilon'},q)=\mu_2(L_{\upsilon'}) = d_{L_\sigma}+d_{L_{\gamma'}}= 
\frac{2}{\pi}( (\Delta\varphi)_\sigma+(\Delta\varphi)_{\gamma'}))=
\frac{2}{\pi}(\pi-\pi)=0. 
\]

%
%
This in particular implies that the monotone Clifford torus $L_{\upsilon'}$ 
is not Hamiltonian isotopic to the Chekanov torus $L_\upsilon$.

{\bf General Clifford torus.} \ 
There is another way to describe a monotone Clifford torus. 
A (not necessarily monotone) Clifford torus is defined as 
\[ 
T_{a,b}:=\{ x_1^2+y_1^2=a, \ x_2^2+y_2^2=b\} \subset \R^4, 
\] 
where $a,b>0$ are constants. Then $T_{a,b}$ is monotone 
iff $a=b$. It is easy to see that with respect to the standard 
complex structure $J$ and the unitary framing $\pa_{x_1},\pa_{x_2}$, the image in $\dP(K')$ of $T_{a,b}$ under the 
map $g'_{T_{a,b}}$ degenerates to a circle. Hence 
$\mu_2(T_{a,b})=0$ for any $a,b>0$. 
Also $y(T_{a,b},q)=0$ because $g'_{T_{a,b}}$ is not surjective.

\subsubsection{Fibers of an integrable system} \label{integrable}
  
It is possible to construct an integrable system 
whose Lagrangian fibers include a Whitney sphere, Chekanov 
tori, as well as Clifford tori. 
 In this section we study the $\mu_2$ index and the $y$-index of Lagrangian fibers of such 
 an integrable system in $\R^4$ defined by a pair of commutative 
 Hamiltonian functions $G,H$ where 
 \begin{eqnarray*}  
 G(x_1,y_1,x_2,y_2) & := & -x_1y_2+x_2y_1, \\ 
H(x_1,y_1,x_2,y_2) & := & y_1^2+y_2^2-\frac{1}{2}(x_1^2+x_2^2)+\frac{1}{4}(x_1^2+x_2^2)^2. 
\end{eqnarray*} 
On can check that the Poisson bracket $\{ H,G\}=0$, in 
other words, the Lie bracket of the corresponding Hamiltonian 
vector fields is $[X_H,X_G]=0$. We remark here that a similar  integrable system was considered by 
Bates in \cite{Ba1}.  The example considered by Auroux in 
\S 5.1 of \cite{Aur1} is also topologically equivalent the one 
we consider here. Interested readers are advised to consult 
the references for more detail. Here we only state what are 
relevant to our purpose.

The pair $(G,H)$ defines a map by assigning to 
$p\in\R^4$ its $(G,H)$-values. 
\[ 
\cZ :\R^4\to \R^2, \quad \Xi (p):=(G(p), H(p)). 
\] 
Let $\cR$ denote the range of $\cZ$. It is an unbounded 
closed domain of $\R^2$. The boundary $\pa\cR$ is a smooth 
curve containing the point $(0,-\frac{1}{4})$ where $-\frac{1}{4}$ 
is the minimum value of $H$. The two vector fields $X_G$ and 
$X_H$ are linearly independent except at $\cZ^{-1}(\pa\cR)$ 
where $X_H$ and $X_G$ are nonvanishing but linearly dependent, 
and at the origin $0\in\R^4$ at which $X_G=0=X_H$. 
Both $X_G$ and $X_H$ are tangent to fibers of $\cZ$. 
Let $L_{a,b}:=\cZ^{-1}(a,b)$ denote the fiber of 
$\cZ$ over $(a,b)\in\cR$. For  $(a,b)\in\pa\cR$, $L_{a,b}$ is degenerate and is diffeomorphic to $S^1$. In addition, we 
have the following observation concerning the Lagrangian 
fibers over interior points of $\cR$:

\begin{fact} 
Let $int(\cR)$ denote the interior of $\cR$. Then there are 
four types of Lagrangian fibers over $(a,b)\in int(\cR)$: 
 \begin{enumerate} 
\item $L_{a,b}$ is a Chekanov torus if $a=0$ and $-\frac{1}{4}<b<0$.  
\item $L_{a,b}$ is a monotone Clifford torus if $a=0$ and 
$b>0$.   
\item $L_{a,b}$ is a Lagrangian Whitney sphere if $(a,b)=(0,0)$.   
\item $L_{a,b}$ is a non-monotone embedded Lagrangian torus 
for all $(a,b)\in int(\cR)$ with $a\neq 0$. 
\end{enumerate} 
\end{fact} 

Recall that $X_G$ generates the $\cG$-action on $\R^4$, with 
$g_\h$ as the time $\h$ map of the flow of $X_G$. Thus all 
$L_{a,b}$ are $\cG$-invariant.

We have the following result concerning the $y$-index of 
$L_{a,b}$ for $(a,b)\in\cR^\circ$. 

\begin{prop} \label{y-int}
Let $(a,b)\in int(\cR)$. Then the pair 
\[ 
(\mu_2(L_{a,b}), \ \bar{y}(L_{a,b}))=\begin{cases} 
(0,4) & \text{ if $a=0$ and $-\frac{1}{4}<b<0$}, \\ 
(2,2) & \text{ if $a=0=b$}, \\ 
(0,0) & \text{ if $a=0$ and $B>0$}, \\ 
(0,0) & \text{ if $a\neq 0$.}
\end{cases} 
\] 
\end{prop}

\begin{proof} 
Consider the case of $a=0$ at first. 
Let $h:\R^2_{x_1,y_1}\to \R$, 
\[ 
 h(x_1,y_1):=y_1^2-\frac{1}{2}x_1^2+\frac{1}{4}x_1^4, 
\] 
denote the restriction of $H$ to 
$\R^2_{x_1,y_1}$. Then $L_{0,b}=L_{\gamma_b}:=Orb_\cG(\gamma_b)$ where 
\[ 
\gamma_b=h^{-1}(b)\cap \{ x_1\geq 0\}. 
\] 
By analyzing level curves of $h$ it is easy to see that 
$L_{0,b}$ is Hamiltonian isotopic to a Chekanov torus if 
$b<0$, a Whitney sphere if $b=0$ and a monotone Clifford 
torus  if $b>0$. So it remains to prove the case of $a\neq 0$.

Let $(a,b)\neq (0,0)$ be an interior point of of $\cR$. Then 
$TL_{a,b}=\text{Span}\{ X_G,X_H\}$ where, with respect to 
the basis $\{ \pa_{x_1},\pa_{x_2},\pa_{y_1},\pa_{y_2}\}$, 
\begin{align*} 
X_G  & = (-x_2,x_1,-y_2,y_1)^T,  \\ 
X_H & =(-2y_1, -2y_2, x_1(x_1^2+x_2^2-1),x_2(x_1^2+x_2^2-1))^T.  
\end{align*} 
Recall the $\cG$-invariant $S^1$-family of complex planes $E_\theta$. Since $L_{a,b}$ is also $\cG$-invariant, the locus 
$\Gamma_\h\subset L_{a,b}$ of $E_\h$ is $g_\h(\Gamma_0)$, 
it is enough to analyze the intersection subspaces along $\Gamma_0$. Since $E_0=\pa_{x_1}\wedge \pa_{y_1}$, 
$p:=(x_1,y_1,x_2,y_2)\in \Gamma_0$ iff it satisfies 
the defining equations of $L_{a,b}$: $G=a$, $H=b$, and 
the determinant of the $4\times 4$ matrix 
$\begin{pmatrix} X_H & X_G & \pa_{x_1} & \pa_{y_1}
\end{pmatrix}$ vanishes at $p$, i.e., 
\begin{equation} 
 x_1x_2(x_1^2+x_2^2-1)+2y_1y_2=0.  
\end{equation}

Consider two vector fields $Z_1,Z_2$ on $\R^4$ defined by 
\[ 
Z_1=(0,x_1,0,y_1)^T, \quad (x_2,0,y_2,0)^T. 
\] 
Observe that both $Z_1$ and $Z_2$ are $\cG$-invariant, 
$Z_1-Z_2=X_G$, their inner product is $Z_1\cdot Z_2=0$, and 
$Z_1,Z_2$ are linearly independent precisely at points where 
$x_1^2+y_1^2>0$ and $x_2^2+y_2^2>0$. In particular this 
include all points with $a\neq 0$. We have 
\begin{align*} 
dG(Z_1)=dG(Z_2) & = 0 \\ 
dH(Z_1)=dH(Z_2) & =x_1x_2(x_1^2+x_2^2-1)+2y_1y_2. 
\end{align*} 
So for $p\in L_{a,b}$ and $i=1,2$, 
\[ 
Z_i(p)\in T_pL_{a,b} \quad  \Leftrightarrow \quad p\in \Gamma_0=\Gamma_{\frac{\pi}{2}}, 
\] 
and in this case, 
\[  
Z_2(p)\in T_pL_{a,b}\cap E_0, \quad Z_1(p)\in T_pL_{a,b}\cap E_{\frac{\pi}{2}}. 
\] 

Since $L_{a,b}$ are smoothly isotopic as Lagrangian surfaces for 
all $(a,b)\neq (0,0)$, and $\mu_2$-index is invariant under 
regular homotopy of Lagrangian surfaces, we have 
$\mu_2(L_{a,b})=0$ for all $(a,b)\neq (0,0)$. 
To determine $\bar{y}(L_{a,b})$ with $a\neq 0$, 
observe that $a=x_2y_1-x_1y_2\neq 0$, hence along $\Gamma_0$ the angle of $Z_2$ in $E_0$ (with respect to 
$\pa_{x_1}$) minus the angle of $Z_1$ in $E_{\frac{\pi}{2}}$ 
(with respect to $\pa_{x_2}$) is always 
\begin{enumerate} 
\item  greater than $0$ but smaller than $\pi$-radians 
if $a>0$, 
\item greater than $-\pi$ but smaller than $0$-radians if 
$a<0$. 
\end{enumerate} 
So the entire $T_{a,b}$ is a crossing domain 
and moreover the map $g'_{T_{a,b}}:T_{a,b}\to \dP(K')$ is not 
surjective. Therefore 
\[ 
\bar{y}(T_{a,b})=y(T_{a,b})=\mu_2(T_{a,b})=0. 
\] 
This completes the proof of Proposition \ref{y-int}. 
\end{proof}

%
%
\subsection{Sphere in $T^*S^2$} \label{T*S}

Let $S$ denote the zero section of the cotangent bundle 
$T^*S^2$ of a sphere, $T^*S^2$ is endowed with the standard 
symplectic structure $\omega$. Note that $T^*S^2$ is symplectically 
parallelizable, and $H_1(T^*S^2,\Z)=0=H_3(T^*S^2,\Z)$, so 
its $\cF^\omega$ is connected, and $\mu_2(S,\ff)$ is 
independent of $\ff\in \cF^\omega$.


We identify $S=D\cup \bar{D}$ as the union of two closed Lagrangian 
disks $D$ and $\bar{D}$. 
Parametrize $D$ and $\bar{D}$ as 
\begin{align*} 
D & =\{ (x_1,x_2)\in \R^2\mid x_1^2+x_2^2\leq 1\}, \\ 
\bar{D} & =\{ (\bar{x}_1,\bar{x}_2)\in\R^2\mid \bar{x}_1^2+\bar{x}_2^2\leq 1\} 
\end{align*} 
We also write $x_1=r\cos\h$ and $x_2=r\sin\h$ where 
$(r,\h)$ are polar coordinates of $\R^2_{x_1x_2}$. Similarly 
$\bar{x}_1=\br\cos\bth$ and $\bx_2=\br\sin\bth$ where 
$(\br,\bth)$ are polar coordinates of $\R^2_{\bx_1\bx_2}$.

Consider the diffeomorphism $\psi:\bar{D}\setminus\{ 0\} \to 
D\setminus\{ 0\}$ 
\begin{equation} \label{bDD} 
\psi(\bx_1,\bx_2)=\Big(x_1=\frac{-\bx_1}{\bx_1^2+\bx_2^2}, 
x_2=\frac{\bx_2}{\bx_1^2+\bx_2^2}\Big), \quad 0<\bx_1^2+\bx_2^2\leq 1. 
\end{equation} 
In polar coordinates $\psi(\br,\bth)=(r=\frac{1}{\br},\h=\pi-\bth)$. $\psi$ induces a symplectomorphism 
\begin{align*} 
\Psi & : T^*(\bar{D}\setminus\{ 0\})\to T^*(D\setminus\{ 0\}) \\ 
\Psi(\bx_1,\bx_2,(\by_1,\by_2)^T) & = (\psi(\bx_1,\bx_2), 
(\psi_*^{-1})^*_{(\bx_1,\bx_2)}(\by_1,\by_2)^T). 
\end{align*} 
The differential $\Psi_*$, in matrix form with respect to the 
bases $\{ \pa_{\bx_1},\pa_{\bx_2},\pa_{\by_1},\pa_{\by_2}\}$ 
and $\{ \pa_{x_1},\pa_{x_2},\pa_{y_1},\pa_{y_2}\}$ is 
\begin{align*}  
\Psi_*|_{(\bx,\by)} & =\begin{pmatrix} B & 0 \\ 0 & (B^{-1})^T 
\end{pmatrix} \in Sp(4,\R), \quad \text{where}  \\ 
B & =\begin{pmatrix} x_1^2-x_2^2 & -2x_1x_2 \\ 
2x_1x_2 & x_1^2-x_2^2 \end{pmatrix} =r^2\begin{pmatrix} 
\cos 2\h & -\sin 2\h \\ \sin 2\h & \cos 2\h\end{pmatrix}. 
\end{align*}  
Note that $B$ and hence $\Psi_*$ depend only on $x=(x_1,x_2)$, not 
on $y=(y_1,y_2)$. 
The symplectic manifold $T^*S^2$ can be identified with the 
union $T^*D\cup_{\Psi}T^*\bD$.

We  identify each of $T^*D$ and $T^*\bD$ as open 
domains in $T^*\R^2\cong \C^2$. Let 
$X_i:=\pa_{x_i}$ and $\bX_i:=\pa_{\bx_i}$ for $i=1,2$. 
On $T^*D$ the ordered pair $( X_1,X_2)$ is a unitary framing with 
respect to the standard complex structure $J$, and similarly 
$( \bX_1,\bX_2)$ is a unitary framing on $T^*\bD$ with 
respect to the standard complex structure which is now 
denoted as $\bJ$ to suggest its affinity with $T^*\bD$. 
Note that  $\Psi^*\bar{J}= J$  on $T^*_{\pa D} D
=\{ (x,y)\in \R^2\times \R^2\mid |x|=1\}$, $J$ and $\bJ$ 
match along $T^*_{\pa D} D$ to form a $\omega$-compatible 
almost complex  structure denoted as $J$ on $T^*S^2$. 
However  the two unitary 
bases do not: 
\[ 
\begin{pmatrix}\bX_1 & \bX_2\end{pmatrix}=
\begin{pmatrix} \cos 2\h & -\sin 2\h\\ \sin 2\h & \cos 2\h 
\end{pmatrix} \begin{pmatrix}X_1 & X_2\end{pmatrix} \quad 
\text{ on $T^*_{\pa D}D$}. 
\]

{\bf Unitary framing $(u,v)$.} \ 
We will modify $X_1,X_2$ on $T^*D$ to get a  unitary framing 
$u,v$ on $T^*S^2$ so that $u=\bX_1$ and $v=\bX_2$ on $T^*\bD$. Observe that $B|_{r=1}$ is twice of the contractible loop 
\[ 
G_0(\vh):=\begin{pmatrix} \cos \vh & -\sin \vh\\ \sin \vh & 
\cos \vh \end{pmatrix} \in SU(2) \cong S^3, \quad \vh\in\R/2\pi\Z. 
\] 
 Moreover, any two homotopies between 
$G_0$ and the constant loop $Id$ in $SU(2)$ are homotopic. 
Construct a homotopy between $G_0$ and $Id$ as follows: 
Let 
\[ 
H(t):=\begin{pmatrix} \cos t & i\sin t  \\ i\sin t & \cos t\end{pmatrix} \in SU(2)
\] 
and define 
\begin{equation} \label{Gth}  
G_t(\vh):=H(t)G_0(\vh)H(t)^{-1}. 
\end{equation}  
Note that $G_{t+\pi}(\vh)=G_t(\vh)$, $G_{t+\frac{\pi}{2}}(\vh)=
G_t(-\vh)=G_t(\vh)^{-1}$. 
In matrix form with respect to the basis $\{\pa_{x_1},\pa_{x_2},\pa_{y_1},\pa_{y_2}\}$ 
\[ 
G_t(\vh)=\begin{pmatrix} \cos\vh & -\sin\vh\cos 2t & 
-\sin\vh\sin 2t  & 0 \\ 
 \sin\vh\cos 2t & \cos\vh &  0 & \sin\vh\sin 2t  \\ \sin\vh\sin 2t  & 0 & \cos\vh & -\sin\vh\cos 2t 
 \\  0 & -\sin\vh\sin 2t  & \sin\vh\cos 2t & \cos\vh 
  \end{pmatrix}.   
\] 
The set 
\[ 
\cD:=\{ G_t(\vh)\mid \vh\in [0,\pi], \ t\in [0,\frac{\pi}{2}] \} \subset 
SU(2) 
\] 
is a hemisphere with boundary the group $G_0$. 

For the moment let us consider the diffeomorphism 
\begin{align*} 
f & : \cD \to D \\ 
f(G_t(\vh)) & :=(x_1=\cos\vh, x_2=\sin\vh\cos 2t) 
\end{align*}  

Consider a new unitary framing $u',v'$ over $D$ so that if 
$f^{-1}(x)=G_t(\vh)$ for $x\in D$, then 
\begin{equation}  \label{u'v'} 
u'_x  :=G_t(\vh)(\pa_{x_1}), \quad 
v'_x  :=G_t(\vh)(\pa_{x_2}).  
\end{equation}  
Note that $u'\wedge v'\in \dP(K')$. 
Observe that $\pa D=f(\{ t=0,\frac{\pi}{2}\})$. 
Recall $x=(x_1,x_2)=(r\cos\h, r\sin\h)$, then 
along $f(\{ t=0\})=\{ r=1, \h\in [0,\pi]\}$, 
$\vh=\h$, and 
\begin{align*} 
u'_x & =\cos\h \pa_{x_1}+\sin\h\pa_{x_2}  \\ 
v'_x & =-\sin\h\pa_{x_1}+\cos\h\pa_{x_2}. 
\end{align*} 
Similarly, along $f(\{ t=\frac{\pi}{2}\})=\{ r=1, \ \h\in[-\pi, 0]\}$, 
$\vh=-\h$ and
\begin{align*} 
u'_x & =\cos\h \pa_{x_1}+\sin\h\pa_{x_2} \\ 
v'_x & =-\sin\h\pa_{x_1}+\cos\h\pa_{x_2}. 
\end{align*} 
Also, along the diameter $x_2=0$ and at $x_1=\cos\vh$, 
\begin{equation} \label{x2=0} 
u'_x=\cos\vh\pa_{x_1}+\sin\vh\pa_{y_1}, \quad 
v'_x=\cos\vh\pa_{x_2}-\sin\vh\pa_{y_2} . 
\end{equation} 

Now consider the composite map 
\begin{equation} \label{F} 
F:=f^{-1}\circ  \phi:D\to \cD 
\end{equation} 
where 
\begin{equation} \label{z2} 
\phi:D\to D,  \quad \phi(x_1,x_2):=(x_1^2-x_2^2,2x_1x_2)
\end{equation} 
is the standard 
branched $2:1$ covering map of degree $2$ with $(0,0)$ as 
the only branching point.  $F(x_1,x_2)=G_t(\vh)$ if 
\begin{equation} \label{x1x2} 
x_1^2-x_2^2=\cos\vh, \quad 2x_1x_2=\sin\vh \cos 2t. 
\end{equation}

For $x\in D$ define 
\begin{equation} 
u_x:=G_t(\vh)(\pa_{x_1}), \quad 
v_x:=G_t(\vh)(\pa_{x_2}), \quad \text{ if  
$F(x)=G_t(\vh)$}. 
\end{equation}  
 A direction computation shows that along $\pa D$, 
 \begin{align*} 
 u_x & = \cos 2\h\pa_{x_1}+\sin 2\h\pa_{x_2} = \pa_{\bx_1}, \\
 v_x & = -\sin 2\h\pa_{x_1}+\cos 2\h\pa_{x_2} =\pa_{\bx_2}. 
 \end{align*} 
Extend $( u,v)$ over $\bar{D}$ by setting 
\[ 
u=\pa_{\bx_1}, \quad v=\pa_{\bx_2} \quad \text{ on } \bar{D}, 
\] 
then further extend $(u,v)$ over $T^*S$ by the pullback map of the 
canonical projection $\pi:T^*S\to S=\bar{D}\cup_\psi D$. 
Now the extended $( u,v)$ is a unitary framing on $T^*S$.

{\bf $E'_s$ associated to $( u,v)$.} \ 
Let $\cE':=\{ E'_s\mid s\in\R/\pi\Z\}$ denote the 
$S^1$-family of complex line bundles associated to the 
unitary framing $(u,v)$. 
Recall the $S^1$-family of complex line bundles $\cE:=\{ E_s\mid s\in \R/\pi\Z\}$, associated to the unitary framing $\pa_{x_1},\pa_{x_2}$. 
Then at $(x,y)\in T^*D$, 
\[ 
E'_s|_{(x,y)}=G_t(\vh)(E_s|_{(x,y)}) \quad \text{ if $F(x)=(G_t(\vh))$}. 
\] 
Since $TD$ is spanned by $\pa_{x_1},\pa_{x_2}$, for $x\in D$, 
there exists $s'\in \R/\pi\Z$, $s'$ depends on $x$, such that 
\begin{equation} 
\{ E'_s,E'_{s^\perp}\} =\{ E_{s'}, E_{{s'}^\perp}\} \quad \text{ at $x\in D$}. 
\end{equation} 

To solve for $s,s'$ with 
\begin{equation}  \label{s's} 
E_{s'}=G_t(\vh)(E_s),
\end{equation} 
observe that the two fixed points of the action of $H$ on $\dP(J)$ are 
\[ 
E_{\frac{\pi}{4}}=(\pa_{x_1}+\pa_{x_2})\wedge (\pa_{y_1}+\pa_{y_2}), \quad E_{\frac{3\pi}{4}} 
=E_{{\frac{\pi}{4}}^\perp}=(-\pa_{x_1}+\pa_{x_2})\wedge (-\pa_{y_1}+\pa_{y_2)}. 
\] 
$H(t)$ rotates the 
tangent plane $T_{E_\frac{\pi}{4}}\dP(J)$ clockwise by an angle of $2t$-radians. We take the fixed points of the $G_0$-action on $\dP(J)$ as the poles of $\dP(J)$. Then $\cE\subset \dP(J)$ is the 
equator. Assume that $t,s\in (0,\frac{\pi}{4})$ for the moment. 
Consider the right spherical triangle in $\dP(J)$ cut out 
by $\cE$, $H(-t)(\cE)$ and the longitude $\ell_{t,s}$ passing through the point $H(-t)(E_s)$. The point $H(-t)(E_s)$ is on 
a latitude $m_{t,s}$ (an $G_0$-orbit) of $\dP(J)$. 
Let $E_\tau$ denote the intersection point of $\ell_{t,s}$ with 
$\cE$. $E_\tau$ is  the right angle vertex of the triangle, 
$0<s<\tau$. 
Scale $\dP(J)$ to make it a sphere of radius 1. 
Then the angle at the vertex $E_{\frac{\pi}{4}}$ is $2t$,  the length of the side opposite to the vertex $E_\tau$ is 
$\frac{\pi}{2}-2s$, and the length of the side opposite to the vertex $H(-t)(E_s)$ 
is $\frac{\pi}{2}-2\tau$. We have 
\[ 
\cos 2t = \tan (\frac{\pi}{2}-2\tau)\cot (\frac{\pi}{2}-2s) 
=\cot 2\tau \tan 2s 
\] 
according to spherical trigonometry.

Since $H(-t)(E_{s'})=G_0(\vh)H(-t)(E_s)$, by 
considering the action of $G_0$ on its orbits in 
$\dP(J)$ we must 
have 
\[ 
\vh= -2\tau \mod \pi, 
\] 
and 
\begin{equation} \label{s'} 
 s'=\pi-s \mod \pi.   
\end{equation}  
 Then  $t$, $s$ and $\vh$ satisfy 
\begin{equation} \label{napier}  
\cos 2t =- \tan 2s \cot \vh. 
\end{equation}  
In fact Equations (\ref{s'}) and (\ref{napier}) hold for 
any $2t, \h\in [0,\pi]$ and $s,s'\in \R/\pi\Z$ satisfying 
Equation (\ref{s's}). Together with (\ref{x1x2}) we obtain 
that on $T^*D$ 
\begin{equation} \label{E'Es} 
E'_s=E_{\pi-s} \ \text{ if } \tan 2s =\frac{-2x_1x_2}{x_1^2-x_2^2}. 
\end{equation} 
In particular, on $\{ x_1x_2=0\}\subset T^*D$ 
\begin{equation} \label{E'E} 
E'_0=E_0=\pa_{x_1}\wedge \pa_{y_1},  \quad 
E'_\frac{\pi}{2}=E_{\frac{\pi}{2}}=\pa_{x_2}\wedge \pa_{y_2}. 
\end{equation}

{\bf  $\mu_2$- and $y$-indexes.}  \ 
 Now consider 
 the $4\times 4$ matrix formed by column vectors 
  $u'$,   $v'$, $\cos s'\pa_{x_1}+\sin s'\pa_{x_2}$, and 
  $\cos s'\pa_{y_1}+\sin s'\pa_{y_2}$ with $s'=\pi-s$, of which the determinant 
  is 
  \begin{equation*} 
  \begin{split} 
 \fD' & :=\begin{vmatrix} \cos\vh & -\sin\vh\cos 2t & \cos s' & 0 \\ 
 \sin\vh\cos 2t & \cos\vh & \sin s' & 0 \\ \sin\vh\sin 2t & 0 & 0 
 & \cos s' \\  0 & -\sin\vh\sin 2t  & 0 & \sin s' 
  \end{vmatrix}  \\ 
   & = \sin\vh\sin 2t \cdot (\cos 2t\cos 2s'\sin\vh -\sin 2s'\cos\vh) \\ 
   & = \sin\vh\sin 2t \cdot (\cos 2t\cos 2s\sin\vh +\sin 2s\cos\vh). 
   \end{split} 
  \end{equation*}  
$\fD'=0$ iff one of the followings folds: 
\begin{enumerate} 
\item $\vh=0$ or $\pi$, 
\item $t=0$ or $\frac{\pi}{2}$, 
\item $\vh=\frac{\pi}{2}$ and $t=\frac{\pi}{4}$, 
\item $\cos 2t\cos 2s\sin\vh =-\sin 2s\cos\vh$. 
\end{enumerate} 
In (i)-(iii) $\fD'=0$ for all values of $s$. 
Moreover, (i) holds at four points of $\pa D$, (ii) holds 
precisely on $\pa D$, and (iii) corresponds to the origin of $D$. 
Consider the PLG-map 
\[ 
g'_D:D\to \dP(K')
\] 
 with respect 
to the framing $(u,v)$ so that $\xi'_0:=u\wedge v$ is viewed as the 
south pole of $\dP(K')$ and $\xi'_\infty:=Ju\wedge -Jv$ the north pole. 
Then $g'_D$ is surjective with 
\[ 
(g'_D)^{-1}(\xi'_0)=\pa D, \quad (g'_D)^{-1}(\xi'_\infty)=(0,0)\in D. 
\] 
It is easy to see that $\deg g'=\pm 2$. To determine the $\pm$ 
sign of $\deg g'_D$ we consider 
Case (iv) which implies that (recall (\ref{x1x2}), (\ref{E'E})) 
\[ 
\tan 2s =\frac{-\sin\vh\cos 2t}{\cos \vh}=\frac{-2x_1x_2}{x^2_1-x_2^2}. 
\] 
Hence the intersection of $\Gamma_s:=(g'_D)^{-1}(\lambda_s)$ with the interior $int(D)$ of $D$ is the pair of 
orthogonal line segments defined by 
$-2x_1x_2=(x_1^2-x_2^2)\tan 2s$, $x_1^2+x_2^2<1$. 

Take $s=0$, then $V_0:=(g'_D)^{-1}(D_0)=\{ x_1x_2 > 0\} \cap D$. Write $V_0=V_{01}\cup V_{02}$, where 
$V_{01}=V_0\cap \{ x_2>0\}$, $V_{02}=V_0\cap\{ x_2<0\}$.  

The  boundary 
$\pa V_{01}\cap int(D)=\gamma \cup \sigma$, 
where 
\begin{align*} 
\gamma(t) & =(x_1=0, x_2=\cos 2t),\quad 0<t\leq \frac{\pi}{4}; \\ 
\sigma(t) & =( x_1=-\cos 2t,x_2=0),\quad \frac{\pi}{4}\leq t<\frac{\pi}{2}. 
\end{align*}  
Observe that the map $\phi$ from (\ref{z2}) maps $\gamma\cup\sigma$ onto the diameter $\{ x_2=0\}\cap int(D)$ injectively. By 
(\ref{x2=0}) 
along  $\gamma\cup \sigma$  
\[ 
\begin{pmatrix} u & v\end{pmatrix} =\begin{pmatrix} 
-\cos 2t & 0 \\ 0 & -\cos 2t \\ \sin 2t & 0 \\ 0 & -\sin 2t 
\end{pmatrix}, \ 0<t<\frac{\pi}{2},
\] 
$\pa_{x_1}$ generates the intersection subspace of 
$E'_0:=u\wedge Ju=E_0=\pa_{x_1}\wedge \pa_{y_1}$, $\pa_{x_1}=
-(\cos 2t) u-(\sin 2t)Ju$,  whilst $\pa_{x_2}$ generates the 
intersection subspace of $E'_{\frac{\pi}{2}}:=v\wedge Jv=E_{\frac{\pi}{2}}=\pa_{x_2}\wedge \pa_{y_2}$, 
$\pa_{x_2}=(\cos 2t)v-(\sin 2t)Jv$. 
So as $t$ increases from 
$0$ to $\frac{\pi}{2}$, $\pa_{x_1}$ rotates 
in $E'_0$ by an angle of $\pi$ radians, and $\pa_{x_2}=$ 
rotates 
in $E'_{\frac{\pi}{2}}$ by an angle of $-\pi$ radians. 
We have  
\[ 
 \alpha_{01}=(\Delta\varphi)_{\gamma\cup\sigma}=\pi-(-\pi)=2\pi, 
 \] 
 the restricted map 
 \[ 
 g'_S\mid _{\pa V_{01}}:\pa V_{01}\to \lambda^+ _0 
 \] 
 is of degree 1. 
 Similarly, the the restricted map 
 \[ 
 g'_S\mid _{\pa V_{02}}:\pa V_{02}\to \lambda^+ _0 
 \]
 is also of degree 1. 
 These put together imply that the $\mu_2$-index of the zero-section $S$ of $T^*S$ (with respect 
to the framing $\ff:=(J,u,v)$) is 
\[ 
\mu_2(S)=2. 
\] 
There are no crossing $p$-curves on $S$. Take a regular point $q\in S$ and we have 
\[ 
y(S,q)=\mu_2(S)=2. 
\]

{\bf Plumbing of $T^*S^2$'s.}  \ 
Our construction of a unitary framing on 
$T^*S^2$ can be generalized to symplectic manifolds formed 
by plumbing a finite number of $T^*S^2$'s. For example, 
for any integer $n\geq 1$ 
consider the Stein surface $W_n\subset \C^3$ defined by 
the equation 
\begin{equation} \label{Wneqn} 
z_1^2+z_2^2=z_3^{n+1}+\frac{1}{2}. 
\end{equation}  
$W_n$ is the plumbing of $n$ copies of $T^*S^2$ so that 
the Lagrangian zero-section spheres $S_1$,...,$S_n$ 
form an $A_n$-configuration: 
$S_j\pitchfork S_{j+1}$ and in one point, $S_i\cap S_j=\emptyset$ if $|i-j|\neq 1$. Note that each of $W_n$ is 
a parallelizable symplectic 4-manifold, and their space of 
compatible unitary framings are all connected. It is not hard to 
see that 
\[ 
\mu_2(S_j)=2=y(S_j), \quad j=1,...,n. 
\] 
Moreover, we can orient the pair $S_j,S_{j+1}$ so that their 
intersection is positive. Then by the lagrangian surgery as defined 
in \cite{P2} we can desingularize the double point of 
$S_j\cup S_{j+1}$ to get a smooth Lagrangian sphere. There are 
two different desingularizations. Call the resulting Lagrangian 
spheres as $S'_j$ and $S''_j$. $S'_j$ and $S''_j$ are related by 
a $la$-disk surgery, so $\mu_2(S'_j)=\mu_2(S''_j)$. To 
determine $\mu_2(S'_j)$, recall that as we desingularize a 
Lagrangian Whitney sphere in $\R^4$ the $\mu_2$-index 
is decreased by 2. Then a simple calculation yields 
\[ 
\mu_2(S'_j)=\mu_2(S''_j)=\mu_2(S_j)+\mu_2(S_{j+1})-2=2. 
\]

%
%
\subsection{Surfaces in the plumbing of cotangent bundles} \label{plumb}

We consider the following example to illuminate the effect 
of generalized Dehn twists on the relative $y$-index.

Let $L$ be an orientable compact surface without boundary, and $S$ a $2$-dimensional sphere. Let $W$ be 
the symplectic manifold obtained by plumbing the two symplectic 
manifolds $T^*L$ and $T^*S$ at a single point $p_0\in W$. 
$W$ is symplectically parallelizable, with Lagrangian surfaces 
$L$ and $S$ intersect transversally and at $p_0$.

{\bf Framing $\ff=(J,u,v)$.} \ 
We choose a framing $\ff=(J,u,v)$ for $TW$ so that 
on $T^*S$ it is a slight modification of our earlier construction. 
Recall from Section \ref{T*S} $G_t(\h)\subset SU(2)$, $t\in [0,\frac{\pi}{2}]$, $\h\in [0,\pi]$, and  $S=D\cup_\psi \bD$, $\bD=
\{ (\bx_1,\bx_2)\mid \bx_1^2+\bx_2^2\leq 1\}$
$D=\{ (x_1,x_2)\mid x_1^2+x_2^2\leq 1\}$. We may 
assume that $p_0$ is in the interior of $\bD$. 
We also write $x_1=r\cos \h$, 
$x_2=r\sin\h$, $(r,\h)$ are polar coordinates of $D$. 
Let $p_0=(0,0)\in \bD$ denote the south pole of $S$, and 
$p_\infty:=(0,0)\in D$ the north pole. 
Let $C\subset S$ denote the great circle so that 
$C\cap \bD=\{ \bx_2=0\}$, $C\cap D=\{ x_2=0\}$. 

Consider the smooth surjective map $\rho_D:D\to D$, 
\[ 
\rho_D(r,\h)=(\rho(r),\h), 
\] 
where $\rho:[0,1]\to [0,1]$ 
is a smooth function such that 
\begin{itemize} 
\item $\rho(r)=0$ on $[0,\epsilon)$ and $\rho(r)=1$ on $(1-\epsilon,1]$ for some $\epsilon>0$, 
\item $\frac{d\rho}{dr}\geq 0$. 
\end{itemize} 
For $x\in D$ define 
\begin{equation}  
u_x=G_t(\h)(\pa_{x_1}), \quad v_x:=G_t(\h)(\pa_{x_2}), \quad 
\text{if $F(\rho_D(x))=G_t(\h)$}. 
\end{equation} 
Note that 
\[ 
(u,v)=\begin{cases} (\pa_{y_1},-\pa_{y_2}) & \text{ if $r<\epsilon$ }, \\ 
(\pa_{\bx_1},\pa_{\bx_2}) & \text{ if $r=1$}. 
\end{cases} 
\] 
Extend $u, v$ over $\bar{D}$ by setting $u=\pa_{\bx_1}$ and 
$v=\pa_{\bx_2}$ as before. Then further extend $u,v$ over 
$T^*S$ independent of fiber coordinates or, more precisely, by 
pulling back $u,v$ over $T^*S$ via the canonical projection 
$\pi:T^*S\to S$. Finally extend $J$ and $u,v$ over $T^*L$. 
We denote the resulting framing by $\ff:=(J,u,v)$. 
Let $E'_s$ denote the $S^1$-family of 
$J$-complex line bundles associated to $u,v$. Write 
$u_s:=u\cos s +v\sin s$, $v_s=Ju_s$. 

Let $\Gamma_s(S)$ denote the $E'_s$-locus of $S$. $\Gamma_s(S)$ is the union of $\bD\cup\{ (r,\h)\in D\mid \rho(r)=0 \text { or } 1\}$ and 
a pair of great circles on $S$ intersects orthogonally at the poles $p_0,p_\infty$. With (\ref{E'Es}) we can 
write this pair of great circles as 
$C_s\cup C^\perp_s$, where $C_s$ is the great circle 
tangent to $E'_s$ (and orthogonal to $E'_{s^\perp}$), and 
$C^\perp_s$ the one tangent to $E'_{s^\perp}$ (and orthogonal 
to $E'_s$). Then $C^\perp_s=C_{s^\perp}$. 
In particular, $C_0=C=\{ x_2=0\} \cup \{ \bx_2=0\}$, 
$C^\perp_0=\{ x_1=0\}\cup\{ \bx_1=0\}$.

{\bf $S^1$-action on $T^*S$.} \ 
 Consider an $S^1$-group $\tilde{G}=\{ \tilde{g}_\h\mid \h\in \R/2\pi\Z\}$ acting on $S=D\cup \bD$  by rotating with respect to 
the poles $p_0:=(0,0)\in \bD$ and $p_\infty:=(0,0)\in D$: 
\begin{align*} 
\tilde{g}_\h(\bx_1,\bx_2) & =( \bx_1\cos\h-\bx_2\sin\h, \ 
\bx_1\sin\h +\bx_2\cos\h), \ (\bx_1,\bx_2)\in \bD, \\ 
\tilde{g}_\h(x_1,x_2) & =( x_1\cos\h+x_2\sin\h,- x_1\sin\h +x_2\cos\h), \ (x_1,x_2)\in D. 
\end{align*} 

This group action extends over $T^*S$ symplectically. The group $\tilde{G}$ acts on $T^*S$ Hamiltonianly, preserving 
the complex structure $J$ induced by the standard ones on $T^*D$ and $T^*\bar{D}$ as constructed in Section \ref{T*S}. 
In addition, $\tilde{g}_\h(\Gamma_s(S))=\Gamma_{s+\h}(S)$, 
$\tilde{g}_\h(C_s)=C_{s+\h}$ and $\tilde{g}(C^\perp_s)=C^\perp_{s+\h}$, By (\ref{E'Es}) $E_s$ is tangent to the total space $T^*C_s\subset T^*S$, and $E_{s^\perp}$ is orthogonal to 
$T^*C_s$. Since $\tilde{g}_\h$ maps $T^*C_s$ to $T^*C_{s+\h}$ and $\tilde{g}_\h$ is unitary, $(\tilde{g}_\h)_*$ maps the 
ordered pair $(E_s,E_{s^\perp})|_{T^*C_s}$ to the ordered 
pair $(E_{s+\h},E_{s^\perp+\h})|_{T^*C_{s+\h}}$. This in 
particular implies the following scenario: if $L_\gamma:=Orb_{\tilde{G}}{\gamma}\subset T^*S$ is an embedded connected 
Lagrangian surface and $\gamma\subset T^*C_0$, 
then the union of the $E'_s$-locus $\Gamma_s$ of 
$L_\gamma$, form a 1-dimensional 
smooth foliation of $L_\gamma$. Each of $\Gamma_s$ consists 
of four connected arcs and $\tilde{g}_\h(\Gamma_s)=\Gamma_{s+\h}$. We will utilize this $S^1$-symmetry to facilitate 
the computation of the $y$-index below.

Near $p_0=(0,0)$ $L$ can be identified with the $\by_1\by_2$-plane and $S$ the $\bx_1\bx_2$-plane. 
Let $\Delta\subset L$ be a $\tilde{G}$-invariant disk intersecting  
with $\bar{D}$ at $p_0=(0,0)$. Let $\Delta_0\subset \Delta$ be a smaller $\tilde{G}$-invariant disk. 
We orient $L$ so that $\{ \pa_{\by_1}, \pa_{\by_2}\}$ is a 
positive basis of $T_\Delta L$.

{\bf Positive even Dehn twists.} \  
We may deform $L$ by a $\tilde{G}$-invariant 
Hamiltonian isotopy which is compactly supported in a 
open neighborhood of $p_0\in W$ so that 
$p_0$ and $L\setminus\Delta$ are fixed under the isotopy, 
and  after 
the isotopy $\Delta$ is contained in $T^*\bD$, 
\begin{enumerate} 
\item $\Delta$ contains a smaller disk $\Delta_0=\Delta\cap 
T^*_{p_0}S$, 
\item $\Delta\setminus\Delta_0$ consists of two adjacent 
anular folding domains 
 $U_+=Orb_{\tilde{G}}(\s_+)$ and 
$U_-=Orb_{\tilde{G}}(\s_-)$ so  that 
$\pa \Delta_0\subset \pa U_+$ as shown in Figure \ref{pdehn1}.  By (\ref{mu2Lgamma1}) the PLG-degrees of these three domains are 
\[ 
 d_{U_+}=2, \quad d_{U_-}=-2. 
\] 
\end{enumerate} 

\begin{figure}[th] 
 \begin{center} 
 \includegraphics[scale=0.75]{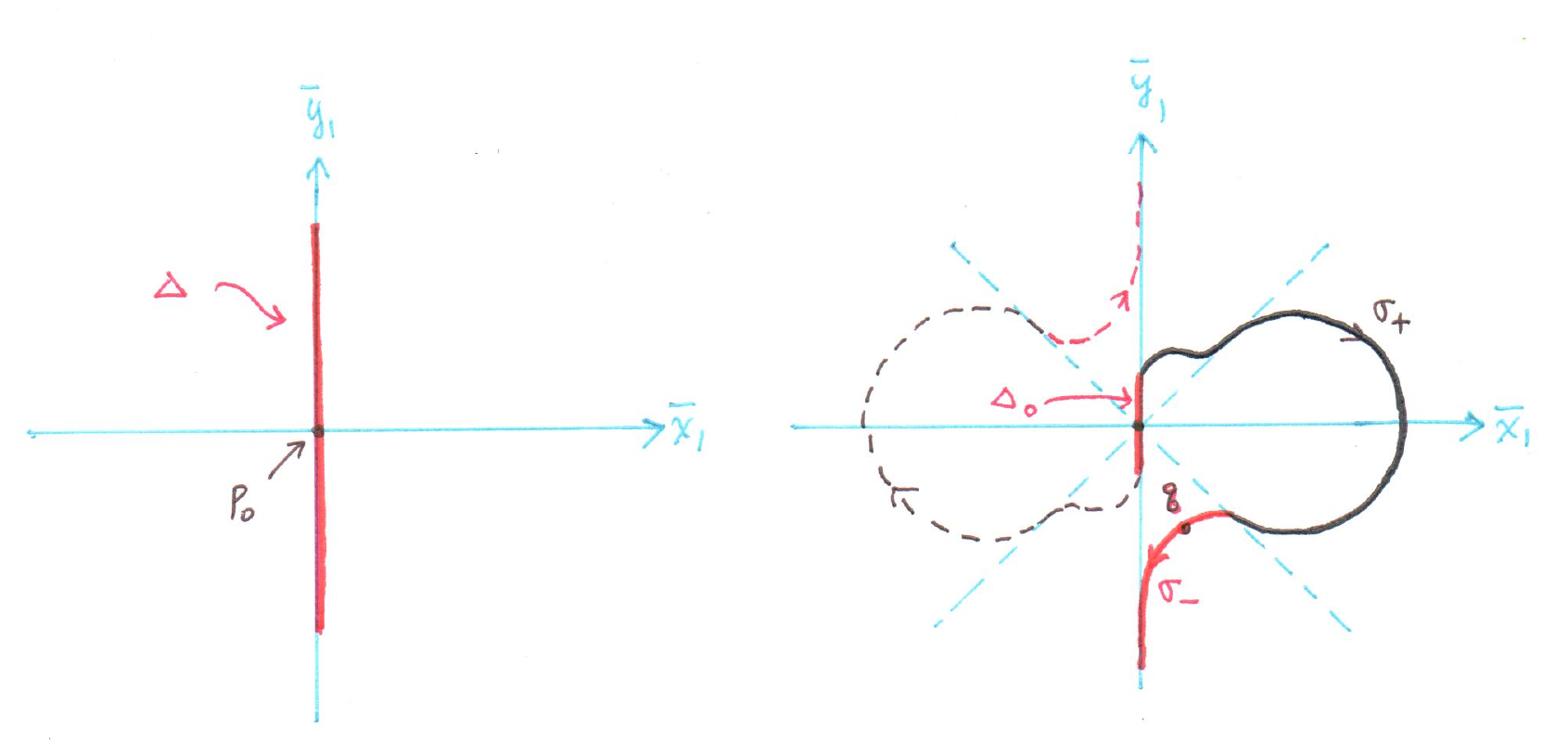} 
\caption{$\Delta$ before and after deformed. \label{pdehn1}} 
\end{center} 
\end{figure} 

The curves  $\s_\pm$ 
are contained in the half plane $\{ \bx_1\geq 0\}$ in the $\bx_1\by_1$ plane, and  are oriented as in Figure \ref{pdehn1}. 
Write $\s:=\s_+\cup \s_-$. Following (\ref{Gamma2}) the oriented Lagrangian 
locus $\check{\Gamma}_0^+\cap (\Delta\setminus\Delta_0)$ is  
\begin{equation} \label{step0} 
-\s \cup \tilde{g}_{\frac{\pi}{2}}(\s) \cup \tilde{g}_\pi (-\s) 
\cup \tilde{g}_{\frac{3\pi}{2}}(\s), 
\end{equation} 
here $-\s$ means $\s$ but with the reversed orientation.

From now on we fix a reference point $q$ in the interior of 
$U_-$ as depicted in the right picture of Figure \ref{pdehn1} so that $g'_L$ is regular 
near $q$. $q$ is contained in a crossing domain of $L$ containing $\Delta$.

\begin{figure}[th] 
 \begin{center} 
 \includegraphics[scale=0.75]{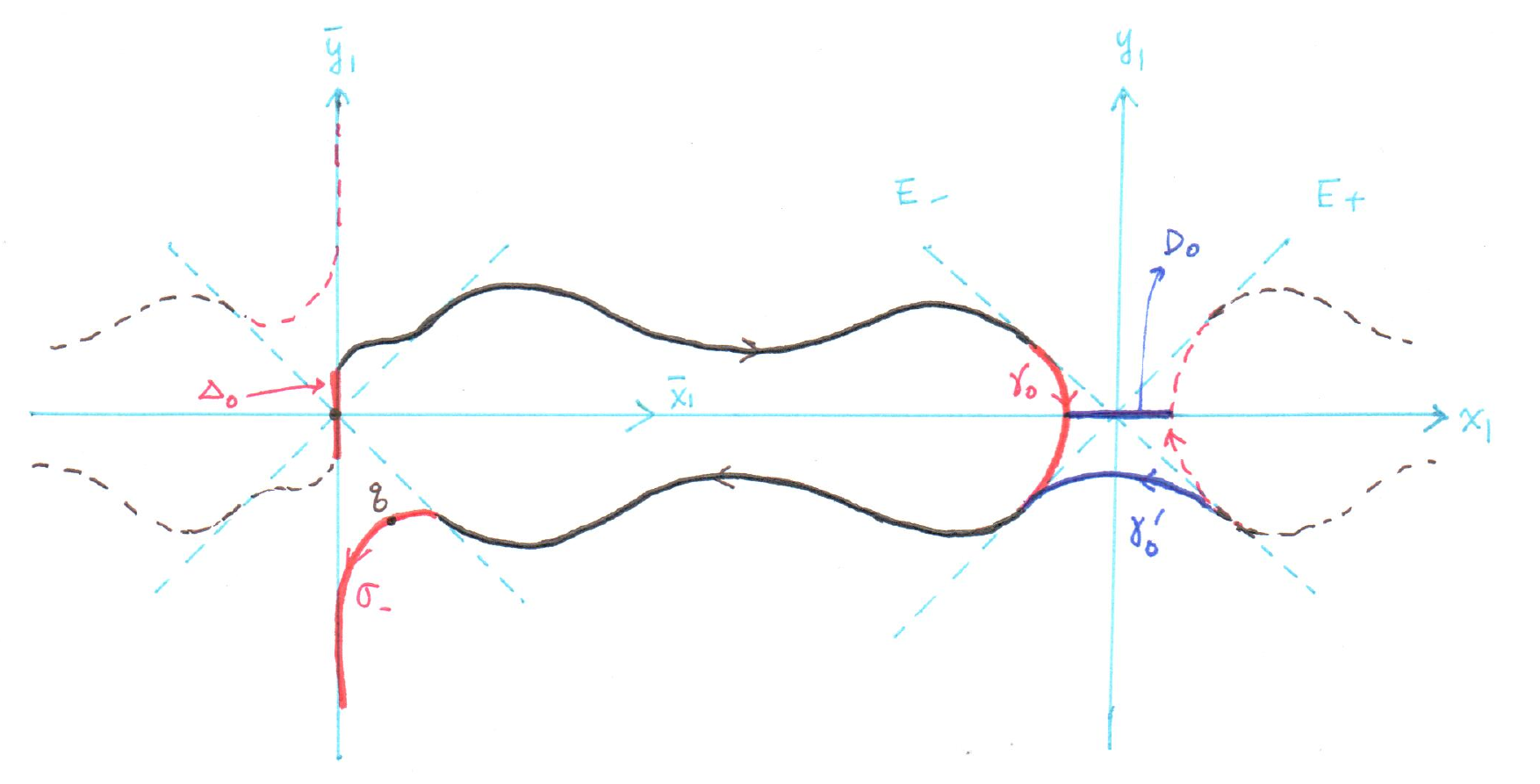} 
\caption{$\Delta$ for "positive" surgery. \label{pdehn2}} 
\end{center} 
\end{figure} 

Next we  apply another $\tilde{G}$-invariant Hamiltonian 
isotopy to $L$, keeping  $L\setminus U_+$ fixed all the time, and turning  
$U_+$ to a new annulus also denoted as $U_+$, so that 
now $U_+$ contains a smaller annulus $U_0=Orb_{\tilde{G}}(\g_0)$ 
as shown in Figure \ref{pdehn2} such that 

\begin{enumerate} 
\item $U_0\subset T^*D_\e$ is a {\em crossing} domain, 
\item $\g_0$ is contained in the half plane $\{ x_1<0\}$ in the 
$x_1y_1$-plane. 
\end{enumerate} 

The loci $\check{\Gamma}^+_s$ together with their orientations 
are preserved under the $\tilde{G}$-invariant Hamiltonian isotopy. 
Then (\ref{step0}) implies that 
\[ 
\check{\Gamma}_0^+\cap U_0=    
(-\g_0)\cup \tilde{g}_{\frac{\pi}{2}}(\g_0) \cup \tilde{g}_\pi (-\g_0) 
\cup \tilde{g}_{\frac{3\pi}{2}}(\g_0).  
\] 
Recall that  $E'_0=\pa_{y_1}\wedge -\pa_{x_1}$ and 
$E'_{\frac{\pi}{2}}=-\pa_{y_2}\wedge \pa_{x_2}$ along $-\gamma_0$. Since $(\Delta\varphi)_{-\gamma_0}=\pi$ 
(compare with Figure \ref{Delta2}(b)) by applying (\ref{mu2Lgamma2}) we have 
\[ 
\alpha_{0,U_0}=4(\Delta\varphi)_{-\gamma_0}=4\pi. 
\] 

Observe that $D_\e\subset S$ contains a  smaller 
disk $D_0$ which is a $la$-disk of $L$ with $\pa D_0$ contained in the interior of $U_0$. 

 Apply $\eta_{D_0}$ to $L$ 
by replacing $U$ by a $\tilde{G}$-invariant Lagrangian 
annulus $U'_0$ contained in a 
symplectic neighborhood of $D_0$ in $T^*S$ with $\pa U'_0=\pa U_0$ to obtain a 
new smooth embedded Lagrangian surface 
\[ 
L^1=(L\setminus U_0)\cup U'_0. 
\] 
Note that 
\[ 
L^1:=\eta_{D_0}(L)=\tau^2_S(L). 
\] 
is also obtained by applying to $L$ a positive double Dehn 
twist with respect to $S$. 

 Write $U'_0=Orb_{\tilde{G}}(\g'_0)$, where $\g'_0$ is with the 
 induced orientation and is contained the 
 half plane $\{ |x_1|<\e, y_1<0\}$ as shown in Figure \ref{pdehn2}. 
 For $L^1$, the Lagrangian locus $\check{\Gamma}^+_0$ 
 restricted to $U_1$ is 
 \[ 
 \check{\Gamma}_0^+\cap U'_0=    
(-\g'_0) \cup \tilde{g}_{\frac{\pi}{2}}(\g'_0) \cup \tilde{g}_\pi (-\g'_0) 
\cup \tilde{g}_{\frac{3\pi}{2}}(\g'_0), 
 \] 
 with $-\g'_0$ denotes $\g'_0$ but with the reversed 
 orientation. We have $(\Delta\varphi)_{-\gamma'_0}=-\pi$ 
and 
\[ 
\alpha_{0,U'_0}=4(\Delta\varphi)_{-\gamma'_0}=-4\pi. 
\] 
Therefore 
\begin{equation} 
\begin{split} 
y(L^1,L,q;\ff)& =y(L^1,q;\ff)-y(L,q;\ff) \\ 
& = \frac{1}{2\pi}(\alpha_{0,U'_0}-\alpha_{0,U_0})\\ 
& = \frac{1}{2\pi}(-4\pi-4\pi) \\ 
& =-4. 
\end{split} 
\end{equation}

Repeat similar Hamiltonian isotopies to $L^1$ so that 
$p_0$ and $L^1\setminus\Delta_0$ are fixed all the time, and the 
resulting $\Delta_0$ contains 
\begin{itemize} 
\item 
a smaller disk $\Delta_1\subset 
T^*_{p_0}S$,  and 
\item 
an annular crossing domain $U_1\subset T^*D_\e$ so that 
for $L^1$, 
the relative $E'_0$-phase along $\check{\Gamma}^+_0\cap U_1$ 
is $\alpha_{0,U_1}=4\pi$, and  
\item there is a $la$-disk $D_1\subset D_\e$ of $L^1$ 
with $\pa D_1$ contained in the the interior of $U_1$. 
\end{itemize} 
Apply $\eta_{D_1}$ to $L^1$ by replacing $U_1$ with 
another annular $U'_1$ with $\pa U_1=\pa U'_1$ we get 
a new Lagrangian surface 
\[ 
L^2=\eta_{D_1}(L^1)=\tau^2_S(L^1)=\tau^4_S(L) 
\] 
which can also be obtained by applying $4$ positive 
generalized Dehn twists to $L$ along $S$. As in the case 
of $U'_0$, the relative $E'_0$-phase along $\check{\Gamma}_0^+\cap U'_1$ 
is 
\[ 
\alpha_{0,U'_1}=-4\pi. 
\] 
So again we have 
\begin{equation} 
\begin{split} 
y(L^2,L^1,q;\ff)& =y(L^2,q,\ff)-y(L^1,q;\ff) \\ 
& = \frac{1}{2\pi}(\alpha_{0,U'_1}-\alpha_{0,U_1})\\ 
& = \frac{1}{2\pi}(-4\pi -4\pi) \\ 
& =-4. 
\end{split} 
\end{equation} 
Also it is easy to see that 
\begin{equation} 
y(L^2,L,q;\ff)=-8. 
\end{equation} 

Repeat this process to $(L^k,\Delta_{k-1},\Delta_k,U_k,D_k)$ to 
get $(L^{k+1}, \Delta_k,\Delta_{k+1}, U_{k+1})$, and so on 
so forth, we obtain an infinite sequence of smoothly isotopic 
Lagrangian surfaces $L^0=L$, $L^1$, $L^2$,..., $L^n$,..., 
with $L^n=\tau^{2n}_S(L)$, and 
\begin{equation} 
y(\tau^{2m}_S(L), \tau^{2n}_S(L), q;\ff)=-4(m-n), \quad m,n\in \N\cup\{ 0\}. 
\end{equation}

{\bf Negative even Dehn twists.} \ 
Proceed to the case of double negative generalized Dehn 
twist. Again 
in $T^*\bar{D}$ we can deform $\Delta\subset L$ as in Figure 
\ref{ndehn1} by  
a different $\tilde{G}$-invariant Hamiltonian isotopy 
 in a neighborhood of $p_0\in W$, keeping $p_0$ and 
 $L\setminus\Delta$ fixed all the time so that after the isotopy, 
 $\Delta$ can be expressed as the union of three 
 consecutive domains 
 \[ 
 \Delta=\Delta_0\cup U_\dagger \cup U_-
 \] 
 where $\Delta_0$ and $U_-$ are as before, $U_\dagger= 
 Orb_{\tilde{G}}(\s_\dagger)$ is a $p$-domain with PLG-degree 
 $d_{U_\dagger}=-2$, $\s_\dagger$ is an oriented curve 
 contained in the $\bx_1\by_1$-plane as shown in Figure \ref{ndehn1}. We still use $q\in U_-$ as the reference point.

 \begin{figure}[th] 
 \begin{center} 
 \includegraphics[scale=0.75]{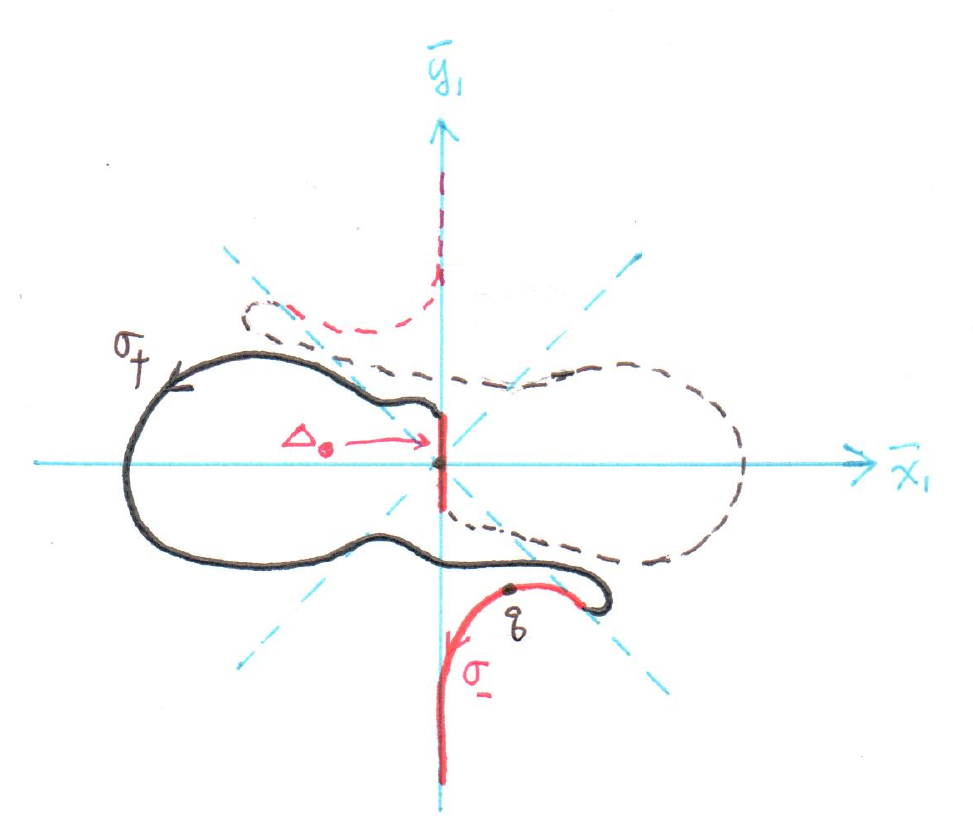} 
\caption{$\Delta$ deformed and before surgery. \label{ndehn1}} 
\end{center} 
\end{figure}

Apply another $\tilde{G}$-invariant Hamiltonian 
isotopy to $L$, keeping  $L\setminus U_\dagger$ fixed all the time, and turning  
$U_\dagger$ to a new annulus also denoted as $U_\dagger$, so that 
now $U_\dagger$ contains $-U_0$ which denotes the crossing domain 
$U_0$ but with the reversed orientation,  as shown in Figure \ref{ndehn2}. 
We can write 
$-U_0=Orb_{\tilde{G}}(-\bar{\g}_0)$ where 
\[ 
-\bar{\g}_0:=\tilde{g}_\pi(-\g)  
\] 
is a curve contained in the half plane $\{ 0<x_1<\e \}$ in the 
$x_1y_1$-plane.

\begin{figure}[th] 
 \begin{center} 
 \includegraphics[scale=0.75]{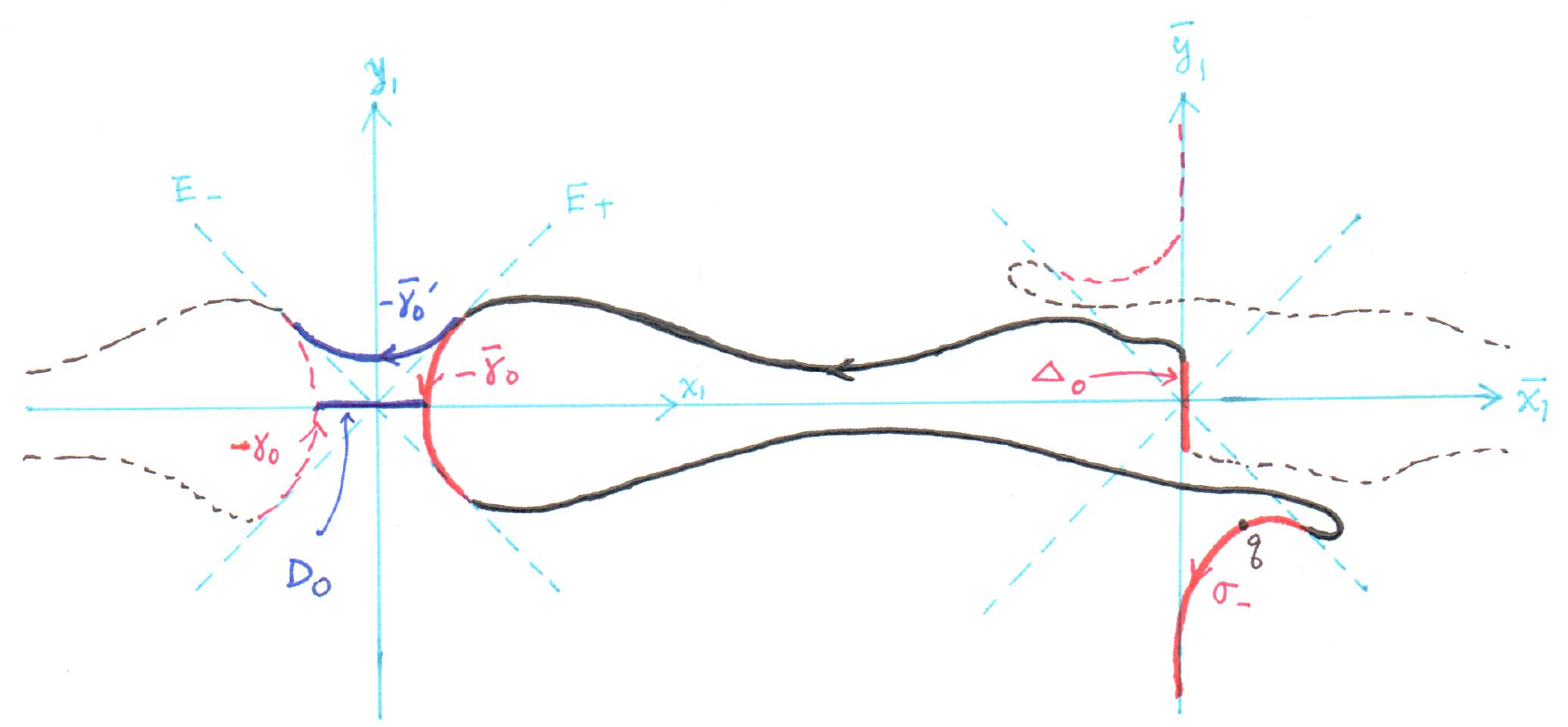} 
\caption{$\Delta$ and $D_0$ for "negative" surgery. \label{ndehn2}} 
\end{center} 
\end{figure}

We have 
\[ 
\check{\Gamma}_0^+\cap -U_0=    
\g_0\cup \tilde{g}_{\frac{\pi}{2}}(-\g_0) \cup \tilde{g}_\pi (\g_0) 
\cup \tilde{g}_{\frac{3\pi}{2}}(-\g_0).  
\] 
Since $(\Delta\varphi)_{\gamma_0}=-\pi$ we have 
\[ 
\alpha_{0,-U_0}=4(\Delta\varphi)_{\gamma_0}=-4\pi. 
\] 

Recall $D_\e\subset S$ contains a smaller 
disk $D_0$ which is a $la$-disk of $L$ with $\pa D_0$ contained in the interior of $U^-_0$. 

 Apply $\eta_{D_0}$ to $L$ 
by replacing $-U_0$ by a $\tilde{G}$-invariant Lagrangian 
annulus $-U'_0$, the annulus $U'_0$ with the reversed orientation, contained in a 
symplectic neighborhood of $D_0$ in $T^*S$ with $\pa (-U'_0)=\pa (-U_0)$ to obtain a 
new smooth embedded Lagrangian surface 
\[ 
L^{-1}:=(L\setminus(-U_0))\cup (-U'_0). 
\] 
Note that 
\[ 
L^{-1}=\eta_{D_0}(L)=\tau^{-2}_S(L). 
\] 
is also obtained by applying to $L$ a negative double Dehn 
twist with respect to $S$. 

Like the case for $-U_0$, we can write $-U'_0=Orb_{\tilde{G}}
(-\g'_0)$, where $-\gamma'_0$ is $\gamma'_0$ but  with the reversed orientation 
 as shown in Figure \ref{ndehn2}. 
 For $L^{-1}$, the Lagrangian locus $\check{\Gamma}^+_0$ 
 restricted to $-U'_0$ is 
 \[ 
 \check{\Gamma}_0^+\cap(-U'_0)=    
\g'_0 \cup \tilde{g}_{\frac{\pi}{2}}(-\g'_0) \cup \tilde{g}_\pi (\g'_0) 
\cup \tilde{g}_{\frac{3\pi}{2}}(-\g'_0).
 \] 
As $(\Delta\varphi)_{\gamma'_0}=\pi$ it follows that 
\[ 
\alpha_{0,-U'_0}=4(\Delta\varphi)_{\gamma'_0}=4\pi. 
\] 
Therefore 
\begin{equation} 
\begin{split} 
y(L^{-1},L,q;\ff)& =y(L^{-1},q,\ff)-y(L,q;\ff) \\ 
& = \frac{1}{2\pi}(\alpha_{0,-U'_0}-\alpha_{0,-U_0})\\ 
& = \frac{1}{2\pi}(4\pi +(-4\pi)) \\ 
& =4. 
\end{split} 
\end{equation}

Apply to $(L^{-1},\Delta_0)$ $\tilde{G}$-invariant Hamiltonian isotopies  similar to those applied to $(L,\Delta)$ we 
get $(L^{-1}, \Delta_0, \Delta_1,-U_1,D_1)$ 
with $\Delta_1\cup (-U_1)\subset \Delta_0$. Then apply $\eta_{D_1}$ to get $L^{-2}=\tau^{-2}_S(L^{-1})=\tau^{-4}_S(L)$, 
and so on so forth,  
 we obtain an infinite sequence $(L^{-n},\Delta_{n-1},\Delta_n,-U_n,D_n)$, $n\in \N$, 
with $L^{-n}=\tau^{-2n}_S(L)$, and 
\begin{equation} 
y(\tau^{-2m}_S(L), \tau^{-2n}_S(L), q;\ff)=4(m-n), \quad m,n\in \N\cup\{ 0\}. 
\end{equation} 

 Summing up the results for positive and negative even 
 Dehn twists, we arrive at the following conclusion:

\begin{prop} 
Assume that an orientable Lagrangian surface $L$ intersects transversally with a Lagrangian 
sphere $S$ and in a point $p$. Let $L^n:=\tau^{2n}_S(L)$
for $n\in \Z$. 
Then 
\[ 
y(L^n,L^m,q;\ff)=4(m-n), m,n\in \Z. 
\] 
\[ 
y(L^n,L,q;\ff)=-4n, \quad n\in\Z. 
\] 
\end{prop}

%
%
\subsection{Monotone tori in $W_n$} \label{Wn}

Recall for any integer $n\geq 1$ 
the Stein surface $W_n\subset \C^3$ defined by 
the equation 
\[ 
z_1^2+z_2^2=z_3^{n+1}+\frac{1}{2}. 
\] 
$W_n$ is the plumbing of $n$ copies of $T^*S^2$ so that 
the Lagrangian zero-section spheres $S_1$,...,$S_n$ 
form an $A_n$-configuration: 
$S_j\pitchfork S_{j+1}$ and in one point, $S_i\cap S_j=\emptyset$ if $|i-j|\neq 1$. Each of $W_n$ is 
a parallelizable symplectic 4-manifold, and its unitary 
framing is unique up to homotopy.

To facilitate later computation we identify for each $k$ 
$S_k=D_k\cup_{\psi_k}\bD_k$, where $D_k:=\{x^k=(x^k_1,x^k_2)\in \R^2\mid |x^k|\leq 1\}$, 
$\bD_k:=\{\bx^k=(\bx^k_1,\bx^k_2)\in \R^2\mid |\bx^k|\leq 1\}$, and $\psi_k:\bD_k-\{0\}\to D_k-\{ 0\}$  a 
diffeomorphism defined similar to $\psi$ as in (\ref{bDD}). 
Let $T^{*,\e} D_k$ denote the cotangent disk bundle of $D_k$ 
with fiber coordinates $y^k=(y^k_1,y^k_2)$, $|y^k|<\e$. Similarly $T^{*,\e}\bD_k$ denotes the cotangent disk bundle of $\bD_k$ 
with fiber coordinates $\by^k=(\by^k_1,\by^k_2)$, $|\by^k|<\e$. Here $\e>0$ is some fixed small number. 
Then $\psi_k$ induces a symplectic diffeomorphism 
$\Psi_k:T^{*,\e} (\bD\setminus\{ 0\})\to T^{*,\e}(D\setminus\{0\})$. 
We consider for $k=2,3,...n$  plumbings 
\[ 
\chi_k: T^{*,\e}(\{ |\bx^k|<\e\}) \to 
T^{*,\e}(\{ |x^{k-1}|<\e\})
\] 
defined by the  following identifications 
\begin{align*} 
\bx^k_1 \  & \longleftrightarrow \  y^{k-1}_1 \\ 
\bx^k_2 \ & \longleftrightarrow  \  -y^{k-1}_2 \\
\by^k_1 \ & \longleftrightarrow  \  -x^{k-1}_1 \\ 
\by^k_2 \ & \longleftrightarrow  \  x^{k-1}_2 
\end{align*} 
The resulting manifold is $W_n$. 

Identify $T^{*,\e} D_k$ as complex domains in $\C^2$ with 
complex coordinates $z^k_j:=x^k_j+\sqrt{-1}y^k_j$, and 
$T^{*,\e} \bD_k$ as complex domains in $\C^2$ with 
complex coordinates $w^k_j:=\bx^k_j+\sqrt{-1}\by^k_j$. 
Then both $\Psi_k$ and $\chi_k$ are holomorphic and $\tilde{G}$-invariant 
where $\tilde{G}$ is an $S^1$-group acting on $T^{*,\e}\bD_k$ 
as counterclockwise rotations 
\[ 
\begin{pmatrix} \cos\h & -\sin\h \\ \sin\h & \cos\h 
\end{pmatrix} 
\]
with respect to the complex coordinates $(w^k_1,w^k_2)$, and 
acting on $T^{*,\e} D_k$ as clockwise rotations 
\[ 
\begin{pmatrix} \cos\h & \sin\h \\ -\sin\h & \cos\h 
\end{pmatrix} 
\]
with respect to the complex coordinates $(z^k_1,z^k_2)$,

We fix a unitary basis $(u,v)$ on $W_n$ so that on each 
$T^{*,\e}S_k$ it is  the unitary basis 
on $T^*S$ as used in Section \ref{plumb}, 
\[ 
(u,v)=\begin{cases} 
(\pa_{\bx^k_1}, \pa_{\bx^k_2}) & \text{ on $\{ |\bx^k|<\e, \ |\by^k|<\e\}$}, \\ 
(\pa_{y^k_1}, -\pa_{y^k_2}) & \text{ on $\{ |x^k|<\e, \ |y^k|<\e\}$}.  
\end{cases} 
\] 
Let $E'_s$, $s\in \R/\pi\Z$ denote the $S^1$-family of 
complex line bundles associated to $(u,v)$, 
\[ 
E'_s=\begin{cases} 
(\cos s\pa_{\bx^k_1}+\sin s\pa_{\bx^k_2})\wedge 
(\cos s\pa_{\by^k_1}+\sin s\pa_{\by^k_2})  & \text{ on $T^{*,\e}\bD_k$}, \\ 
(\cos s\pa_{y^k_1}-\sin s\pa_{y^k_2})\wedge 
(-\cos s\pa_{x^k_1}+\sin s\pa_{x^k_2})  & \text{ on $T^{*,\e}D_k$}
\end{cases} 
\] 
Let $C^k_0, C^k_{\frac{\pi}{2}}$ be the equators of $S_k$ defined by $C^k_0:=\{ x_2=0\}\cup \{ \bx_2=0\}$, 
$C^k_{\frac{\pi}{2}}:=\{ x_1=0\}\cup\{ \bx_1=0\}$. Then 
similar to the corresponding case in Section \ref{plumb} we have that $E'_0$ is tangent to $\cup_{k=1}^nT^*C^k_0$ and orthogonal to $\cup_{k=1}^nT^*C^k_\frac{\pi}{2}$, $E'_\frac{\pi}{2}$ is tangent to $\cup_{k=1}^nT^*C^k_\frac{\pi}{2}$ and 
orthogonal to $\cup_{k=1}^nT^*C^k_0$.

Let $T_{-1}\subset \{ |\bx^1|<\e, |\by^1|<\e\}\subset T^{*,\e}S_1$ be the (Chekanov) torus defined as the $\tilde{G}$ orbit 
of the circle $\upsilon=\sigma\cup \gamma_{-1}\subset \{ 
\by^1_1>0\}\cap \Sigma_0$ (see Figure \ref{T10}),  
\begin{align} 
\sigma(s) & :=(\bx^1_1=-r\sin s, \ \by^1_1=\sqrt{2}r+r\cos s), \quad \frac{-3\pi}{4}\leq s\leq \frac{3\pi}{4}, \\ 
\gamma_{-1}(s) & :=(\bx^1_1=-r\sin s, \ \by^1_1=\sqrt{2}r+r\cos s), \quad \frac{3\pi}{4}\leq s\leq \frac{5\pi}{4}. 
\end{align}  
Orient $T_{-1}$ so that $\{ \dot{\upsilon}, \tilde{X}(\upsilon)\}$ is a positive basis of $T_\upsilon (T_{-1})$. 
 $U_\sigma:=Orb_{\tilde{G}}(\sigma)$ and $U_{-1}:=Orb_{\tilde{G}}(\gamma_{-1})$ are crossing domains of $T_{-1}$. Fix a regular  interior point $q$ of $U_\sigma$ as the reference point. 

The oriented proper $E'_0$-locus of $T_{-1}$ is 
\[ 
\check{\Gamma}'_0=\upsilon \cup \tilde{g}_{\frac{\pi}{2}}(-\upsilon) \cup \tilde{g}_\pi (\upsilon) \cup \tilde{g}_{\frac{3\pi}{2}}
(-\upsilon). 
\] 
We have 
\[ 
(\Delta\varphi)_\sigma=\pi, \quad (\Delta\varphi)_{\gamma_{-1}}=\pi, 
\] 
\begin{equation} 
y(T_{-1}, q)=\frac{1}{2\pi}(4(\Delta\varphi)_\sigma+4(\Delta\varphi)_{\gamma_{-1}})=\frac{1}{2\pi}(4\pi+4\pi)=4. 
\end{equation} 

\begin{figure}[th] 
 \begin{center} 
 \includegraphics[scale=0.75]{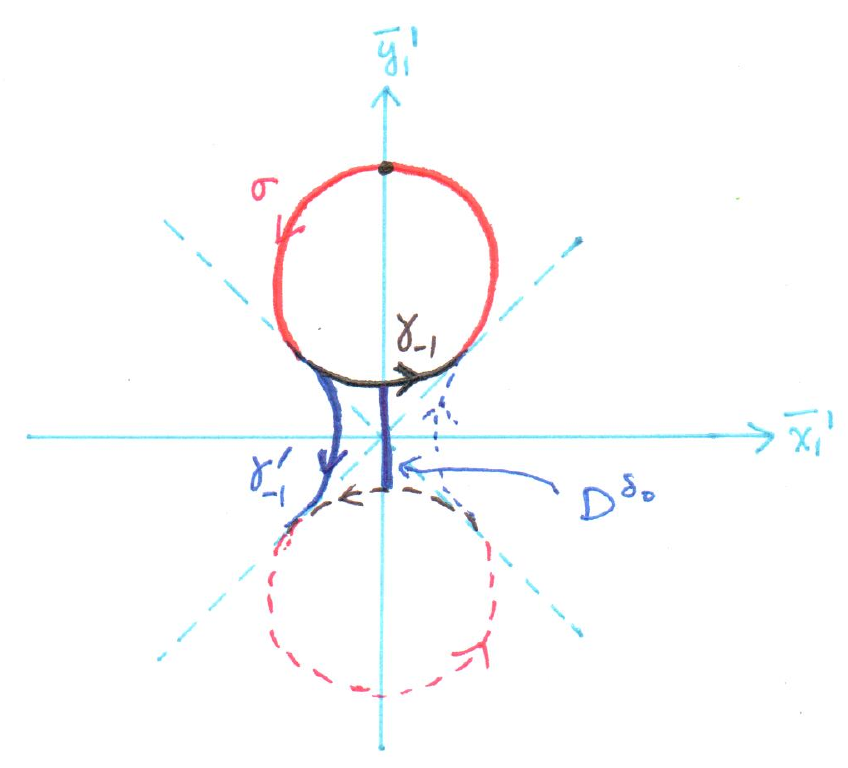} 
\caption{From $T_{-1}$ to $T_0$. \label{T10}} 
\end{center} 
\end{figure}

The disk fiber $T^{*,\e}_{p_0}(\bD_1)$ contains a 
stable 
$la$-disk $D^{\delta_0}$ of $T_{-1}$ with its boundary lies in the interior of 
$U_{-1}$ (see Figure \ref{T10}). 
Let $\gamma'_{-1}:=M(\gamma_{-1})$, and $U'_{-1}:=Orb_{\tilde{G}}(\gamma'_{-1})$. 
Let $T_0=\eta_{D^{\delta_0}}(T_{-1})$ denote the Lagrangian torus 
obtained by replacing $U_{-1}$ by $U'_{-1}$: 
\[ 
T_0=U_\sigma\cup U'_{-1}. 
\] 
Observe that $(\Delta\varphi)_{\gamma'_{-1}}=-\pi$, so  
\begin{equation} 
y(T_0,q)=\frac{1}{2\pi}(4(\Delta\varphi)_{\sigma}+4(\Delta\varphi)_{\gamma'_{-1}})=0, 
\end{equation} 
\begin{equation} 
\begin{split} 
y(T_0,T_{-1},q)  & := y(T_0,q)-y(T_{-1},q) \\ 
& = \frac{1}{2\pi}(4(\Delta\varphi)_{\gamma'_{-1}}-
4(\Delta\varphi)_{\gamma_{-1}}) \\ 
& = -4. 
\end{split} 
\end{equation}

\begin{figure}[th] 
 \begin{center} 
 \includegraphics[scale=0.75]{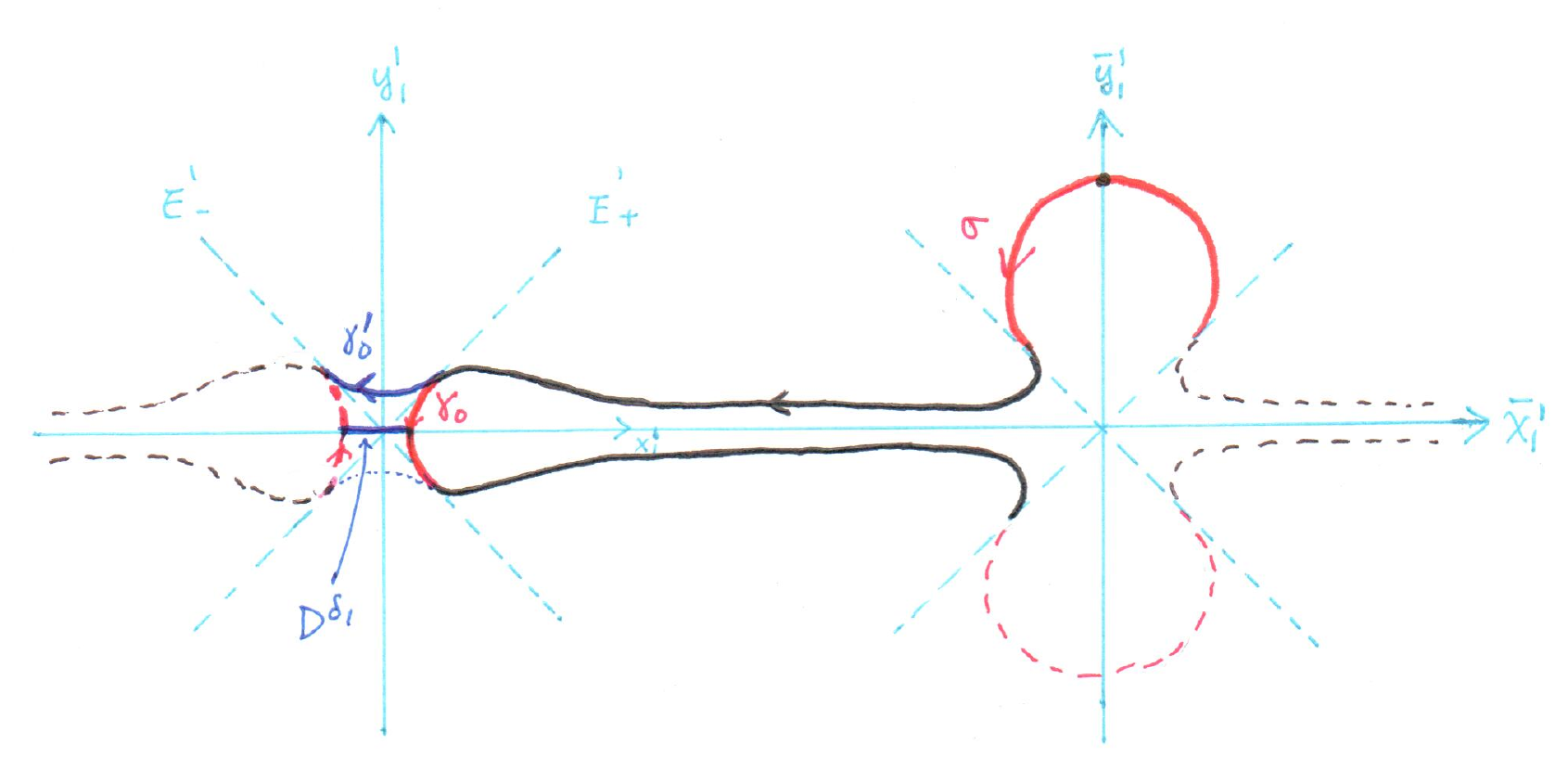} 
\caption{From $T_0$ to $T_1$. \label{T01}} 
\end{center} 
\end{figure} 

With a $\tilde{G}$-invariant Hamiltonian isotopy with 
$U_\sigma$ kept fixed all the time we may deform $U'_{-1}$ 
to a new Lagrangian annulus, also denoted as $U'_{-1}$, 
so that the deformed $\gamma'_{-1}$ contains an arc 
$\gamma_0$ lying in $(T^{*,\e}D_1)\cap\{ x^1_2=0=y^1_2\}$ 
as shown in Figure \ref{T01}, 
so that $U_0:=Orb_{\tilde{G}}(\gamma_0)$ is a crossing domain 
and 
\[ 
(\Delta\varphi)_{\gamma_0}=\pi.  
\] 
For example we may take $\gamma_0$ to be the oriented arc 
\[ 
\gamma_0(s)=(x^1_1=-\sqrt{2}r_1-r_1\cos s, y^1_1=-r_1\sin s), \quad \frac{3\pi}{4}\leq s\leq \frac{5\pi}{4}. 
\]

Observe that $D_1$ contains a stable $la$-disk $D^{\delta_1}$ of $T_0$ 
with its boundary lying in the interior of $U_0$. Let 
$\gamma'_0:=M(\gamma_0)$, then similar to the case of the pair $(\gamma_{-1},\gamma'_{-1}:=M(\gamma_{-1}))$ we have 
\[ 
(\Delta\varphi)_{\gamma'_0}=-\pi=-(\Delta\varphi)_{\gamma_0}. 
\] 
Let $T_1:=\eta_{D^{\delta_1}}$ denote the Lagrangian torus 
obtained from $T_0$ by replacing $U_0$ by $U'_0$: 
\[ 
T_1:=(T_0\setminus U_0)\cup U'_0. 
\] 
Then similar to the case for $T_{-1}$ and $T_0=\eta_{D^{\delta_0}}(T_{-1})$, we have 
\begin{equation} 
y(T_1,T_0,q)=\frac{1}{2\pi}(4(\Delta\varphi)_{\gamma'_0}-4(\Delta\varphi)_{\gamma_0})=\frac{1}{2\pi}(-4\pi-4\pi)=-4, 
\end{equation} 
and 
\begin{equation} 
y(T_1,q)=y(T_1,T_0,q)+y(T_0,q)=-4+0=-4,  
\end{equation} 
\begin{equation} 
y(T_1,T_{-1},q)=-8. 
\end{equation}

\begin{figure}[th] 
 \begin{center} 
 \includegraphics[scale=0.75]{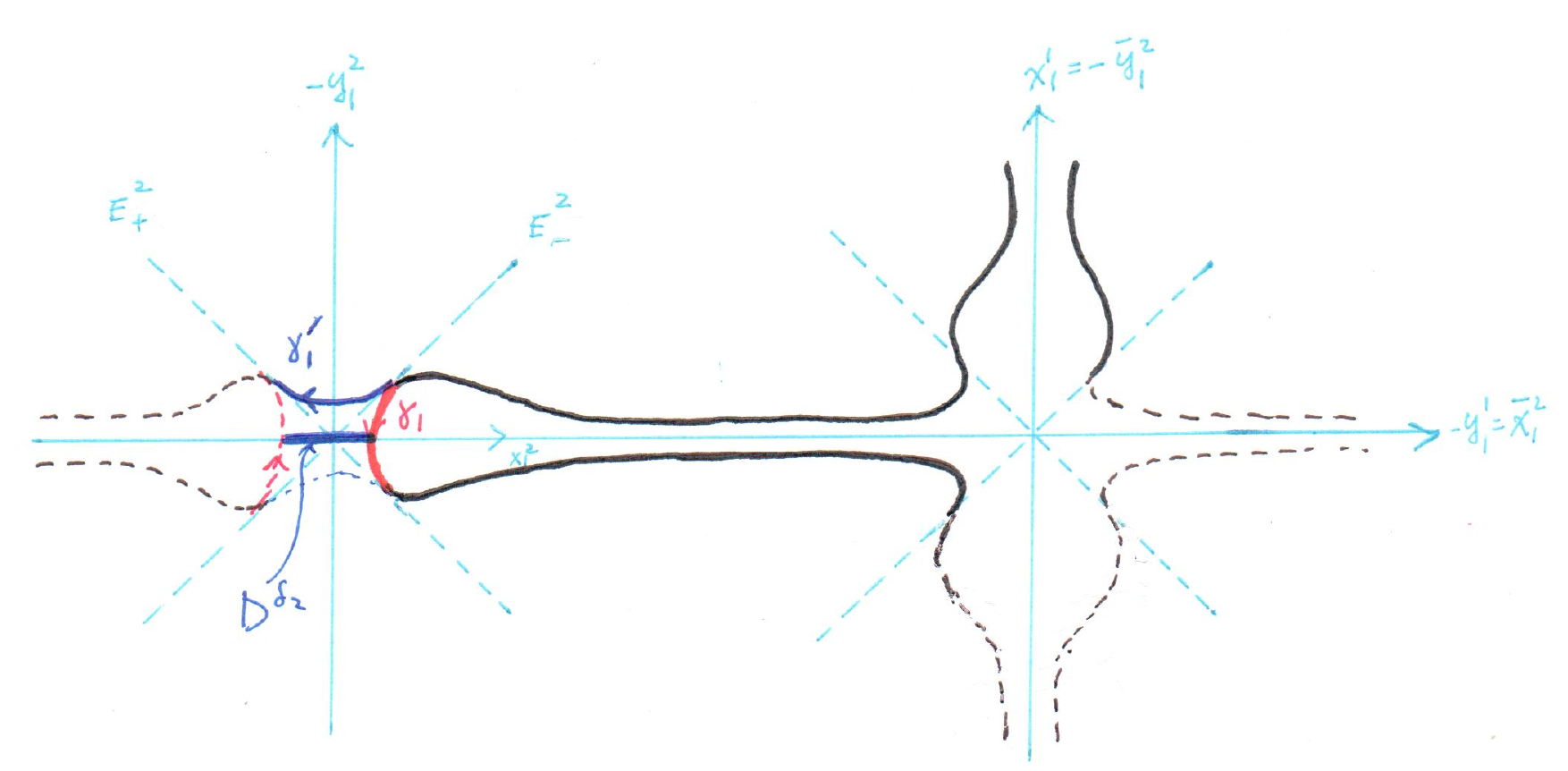} 
\caption{From $T_1$ to $T_2$. \label{T12}} 
\end{center} 
\end{figure}

Now starting from $k=1$ repeat the above procedure 
successively for $k=1,...,n-1$ to get a sequence of monotone 
Lagrangian tori $T_k$ by a sequence of $la$-disk surgeries as 
follows: 
\begin{enumerate} 
\item Hamiltonian isotop $U'_{k-1}$ rel $\pa U'_{k-1}$ of $T_k$ 
in a $\tilde{G}$-invariant fashion so that after the isotopy 
the interior of $U'_{k-1}$ contains a crossing domain 
$U_k:=Orb_{\tilde{G}}(\gamma_k)$ with $(\Delta\varphi)_{\gamma_k}=\pi$. 
\item Apply $la$-disk surgery to $T_k$ to get 
\[ 
T_{k+1}:=(T_k\setminus U_k)\cup U'_k 
\] 
where 
\[ 
 U'_k=Orb_{\tilde{G}}(\gamma'_k), \quad \gamma'_k=M(\gamma_k), 
\] 
and $(\Delta\varphi)_{\gamma'_k}=-\pi$. 
\end{enumerate} 
Note that all of $T_k$ are smoothly isotopic and  contain the domain $U_\sigma$. Then 
\begin{equation} 
y(T_{k+1},T_k,q)  =\frac{1}{2\pi}(4(\Delta\varphi)_{\gamma'_k}-4(\Delta\varphi)_{\gamma_k}) =-4, 
\end{equation} 
hence 
\begin{equation} 
y(T_k,q)=-4k, \quad k=-1,0,1,...,n, 
\end{equation} 
and 
\begin{equation} 
y(T_k,T_j,q)=4(j-k), \quad -1\leq k,j\leq n. 
\end{equation} 
So $T_k$, $k=-1,0,1,...,n$, are pairwise non-Hamiltonian isotopic in $W_n$.


\begin{thebibliography}{99}



\bibitem{AF} P. Albers, U. Frauenfelder, {\em A nondisplaceable Lagrangian torus in $T^*S^2$},  Comm. Pure Appl. Math.  61  (2008),  no. 8, 1046--1051.




\bibitem{Aur1}
 D. Auroux, {\em Mirror symmetry and T-duality in the complement of an anticanonical divisor}. J. Gškova Geom. Topol. GGT 1 (2007), 51--91.



\bibitem{Ba1} L. Bates, {\em 
Monodromy in the champagne bottle}, 
Z. Angew. Math. Phys. 42 (1991), no. 6, 837--847. 





\bibitem{C} 
Y. V. Chekanov, {\em Lagrangian tori in a symplectic vector space and global symplectomorphisms},   Math. Z.  223  (1996),  no. 4, 547--559. 

\bibitem{Du1} J. J. Duistermaat, {\em On global action-angle coordinates},  Comm. Pure Appl. Math.  33  (1980), no. 6, 687--706. 

\bibitem{EH1} 
I. Ekeland, H. Hofer, {\em Symplectic topology and Hamiltonian dynamics}. Math. Z. 200 (1989), no. 3, 355--378. 

\bibitem{EH2} 
I. Ekeland, H. Hofer, {\em  Symplectic topology and Hamiltonian dynamics. II}. Math. Z. 203 (1990), no. 4, 553--567.


\bibitem{E1} Y. Eliashberg, {\em Topological characterization of Stein manifolds of dimension $>2$},  Internat. J. Math.  1  (1990),  no. 1, 29--46.


\bibitem{EP3} Y. Eliashberg, L. Polterovich, {\em  
The problem of Lagrangian knots in four-manifolds}. 
     Geometric topology (Athens, GA, 1993), 313--327, 
AMS/IP Stud. Adv. Math., 2.1, Amer. Math. Soc., Providence, RI, 1997.

%
%
%

%
%
%






\bibitem{H1} R. Hind, {\em Lagrangian spheres in $S^1\times S^2$}. Geom. Funct. Anal. 14 (2004), no. 2, 303--318.

\bibitem{H2} R. Hind, {\em Lagrangian unknottedness in Stein surfaces}. Asian J. Math. 16 (2012), no. 1, 1--36. . 





\bibitem{KS} M. Khovanov and P. Seidel, {\em Quivers, Floer 
cohomology, and braid group actions},  J. Amer. Math. Soc. 15 (2002), no. 1, 203Ð271. 


\bibitem{Kna} A. Knapp, {\em Lagrangian tori in closed 4-manifolds}, 
J. Topol. 3 (2010), no. 2, 333Ð342. 






%


\bibitem{MS} D. McDuff and D. Salamon, {\em Introduction 
to Symplectic Topology}, 2nd ed.,  Oxford Mathematical Monographs,  
1998. 



\bibitem{O1} Y.-G. Oh, {\em Floer cohomology of Lagrangian 
intersections and pseudo-holomorphic disks, I}, 
Comm. Pure Appl. Math. 46 (1993), no. 7, 949Ð993. 


\bibitem{O2} Y.-G. Oh, {\em Addendum to "Floer cohomology 
of Lagrangian intersections and pseudo-holomorphic disks, I"}, 
 Comm. Pure Appl. Math. 48 (1995), no. 11, 1299Ð1302. 






\bibitem{P2} L. Polterovich, {\em The surgery of Lagrange submanifolds},  Geom. Funct. Anal.  1  (1991),  no. 2, 198--210.


\bibitem{Poz} 
M. Po\'{z}niak, {\em Floer homology, Novikov rings and clean intersections}. Northern California Symplectic Geometry Seminar, 119--181, Amer. Math. Soc. Transl. Ser. 2, 196, Amer. Math. Soc., Providence, RI, 1999.




\bibitem{Se} P. Seidel, {\em Lagrangian two-spheres can be 
symplectically knotted}, J. Diff. Geom. 52 (1999) no. 1, 
145--171. 


 
 %
 

\bibitem{W} 
 A. Weinstein, {\em Lectures on symplectic manifolds}. Corrected reprint. CBMS Regional Conference Series in Mathematics, 29. American Mathematical Society, Providence, R.I., 1979.





\bibitem{V1}
C. Viterbo, {\em Capacit\'{e}s symplectiques et applications (d'apr\'{e}s Ekeland-Hofer, Gromov)}. (French) [Symplectic capacities and applications (after Ekeland-Hofer, Gromov)] S\'{e}minaire Bourbaki, Vol. 1988/89. AstŽrisque No. 177-178 (1989), Exp. No. 714, 345--362. 



\bibitem{V2}
C. Viterbo, {\em Plongements lagrangiens et capacit\'{e}s symplectiques de tores dans $\R^{2n}$}. (French) [The symplectic capacity of Lagrangian tori in R2n] C. R. Acad. Sci. Paris S\'{e}r. I Math. 311 (1990), no. 8, 487Ð490.



\bibitem{Y5} M.-L. Yau, {\em Monodromy and isotopy of monotone Lagrangian tori}, Math. Res. Lett. Vol. 16, Issue 3 (2009), 531--541. 


\vspace{.1in } 
Department of Mathematics, National Central University,  Chung-Li, Taiwan. 

{\em Email address}: yau@math.ncu.edu.tw



\end{thebibliography}
\end{document}